\documentclass[10pt]{article}
\usepackage{amsfonts,color}
\usepackage{amssymb}
 \usepackage{footnote}
\usepackage{mathtools}
\usepackage{amsthm,wasysym}
\usepackage[english]{babel}
\usepackage{bbm}
\usepackage{colonequals}

\usepackage{graphicx}
\usepackage{amsfonts}
\usepackage{hyperref}
\usepackage{subcaption}
\usepackage{amsmath}
\usepackage{nicefrac}
\usepackage{gensymb}
\usepackage{enumerate}
\usepackage[nameinlink,capitalize]{cleveref}
\usepackage{cite}
\usepackage[nottoc,numbib]{tocbibind}
\usepackage{pgf,tikz,pgfplots}
\pgfplotsset{compat=1.14}
\usepackage{mathrsfs}
\usetikzlibrary{arrows}
\usepackage{overpic}

\usepackage{chngcntr}

\counterwithin*{equation}{section}

\DeclareMathOperator{\at}{\bigg\vert}
\newcommand{\vb}[1]{\mathbf{#1}}
\newcommand{\bm}[1]{\boldsymbol{#1}}

\DeclareMathOperator{\di}{\mathrm{div}}
\DeclareMathOperator{\Di}{\mathrm{Div}}
\DeclareMathOperator{\sym}{\mathrm{sym}}
\DeclareMathOperator{\skw}{\mathrm{skew}}
\DeclareMathOperator{\Le}{\mathit{L}^2}
\DeclareMathOperator{\Lez}{\mathit{L}^2_0}
\DeclareMathOperator{\Hone}{\mathit{H}^1}
\DeclareMathOperator{\Honez}{\mathit{H}_0^1}
\DeclareMathOperator{\Sym}{\mathrm{Sym}}
\DeclareMathOperator{\so}{\mathfrak{so}}

\DeclareMathOperator{\X}{\mathit{X}}

\DeclareMathOperator{\Po}{\mathit{P}}
\DeclareMathOperator{\C}{\mathit{C}}

\DeclareMathOperator{\Ned}{\mathcal{N}}
\DeclareMathOperator{\RT}{\mathcal{RT}}
\DeclareMathOperator{\BDM}{\mathcal{BDM}}
\DeclareMathOperator{\range}{\mathrm{range}}

\newcommand{\Hd}[1]{\mathit{H}(\mathrm{div}{#1})}
\newcommand{\HD}[1]{\mathit{H}(\mathrm{Div}{#1})}
\newcommand{\HDz}[1]{\mathit{H}_0(\mathrm{Div}{#1})}
\newcommand{\Hdz}[1]{\mathit{H}_0(\mathrm{div}{#1})}

\newcommand{\HDzd}[1]{\overset{\circ}{\mathit{H}}(\mathrm{Div}{#1})}
\newcommand{\Hdzdz}[1]{\overset{\circ}{\mathit{H}_0}(\mathrm{div}{#1})}
\newcommand{\HDzdz}[1]{\overset{\circ}{\mathit{H}_0}(\mathrm{Div}{#1})}
\newcommand{\Hcz}[1]{\mathit{H}_0(\mathrm{curl}{#1})}
\newcommand{\Hc}[1]{\mathit{H}(\mathrm{curl}{#1})}
\newcommand{\HC}[1]{\mathit{H}(\mathrm{Curl}{#1})}
\newcommand{\HCz}[1]{\mathit{H}_0(\mathrm{Curl}{#1})}
\DeclareMathOperator{\spa}{\mathrm{span}}
\DeclareMathOperator{\D}{\mathrm{D}\hspace{-0.1em}}
\DeclareMathOperator{\du}{\D \vb{u}}
\DeclareMathOperator{\curl}{\mathrm{curl}}
\DeclareMathOperator{\cof}{\mathrm{Cof}}
\DeclareMathOperator{\Curl}{\mathrm{Curl}}
\DeclareMathOperator{\Ce}{\mathbb{C}_{\mathrm{e}}}
\DeclareMathOperator{\Cc}{\mathbb{C}_{\mathrm{c}}}
\DeclareMathOperator{\Cm}{\mathbb{C}_{\mathrm{micro}}}
\DeclareMathOperator{\muma}{\mu_{\mathrm{macro}}}
\DeclareMathOperator{\mue}{\mu_{\mathrm{e}}}
\DeclareMathOperator{\mumi}{\mu_{\mathrm{micro}}}
\DeclareMathOperator{\muc}{\mu_{\mathrm{c}}}
\newcommand{\Lc}{L_\mathrm{c}}
\DeclareMathOperator{\one}{\bm{\mathbbm{1}}}
\DeclareMathOperator{\Sy}{\mathbb{S}}
\DeclareMathOperator{\A}{\mathbb{A}}
\newcommand{\RN}[1]{%
	\textup{\uppercase\expandafter{\romannumeral#1}}%
}

\newcommand{\dd}{\mathrm{d}}

\DeclareMathOperator{\Pm}{\bm{P}}
\DeclareMathOperator{\Ds}{\bm{D}}

\allowdisplaybreaks[1]

\makeindex
                                                     
\newtheorem{theorem}{Theorem}
\newtheorem{lemma}{Lemma}
\newtheorem{corollary}{Corollary}
\newtheorem{remark}{Remark}
\usepackage{chngcntr}
\counterwithin{theorem}{section}
\counterwithin{lemma}{section}
\counterwithin{corollary}{section}
\counterwithin{remark}{section}

   \makeatletter
\let\@fnsymbol\@arabic
\makeatother

\crefname{Problem}{Problem.}{Problem.}
\title{Primal and mixed finite element formulations for the relaxed micromorphic model}
\author{\normalsize{Adam Sky}\thanks{Corresponding author: Adam Sky, Institute of Structural Mechanics, Statics and Dynamics, Technische Universit\"at Dortmund, August-Schmidt-Str. 8, 44227 Dortmund, Germany, email: adam.sky@tu-dortmund.de}
	, \quad
	\normalsize{Michael Neunteufel}\thanks{Michael Neunteufel, Institute of Analysis and Scientific Computing, Technische Universit\"at Wien, Wiedner Hauptstr. 8-10 , 1040 Wien, Austria, email: michael.neunteufel@tuwien.ac.at}
	, \quad
	\normalsize{Ingo Muench}\thanks{Ingo Muench, Institute of Structural Mechanics, Statics and Dynamics, Technische Universit\"at Dortmund, August-Schmidt-Str. 8, 44227 Dortmund, Germany, email: ingo.muench@tu-dortmund.de}
	, \quad
	\normalsize{Joachim Sch\"oberl}\thanks{Joachim Schöberl, Institute of Analysis and Scientific Computing, Technische Universit\"at Wien, Wiedner Hauptstr. 8-10 , 1040 Wien, Austria, email: joachim.schoeberl@tuwien.ac.at}
	\\
	and \quad
	\normalsize{Patrizio Neff}\thanks{Patrizio Neff,  \ \ Chair for Nonlinear 
		Analysis and Modelling, Faculty of Mathematics, Universit\"{a}t Duisburg-Essen,
		Thea-Leymann Str. 9, 45127 Essen, Germany, email: patrizio.neff@uni-due.de}
}

\begin{document}

\maketitle

\begin{abstract}
The classical Cauchy continuum theory is suitable to model highly homogeneous materials. However, many materials, such as porous media or metamaterials, exhibit a pronounced microstructure. As a result, the classical continuum theory cannot capture their mechanical behaviour without fully resolving the underlying microstructure. In terms of finite element computations, this can be done by modelling the entire body, including every interior cell. The relaxed micromorphic continuum offers an alternative method by instead enriching the kinematics of the mathematical model. The theory introduces a microdistortion field, encompassing nine extra degrees of freedom for each material point. The corresponding elastic energy functional contains the gradient of the displacement field, the microdistortion field and its Curl (the micro-dislocation). Therefore, the natural spaces of the fields are $[\Hone]^3$ for the displacement and $[\Hc{}]^3$ for the microdistortion, leading to unusual finite element formulations.
In this work we describe the construction of appropriate finite elements using N\'ed\'elec and Raviart-Thomas subspaces, encompassing solutions to the orientation problem and the discrete consistent coupling condition. Further, we explore the numerical behaviour of the relaxed micromorphic model for both a primal and a mixed formulation. 
The focus of our benchmarks lies in the influence of the characteristic length $\Lc$ and the correlation to the classical Cauchy continuum theory.
\\
\vspace*{0.25cm}

\end{abstract}

\tableofcontents

\section{Introduction}
A common problem in the computation of materials with a pronounced micro-structure is the internal complexity of the geometry. In order to fully capture the kinematics one might resolve the underlying micro-structure. This can be done either with multi-scale finite element methods \cite{Eidel2018} or by modelling the finite element mesh to fully incorporate the microstructure. In both cases, the computational cost increases, leading to longer computation times and decreasing applicability of such approaches. Alternatively, one may enrich the mathematical model in order to account for the increase in the kinematical complexity. This approach gives rise to generalized continuum theories such as higher order gradient methods \cite{Kirchner2006,Mindlin1964,Neff2009} or micromorphic continua \cite{Neff2007,Steigmann2012,Xiaozhe}. Micromorphic continuum theories extend the kinematics of the material point with additional degrees of freedom, the choice of which defines the specific theory. Common examples are micropolar Cosserat \cite{Jeong2010, Munch2011}, microstretch \cite{Romeo2020}, and microstrain models \cite{Forest2006, Hutter2016}. The latter represent sub-types of the micromorphic continuum. In its most general case, micromorphic continua assume an affine deformable microbody for each material point. As such, this deformation, called here the microdistortion $\Pm$, is fully captured by three-by-three matrices and introduces nine extra degrees of freedom. The full micromorphic continuum, as introduced by Eringen and Mindlin \cite{Mindlin1968,Eringen1999}, incorporates the gradient of the microdistortion $\D \Pm$ into the free energy functional. The resulting hyperstress term is a third order tensor. As such, it is unclear how this term is to be interpreted or applied. 
Typically, micromorphic continuum theories introduce a characteristic length scale parameter $\Lc$, which abstractly relates the dimension of the micro-body to that of the macro-body.
In the full micromorphic model, when the characteristic length $\Lc$ becomes very large, the microdistortion $\Pm$ must become constant in order for the theory to generate finite energies, which may lead to boundary layer problems \cite{Sky2021}.

The relaxed micromorphic continuum theory \cite{Neff2014} takes a different approach by instead incorporating only the Curl of the microdistortion $\Curl \Pm$ into the free energy function. The latter term, known as the micro-dislocation, remains a second order tensor and as such, induces a matrix-valued right-hand-side, known as the micro-moment $\bm{M}$. Further, large characteristic lengths $\Lc$ maintain finite energies \cite{Sky2021, Neff2010Finite}.
The theory aims to capture the mechanical behaviour of both highly homogeneous materials and materials with a pronounced microstructure by governing the relation to the classical Cauchy continuum using the characteristic length $\Lc$ \cite{Barbagallo2017} and shows great promise with respect to applications utilizing metamaterials, such as band-gap materials \cite{Madeo2016,Madeo2018,Agostino2020Band,BARBAGALLO2019148} and shielding against elastic waves \cite{Rizzi_shield, Rizzi2021Wave}. Furthermore, analytical solutions have been derived for bending \cite{Rizzi_bending}, torsion \cite{Rizzi_torsion}, shear \cite{Rizzi_shear}, and extension \cite{Rizzi_extension} kinematics. 
The inclusion of the Curl of microdistortion in the free energy function implies the existence of unique solutions \cite{GNMPR15,Neff2015} in the space $\X = [\Hone]^3 \times [\Hc{}]^3$, as shown in \cref{ssec:primal}.  While the construction of finite elements for the space $\Hone$ is well-known in the field of mechanics, $\Hc{}$-finite elements are commonly used for the Maxwell equations, for example in magnetostatics \cite{Joachim2005}. 
For the construction of finite elements for $\Hc{}$ one may use N\'ed\'elec subspaces \cite{Nedelec1980, Ned2}. The formulation of higher order elements is detailed in \cite{Zaglmayr2006}. 
For large characteristic length values $\Lc$ the computation with the primal formulation ($\mathit{X}^h \subset \mathit{X} = [\Hone]^3 \times [\Hc{}]^3$) becomes unstable \cite{Sky2021}. However, it can be re-stabilised by using a mixed formulation. The latter requires the employment of Raviart-Thomas- \cite{Raviart} or Brezzi-Douglas-Marini elements \cite{BDM85} and fully discontinuous finite elements as per the de Rham diagram \cite{Demkowicz2000, Arnold2021, DEMKOWICZ2005267}. 

In this work we demonstrate the existence and uniqueness of the primal formulation using the Lax-Milgram theorem, \cref{ssec:primal}. Further, we introduce a mixed formulation, which is stable for large characteristic length $\Lc$ values, \cref{ssec:mixed}. 
In \cref{ch:3} we derive the corresponding convergence rates for the discrete spaces. 
The construction of lower order finite elements for both formulations is explained in \cref{ch:4}, with focus on a solution to the orientation problem and application of the discrete consistent coupling condition.
\cref{ch:7} is devoted to numerical benchmarks of the finite element formulations but also features the convergence characteristics of higher order elements using NETGEN/NGSolve \cite{Sch1997,Sch2014}. 
Finally, we present our conclusions and outlook in \cref{ch:8}.

\section{The linear relaxed micromorphic continuum}
The linear relaxed micromorphic continuum \cite{Neff2019,Neff2014,Neff2015} is described by its free energy functional, incorporating the gradient of the displacement field, the microdistortion and its Curl
\begin{align}
    I(\vb{u}, \bm{P}) = \dfrac{1}{2} \int_{\Omega} &\langle \Ce \sym(\D \vb{u} - \bm{P}) , \, \sym(\D \vb{u} - \bm{P}) \rangle
		+  \langle \Cm \sym\bm{P} , \, \sym\bm{P} \rangle \label{eq:1} \\ 
		& + \langle \Cc \skw(\D\vb{u} - \bm{P}) , \, \skw (\D \vb{u} - \bm{P}) \rangle
		+ \muma \Lc^2 \, \| \text{Curl}\bm{P} \|^2 \, \dd X  - \int_{\Omega} \langle \vb{u} , \, \vb{f} \rangle  + \langle \bm{P} , \, \bm{M} \rangle \, \dd X \, , 
		\notag
\end{align}
with $\vb{u} : \Omega \subset \mathbb{R}^3 \to \mathbb{R}^3$ and $\bm{P}: \Omega \subset \mathbb{R}^3 \to \mathbb{R}^{3 \times 3}$ representing the displacement and the non-symmetric microdistortion, respectively. Here, $\Ce$ and $\Cm$ are standard fourth order elasticity tensors and $\Cc$ is a positive semi-definite coupling tensor for (infinitesimal) rotations. The macroscopic shear modulus is denoted by $\muma$ and the parameter $\Lc \geq 0$ represents the characteristic length scale motivated by the geometry of the microstructure. The body forces and micro-moments are denoted with $\vb{f}$ and $\bm{M}$, respectively.
The differential operators are defined as
\begin{align}
    \D \vb{u} = \begin{bmatrix}
		u_{1,1} & u_{1,2} & u_{1,3} \\
		u_{2,1} & u_{2,2} & u_{2,3} \\
		u_{3,1} & u_{3,2} & u_{3,3} 
	\end{bmatrix}
	\, , \quad
	\text{Curl}\bm{P} = \begin{bmatrix}
		(\text{curl} \begin{bmatrix}
			P_{11} & P_{12} & P_{13} 
		\end{bmatrix})^T \\[1ex]
		(\text{curl} \begin{bmatrix}
			P_{21} & P_{22} & P_{23} 
		\end{bmatrix})^T \\[1ex]
		(\text{curl} \begin{bmatrix}
			P_{31} & P_{32} & P_{33} 
		\end{bmatrix})^T
	\end{bmatrix} \, , \quad 
	\text{curl}\vb{v} = \nabla \times \vb{v} \, .
\end{align}
For isotropic materials the material tensors have the following structure
\begin{align}
	\Ce &= 2 \mu_{\textrm{e}} \Sy  + \lambda_{\textrm{e}} \one \otimes \one \, , & \Cm &= 2 \mu_{\textrm{micro}} \Sy + \lambda_{\textrm{micro}} \one \otimes \one  \, , &
	\Cc &= 2 \mu_c \A \, ,
\end{align}
where $\Sy:\mathbb{R}^{3 \times 3} \mapsto \Sym(3)$ and $\A:\mathbb{R}^{3 \times 3} \mapsto \so(3)$ are the fourth order symmetry and anti-symmetry tensors, respectively.
Taking variations with respect to the displacement $\vb{u}$
\begin{align}
    \delta_u I = \int_\Omega \langle \Ce \sym \D \delta  \vb{u} , \, \sym(\D \vb{u} - \Pm) \rangle + \langle \Cc \skw \D \delta  \vb{u} , \, \skw(\D \vb{u} - \Pm) \rangle - \langle \delta \vb{u} ,\, \vb{f} \rangle \, \dd X \, , 
    \label{eq:varu}
\end{align}
and the microdistortion $\bm{P}$
\begin{align}
    \delta_P I = \int_\Omega & - \langle \Ce \sym \delta \Pm , \, \sym(\D \vb{u} - \Pm) \rangle - \langle \Cc \skw \delta \Pm , \, \skw(\D \vb{u} - \Pm) \rangle + \langle \Cm \sym \delta \Pm , \, \sym\Pm \rangle \notag \\ 
    &+ \muma \Lc^2 \, \langle \Curl \delta \Pm , \, \Curl \Pm \rangle -\langle \delta \Pm, \, \bm{M} \rangle \, \dd X \, , 
    \label{eq:varp}
\end{align}
yields the symmetric bilinear form
\begin{align}
\label{eq:bf_primal}
	a(\{\delta \vb{u}, \, \delta \bm{P}\}, \{\vb{u}, \, \bm{P}\}) =   \int_{\Omega} & \langle \Ce \sym(\D \delta \vb{u} - \delta \bm{P}) , \, \sym(\D \vb{u} - \bm{P}) \rangle
	+  \langle \Cm \sym \delta \bm{P} , \, \sym\bm{P} \rangle \notag \\ 
	& + \langle \Cc \skw(\D \delta \vb{u} - \delta \bm{P}) , \, \skw (\D \vb{u} - \bm{P}) \rangle
	+ \muma \, \Lc^2 \, \langle \text{Curl} \delta \bm{P} , \,  \Curl \Pm \rangle \, \dd X \, ,
\end{align}
and the linear form for the load
\begin{align}
\label{eq:rhs_primal}
	l(\{\delta \vb{u}, \, \delta \bm{P}\}) = 
	\int_{\Omega} \langle \delta \vb{u} , \, \vb{f} \rangle  + \langle \delta \bm{P} , \, \bm{M} \rangle \, \dd X  \, .
\end{align}
The corresponding strong form follows from partial integration, see \cref{ap:A}
\begin{subequations}
\begin{align}
	-\Di[\Ce \sym (\D \vb{u} - \bm{P}) + \Cc \skw (\D \vb{u} - \bm{P})] &= \vb{f} && \text{in} \quad \Omega \, , \label{eq:strong_form_du}\\
	-\Ce  \sym (\D \vb{u} - \Pm) - \Cc  \skw(\D \vb{u} - \Pm) + \Cm \sym \Pm + \muma \, \Lc ^ 2  \Curl(\Curl\Pm) &= \bm{M} && \text{in} \quad \Omega \, ,\label{eq:strong_form_dP} \\
	\vb{u} &= \widetilde{\vb{u}} && \text{on} \quad \Gamma_D^u \, , \\
	\Pm \times \, \bm{\nu} &= \widetilde{\Pm} \times \bm{\nu} && \text{on} \quad \Gamma_D^P \, , \\
	[\Ce \sym (\D \vb{u}- \bm{P}) + \Cc \skw (\D \vb{u} - \bm{P})] \, \bm{\nu} &= 0 && \text{on} \quad \Gamma_N^u \, ,\\
	\Curl \Pm \times \, \bm{\nu}  &= 0  && \text{on} \quad \Gamma_N^P \, ,
\end{align}
\end{subequations}
where $\bm{\nu}$ denotes the outer unit normal vector, such that $\Pm \times \, \bm{\nu}$ is the projection to the tangent surface on the boundary. The terms $\widetilde{\vb{u}}$ and $\widetilde{\Pm}$ are the prescribed displacement and microdistortion fields on $\Gamma_D^u$ and $\Gamma_D^P$, respectively, see \cref{fig:domain}.
\begin{figure}
    \centering
    \hspace{2em}
    \definecolor{uuuuuu}{rgb}{0.26666666666666666,0.26666666666666666,0.26666666666666666}
    \definecolor{ududff}{rgb}{0.30196078431372547,0.30196078431372547,1}
    \definecolor{qqqqff}{rgb}{0,0,1}
    \definecolor{xfqqff}{rgb}{0.4980392156862745,0,1}
    \definecolor{qqzzcc}{rgb}{0,0.6,0.8}
    \begin{overpic}[width=0.8\linewidth]{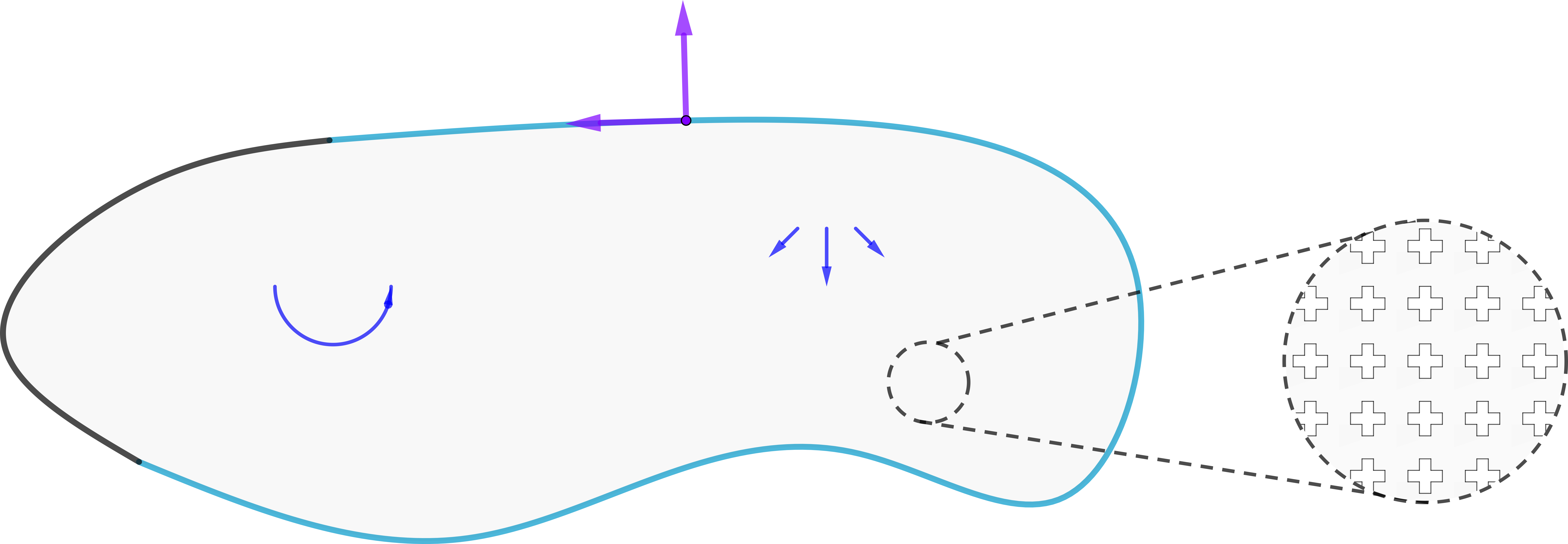}
    \put (36,15) {$\Omega$}
    \put (19.9,15.5) {\color{blue}$\bm{M}$}
    \put (52.4,21.5) {\color{blue}$\vb{f}$}
    \put (-14.5,5.5) {$\Gamma_D = \Gamma_D^u = \Gamma_D^P$}
    \put (-14.5,29.5) {$\Gamma_D^u: \; \vb{u} = \widetilde{\vb{u}}$}
    \put (-14.5,25.5) {$\Gamma_D^P: \; \Pm \times \, \bm{\nu} = \D \widetilde{\vb{u}} \times \bm{\nu}$}
    \put (41.5,30.5) {\color{xfqqff}$\bm{\nu}$}
    \put (38.5,24.9) {\color{xfqqff}$\bm{\tau}$}
    \put (66,26) {\color{qqzzcc}$\Gamma_N = \Gamma_N^u = \Gamma_N^P$}
\end{overpic}
\caption{Depiction of the reference domain $\Omega$ with an underlying micro-structure, the Dirichlet boundary $\Gamma_D$ according to the consistent coupling condition and the Neumann boundaries $\Gamma_N^u$ and $\Gamma_N^P$ for the displacement and the microdistortion, respectively. The vector $\bm{\nu}$ is the unit normal vector on the boundary. The unit tangent vectors are illustrated with $\bm{\tau}$. The body force is given by $\vb{f}$ and the micro-moment by $\bm{M}$.}
    \label{fig:domain}
\end{figure}
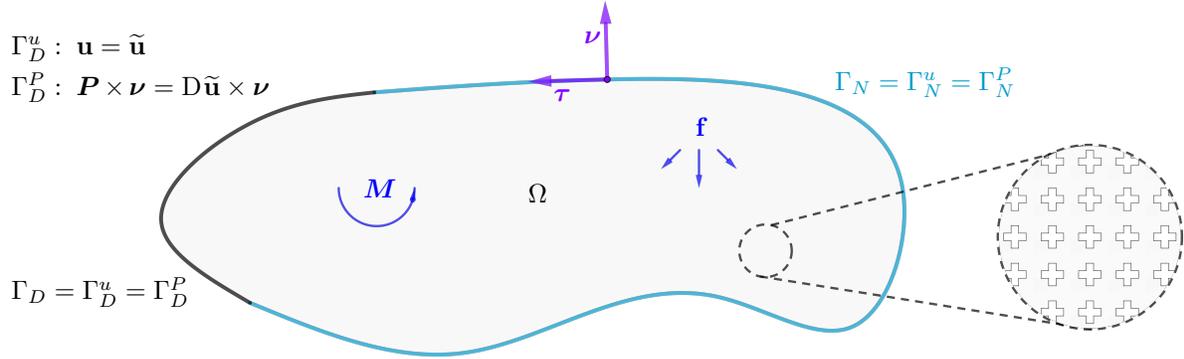

From a physical point of view, it is impossible to control the micro-movements of the material point on the Dirichlet boundary without also controlling the displacement. Consequently, the relaxed micromorphic theory introduces the so called \textbf{consistent coupling condition} \cite{dagostino2021consistent}
\begin{align}
    &\bm{P} \times \bm{\nu} = \D \widetilde{\vb{u}} \times \bm{\nu} \quad \text{on} \quad \Gamma_D^P \, ,
\end{align}
where the prescribed displacement $\widetilde{\vb{u}}$ on the boundary automatically prescribes the tangential component of the microdistortion $\Pm$ on the same boundary, effectively 
inducing the definition $\Gamma_D = \Gamma_D^P = \Gamma_D^u$. 

\section{Solvability and limit problems} \label{ch:2}

\subsection{Continuous case} 
In this section existence and uniqueness of the weak formulation of the relaxed micromorphic continuum for both a primal and a mixed method is discussed. 
In order to simplify the proof, homogeneous boundary conditions are assumed on the entire boundary, $\Gamma^u_D = \partial \Omega$. The proof can be easily adjusted for mixed or inhomogeneous boundary conditions as long as the Dirichlet boundary does not vanish $|\Gamma^u_D| > 0$.\\

For the following formulations, we define corresponding Hilbert spaces and their particular norms
\begin{subequations}
\begin{align}
	\Le(\Omega) &=  \{ u: \Omega \mapsto \mathbb{R} \, | \,  \|u\|_{\Le}^2 < \infty \} \, , && \|u \|_{\Le}^2 = \int_{\Omega} \|u\|^2 \, \dd X \,,\\
	\Lez(\Omega) &=  \left \{ u \in \Le(\Omega) \, | \, \int_{\Omega} u \, \dd X = 0 \right \} \, , \\[1ex]
	\Hone(\Omega) &= \{ u \in \Le(\Omega) \, | \,  \nabla u \in [\Le(\Omega)]^3 \} \, , && \|u \|_{\Hone}^2 = \|u\|_{\Le}^2 + \| \nabla u \|_{\Le}^2 \,,\\[2ex]
	\Honez(\Omega) &= \{ u \in \Hone(\Omega) \, | \, u=0 \text{ on }\partial \Omega \} \, ,\\[2ex]
	\Hc{,\, \Omega} &= \{ \vb{p} \in [\Le(\Omega)]^3 \, | \curl\vb{p} \in [\Le(\Omega)]^3   \} \, , && \| \vb{p} \|_{\Hc{}}^2 = \| \vb{p} \|_{\Le}^2 + \| \curl\vb{p} \|_{\Le}^2 \, ,\\[2ex]
	\Hcz{, \, \Omega} &= \{ \vb{p} \in \Hc{, \, \Omega} \, | \,  \vb{p} \times \bm{\nu}  = \vb{0}  \text{ on }\partial \Omega \} \, , \\[2ex]
	\Hd{,\, \Omega} &= \{ \vb{p} \in [\Le(\Omega)]^3 \, | \di\vb{p} \in \Le(\Omega)   \} \, , && \| \vb{p} \|_{\Hd{}}^2 = \| \vb{p} \|_{\Le}^2 + \| \di\vb{p} \|_{\Le}^2 \, ,\\[2ex]
	\Hdz{, \, \Omega} &= \{ \vb{p} \in \Hd{, \, \Omega} \, | \, \langle \vb{p} , \, \bm{\nu} \rangle = 0  \text{ on }\partial \Omega \} \, ,
\end{align}
\end{subequations}
from which we derive the corresponding spaces for our higher dimensional problem
\begin{align}
    &\HC{,\Omega} = [\Hc{,\Omega}]^3 \, , && \HD{,\Omega} = [\Hd{,\Omega}]^3 \, ,
\end{align}
where both spaces are to be understood as row-wise matrices of the vectorial spaces.

\begin{remark}
    In the following sections we assume a contractible domain $\Omega$ for all proofs.
\end{remark}

\subsubsection{Primal form}\label{ssec:primal}
The proof of existence and uniqueness has already been given in \cite{Neff2014} and subsequently generalized to the dynamic setting in \cite{GNMPR15}. For the reader's convenience and for later use we present a proof based on the Lax--Milgram theorem and Korn's inequalities for incompatible fields \cite{LewintanInc,LewintanInc2,Lewintan2021K,Starke2,NEFF20151267,Starke1,Neff2012}, similar to \cite{Neff2015}, where only the case $\muc=0$ together with $\Gamma_D^P=\partial\Omega$ has been considered. Therefore, we define the product space $\X = [\Honez(\Omega)]^3\times \HC{,\Omega}$ with its standard product norm $\|\{\vb{u},\bm{P}\}\|^2_{\X}=\|\vb{u}\|^2_{\Hone}+\|\bm{P}\|^2_{\HC{}}$ and examine the problem: Find $\{\vb{u},\bm{P}\}\in \X$ such that 
\begin{align}
\label{eq:primal_problem}
    a(\{\vb{u},\bm{P}\},\{\delta\vb{u},\delta\bm{P}\}) = l(\{\delta\vb{u},\delta\bm{P}\}) \quad  \forall\, \{\delta\vb{u},\delta\bm{P}\}\in\X,
\end{align}
where $a(\cdot,\cdot)$ and $l(\cdot)$ are defined as in \eqref{eq:bf_primal} and \eqref{eq:rhs_primal}, respectively.
\begin{theorem}
	\label{thm:ex_un_primal}
	Let $\muma,\Lc>0$, $\vb{f}\in [\Le(\Omega)]^3,\bm{M}\in[\Le(\Omega)]^{3\times 3}$, and $\Ce,\Cm$ be positive definite on $\Sym(3)$. Further, assume that $\Cc$ is positive definite on $\so(3)$, or positive semi-definite and $\Gamma_D^P\neq\emptyset$. Then Problem~\ref{eq:primal_problem} is uniquely solvable and there holds the stability estimate
	\begin{align*}
	\|\{\vb{u},\bm{P}\}\|_{\X}\leq c \left(\|\vb{f}\|_{\Le}+\|\bm{M}\|_{\Le}\right),\qquad c=c(\Ce,\Cc,\Cm,\muma,\Lc).
	\end{align*}
\end{theorem}
\begin{proof}
	We show continuity and coercivity of \eqref{eq:bf_primal}. During the proof we denote with $c>0$ a generic constant which may change from line to line. Continuity follows immediately with Cauchy-Schwarz inequality
	\begin{align*}
	|a(\{\vb{u},\bm{P}\},\{\delta\vb{u},\delta\bm{P}\})| &\leq c\Big((\|\vb{u}\|_{\Hone}+\|\bm{P}\|_{\Le})(\|\delta\vb{u}\|_{\Hone}+\|\delta\bm{P}\|_{\Le}) + \|\sym\bm{P}\|_{\Le}\|\sym\delta\bm{P}\|_{\Le} + \|\Curl\bm{P}\|_{\Le}\|\Curl\delta\bm{P}\|_{\Le}\Big)\\
	&\leq c\Big((\|\vb{u}\|_{\Hone}+\|\bm{P}\|_{\HC{}})(\|\delta\vb{u}\|_{\Hone}+\|\delta\bm{P}\|_{\HC{}}) + \|\bm{P}\|_{\HC{}}\|\delta\bm{P}\|_{\HC{}}\Big)\\
	&\leq c (\|\vb{u}\|_{\Hone}+\|\bm{P}\|_{\HC{}})(\|\delta\vb{u}\|_{\Hone}+\|\delta\bm{P}\|_{\HC{}})\leq c\|\{\vb{u},\bm{P}\}\|_{X}\|\{\delta\vb{u},\delta\bm{P}\}\|_{X}.
	\end{align*}
	For the coercivity we first consider the case of positive definiteness of the tensors
	\begin{align*}
	\langle \Ce \bm{S},\bm{S}\rangle \geq c_\mathrm{e}\|\bm{S}\|^2  \quad  \forall \bm{S}\in\Sym(3) \, ,
	&&\langle \Cm \bm{S},\bm{S}\rangle \geq c_\mathrm{micro}\|\bm{S}\|^2\quad\forall \bm{S}\in\Sym(3) \, , 
	&&\langle \Cc \bm{A},\bm{A}\rangle \geq c_\mathrm{c}\|\bm{A}\|^2\quad\forall \bm{A}\in \so(3),
	\end{align*}
	to deduce with $c_1=\min\{c_\mathrm{e},c_\mathrm{c}\}$
	\begin{align*}
	a(\{\vb{u},\bm{P}\},\{\vb{u},\bm{P}\}) &\geq c_1(\|\sym(\D \vb{u} - \bm{P})\|_{\Le}^2+\|\skw(\D \vb{u} - \bm{P})\|_{\Le}^2) + c_\mathrm{micro} \|\sym\bm{P}\|_{\Le}^2+\muma\Lc^2\|\Curl \bm{P}\|_{\Le}^2.
	\end{align*}
	First, we consider the symmetric terms. Analogously to \cite{Sky2021}, with Young's inequality\footnote{Young: $a\, b\leq \frac{\varepsilon a^2}{2}+\frac{b^2}{2\varepsilon},\quad\forall\, a,b\in \mathbb{R}$, $\varepsilon>0 \, .$} and Korn's inequality\footnote{Korn: $\|\sym \D \vb{u}\|_{\Le}\geq c_K \| \D \vb{u}\|_{\Le}\quad \forall\, \vb{u}\in [\Honez(\Omega)]^3 \, .$} \cite{neff_2002} we obtain
	\begin{align}
	   c_1\|\sym(\D \vb{u} - \bm{P})\|_{\Le}^2 + c_\mathrm{micro} \|\sym\bm{P}\|_{\Le}^2 &=c_1\left(\|\sym\D\vb{u}\|^2_{\Le} - 2\langle \sym\D\vb{u},\sym\Pm\rangle_{\Le}+\|\sym\bm{P}\|^2_{\Le}\right) + c_\mathrm{micro} \|\sym\bm{P}\|_{\Le}^2 \notag\\
	   &\hspace{-0.7em} \overset{\text{Young}}{\geq} c_1(1-\varepsilon)\|\sym\D\vb{u}\|^2_{\Le} + (c_1-\frac{c_1}{\varepsilon}+c_\mathrm{micro})\|\sym \bm{P}\|_{\Le}^2 \notag \\
	   &\hspace{-0.5em}\overset{\text{Korn}}{\geq} c_Kc_1(1-\varepsilon)\|\D\vb{u}\|^2_{\Le} + (c_1-\frac{c_1}{\varepsilon}+c_\mathrm{micro})\|\sym \bm{P}\|_{\Le}^2.
	\end{align}
	We can choose $\varepsilon=\frac{1}{2}(1+\frac{c_1}{c_1+c_\mathrm{micro}})$ such that both terms are positive. Next, we estimate the skew-symmetric part
	\begin{align}
	    \|\skw(\D \vb{u} - \bm{P})\|_{\Le}^2 &=\|\skw\D\vb{u}\|^2_{\Le} - 2\langle \skw\D\vb{u},\skw\bm{P}\rangle_{\Le}+\|\skw\bm{P}\|^2_{\Le} \notag\\
	    &\hspace{-0.7em} \overset{\text{Young}}{\geq} (1-\frac{1}{\delta})\|\skw\D\vb{u}\|^2_{\Le} + (1-\delta)\|\skw \bm{P}\|_{\Le}^2 \notag\\
	    &\hspace{-0.9em}\overset{1-\frac{1}{\delta}<0}{\geq} (1-\frac{1}{\delta})\|\D\vb{u}\|^2_{\Le} + (1-\delta)\|\skw \bm{P}\|_{\Le}^2.
	\end{align}
	With $0<\delta < 1$ only the second term is positive. By combining both estimates we conclude by choosing $\delta = \frac{1}{2}(1+\frac{1}{1+c_K(1-\varepsilon)})$
	\begin{align}
	a(\{\vb{u},\bm{P}\},\{\vb{u},\bm{P}\}) &\geq c_1(c_K-c_K\varepsilon+1-\frac{1}{\delta})\|\D\vb{u}\|^2_{\Le} + (c_1-\frac{c_1}{\varepsilon}+c_m)\|\sym \bm{P}\|_{\Le}^2 + (1-\delta)\|\skw \bm{P}\|_{\Le}^2+c\|\Curl \bm{P}\|_{\Le}^2 \notag\\
	&\geq c (\|\D\vb{u}\|^2_{\Le} + \|\bm{P}\|^2_{\HC{}}).
	\end{align}
	Thus, with Poincar\`e-Friedrich's inequality\footnote{Poincar\`e-Friedrich: $\|\vb{u}\|_{\Le}\leq c_F\|\D \vb{u}\|_{\Le}\quad\forall\,\vb{u}\in [\Honez(\Omega)]^3 \, .$} we obtain the coercivity 
	\begin{align}
	    a(\{\vb{u},\bm{P}\},\{\vb{u},\bm{P}\}) \geq c(\|\D \vb{u}\|^2_{\Le} + \|\bm{P}\|^2_{\HC{}})\geq C\|\{\vb{u},\bm{P}\}\|_{\X}^2.
	\end{align}
	
	In the case of a positive semi-definite $\Cc$ (even $\Cc = 0$ is allowed: absence of rotational coupling) together with $\Gamma_D^P\neq\emptyset$ we must use the generalized Korn's inequality for incompatible fields, cf. \cite{GNMPR15,LewintanInc,LewintanInc2,Starke1,NEFF20151267,Starke1,Neff2012}
	\begin{align}
	\label{eq:control_P_hc}
	    \|\sym \bm{P}\|^2_{\Le} + \|\Curl\bm{P}\|^2_{\Le}\geq c  \|\bm{P}\|^2_{\HC{}},\qquad \forall\, \bm{P} \in \HC{,\Omega} \text{ with } \bm{P}\times\bm{\nu} = \vb{0} \text{ on }\Gamma_D^P
	\end{align}
	and estimate
	\begin{align}
	a(\{\vb{u},\bm{P}\},\{\vb{u},\bm{P}\}) &\geq c_\mathrm{e}\|\sym(\D \vb{u} - \bm{P})\|_{\Le}^2+ c_\mathrm{micro} \|\sym\bm{P}\|_{\Le}^2+\muma\Lc^2\|\Curl \bm{P}\|_{\Le}^2 \notag\\
	&\hspace{-1.8em} \overset{\text{Young, \eqref{eq:control_P_hc}}}{\geq} (1-\varepsilon)\|\sym\D\vb{u}\|^2_{\Le} + (1-\frac{1}{\varepsilon})\|\sym \bm{P}\|^2_{\Le} + \frac{c_\mathrm{micro}}{2}\|\sym\bm{P}\|_{\Le}^2 + c\|\bm{P}\|_{\HC{}}^2 \notag\\
	&\hspace{-0.5em}\overset{\text{Korn}}{\geq}c_K(1-\varepsilon)\|\D \vb{u}\|^2_{\Le}+(1-\frac{1}{\varepsilon}+\frac{c_\mathrm{micro}}{2})\|\sym\bm{P}\|_{\Le}^2 + c\|\bm{P}\|_{\HC{}}^2 \notag\\
	&\geq c(\|\D \vb{u}\|^2_{\Le}+\|\bm{P}\|_{\HC{}}^2).
	\end{align}
	
	Continuity of the right-hand side $l$ is obvious and thus we can apply the Lax--Milgram theorem finishing the proof.
\end{proof}

\begin{remark}
	The proof above fails if we consider $\bm{P}\in [\Hone(\Omega)]^{3\times 3}$ instead of $\HC{,\Omega}$. In fact, $a(\cdot,\cdot)$ is then no longer coercive as we cannot bound $\|\bm{P}\|_{\Le}^2+\|\Curl\bm{P}\|_{\Le}^2\geq c\|\bm{P}\|_{\Hone}^2$ uniformly in $\bm{P}$. In \cite{Sky2021} and \cite{SSSN21} it is demonstrated that already in two spatial dimensions a loss of optimal convergence rates is obtained when using  $\bm{P}\in [\Hone(\Omega)]^{2\times 2}$.
\end{remark}

\begin{remark}
    Note that in \cite{Reg,Reg2} it is shown that for certain sufficiently smooth data, the regularity of $\Pm$ can yet be improved to $\Pm \in [\Hone(\Omega)]^{3\times 3}$.
\end{remark}

\subsubsection{Limit of vanishing characteristic length $\Lc\to 0$} \label{sec:lczero}
In the limit $\Lc\to 0$ the bilinear form \eqref{eq:bf_primal} reduces to
\begin{align}
\label{eq:bf_primal_lc0}
	a(\{\delta \vb{u}, \, \delta \bm{P}\}, \{\vb{u}, \, \bm{P}\}) =   \int_{\Omega} & \langle \Ce \sym(\D \delta \vb{u} - \delta \bm{P}) , \, \sym(\D \vb{u} - \bm{P}) \rangle
	+  \langle \Cm \sym \delta \bm{P} , \, \sym\bm{P} \rangle \notag \\ 
	& + \langle \Cc \skw(\D\delta \vb{u} - \delta \bm{P}) , \, \skw (\D \vb{u} - \bm{P}) \rangle
	 \, \dd X \,.
\end{align}
Therefore, we lose control over the Curl of $\bm{P}$ yielding a loss of regularity for $\bm{P}$ from $\HC{,\Omega}$ to $[\Le(\Omega)]^{3\times3}$. We emphasise that the proof of Theorem~\ref{thm:ex_un_primal} can be directly applied with the adapted product space $\X= [\Honez(\Omega)]^3\times [\Le(\Omega)]^{3\times3}$ together with the requirement of positive definite $\Cc$ on the set of skew-symmetric matrices. Otherwise, control over the skew-symmetric part of $\bm{P}$ is completely lost in the limit leading to unstable results. Note that in this case no prescription of boundary conditions for $\bm{P}$ is possible.

\cref{eq:strong_form_dP} can
be reformulated as
\begin{align}
    -\Ce  \sym (\D \vb{u} - \Pm) - \Cc  \skw(\D \vb{u} - \Pm) + \Cm \sym \Pm = \bm{M} \, ,
    \label{eq:cond}
\end{align} 
and used to express $\bm{P}$ algebraically
\begin{align}
\Pm = (\Ce+\Cm)^{-1}\Ce  \sym \du + \skw\du  + \bm{M}\, .\label{eq:expr_P_limit0}
\end{align}
Setting $\bm{M} = 0$ in \cref{eq:cond} implies $\Cc\skw(\du - \Pm)= 0$ and consequently
\begin{align}
    \Ce\sym(\du - \Pm) &= \Cm \sym \Pm  \, , &\sym \Pm &= (\Ce+\Cm)^{-1}\Ce  \sym \du \, .
\end{align}
Applying the latter to \cref{eq:strong_form_du} yields
\begin{align}
    -\Di(\mathbb{C}_\mathrm{macro} \sym \du) = \vb{f} \, , && \mathbb{C}_\mathrm{macro} = \Cm (\Ce+\Cm)^{-1}\Ce \, .
    \label{eq:tocauchy}
\end{align}
The upper definition is derived in \cite{Barbagallo2017} and relates the meso- and micro-elasticity tensors to the classical macro-elasticity tensor $\mathbb{C}_\mathrm{macro}$ of the Cauchy continuum, allowing to extract the macro material constants in the isotropic case
\begin{align}
    \muma &= \dfrac{\mu_\mathrm{e} \, \mu_\mathrm{micro}}{\mu_\mathrm{e} + \mu_\mathrm{micro}} \, , & 2 \muma + 3 \, \lambda_\mathrm{macro} &= \dfrac{(2\,\mu_\mathrm{e} + 3 \, \lambda_\mathrm{e})(2\,\mu_\mathrm{micro} + 3\, \lambda_\mathrm{micro})}{(2\,\mu_\mathrm{e} + 3\, \lambda_\mathrm{e}) + (2\,\mu_\mathrm{micro} + 3\, \lambda_\mathrm{micro})} \, .
    \label{eq:lame}
\end{align}
In fact, $\mathbb{C}_\mathrm{macro}$ contains the material constants that arise from classical homogenization for large periodic structures.


From \cref{eq:expr_P_limit0} we observe that in general $\bm{P}$ is not a gradient field, even if $\bm{M}=\bm{0}$. Therefore, by setting $\lambda_\mathrm{e}=\lambda_\mathrm{micro}=0$ we obtain that $\skw \bm{P} = \skw \du$ and $\sym \bm{P}= \frac{\mu_{\mathrm{e}}}{\mu_{\mathrm{e}}+\mu_{\mathrm{micro}}}\sym \du$. As $\mu_{\mathrm{micro}}>0$ we deduce that $\bm{P}$ is not a gradient field. This is a significant deviation from the two dimensional relaxed micromorphic model of antiplane shear analyzed in \cite{Sky2021}, where a gradient field as a right-hand side leads to a gradient field for the microdistortion.

\subsubsection{Mixed form}\label{ssec:mixed}
A major aspect of the relaxed micromorphic continuum is the relation to the classical Cauchy continuum theory. This relation is governed by the material constants, where the characteristic length $\Lc$ plays a significant role. We are therefore interested in robust computations with respect to $\Lc$. To that end, we reformulate the problem as a mixed formulation. The first step consists in introducing the new unknown
\begin{align}
\bm{D} = \muma\,\Lc^2 \Curl \bm{P} \,\in \HD{,\Omega} \, ,\label{eq:def_p}
\end{align}
reminiscent of the micro-dislocation, 
and examining its distribution with a test function
\begin{align}
\int_{\Omega} \langle \Curl \bm{P} , \delta \bm{D} \rangle - \dfrac{1}{\muma\,\Lc^2}\langle \bm{D} , \delta \bm{D} \rangle  \,\dd X = 0 \, .
\end{align} 

In fact, $\bm{D}$ must be solenoidal as
\begin{align}
\Di \bm{D} = \muma\,\Lc^2 \Di\Curl \bm{P} = \vb{0}
\end{align}
and thus the appropriate space for $\bm{D}$ is 
\begin{align}
\HDzd{, \Omega}=\HD{, \Omega} \cap \ker(\Di) = \{\bm{D}\in \HD{, \Omega}\,|\, \Di\bm{D}=\vb{0}\}.
\end{align}

We again assume that Dirichlet boundary conditions are prescribed on the whole boundary and thus, from $\bm{P}\in\HCz{, \Omega}$ there follows $\bm{D}\in\HDz{, \Omega}$ as $\bm{D}\bm{\nu} = \Curl\bm{P}\bm{\nu}=\vb{0}$.

The (bi-)linear forms are now given by
\begin{subequations}
\begin{align}
a(\{\vb{u},\bm{P}\},\{\delta \vb{u},\delta \bm{P}\}) &= \int_{\Omega} \langle\Ce \sym (\D \vb{u} - \bm{P}) , \sym(\D \delta \vb{u} -\delta \bm{P}) \rangle +\langle\Cm \sym\bm{P} ,  \sym\delta \bm{P} \rangle \nonumber\\
&\qquad +  \langle\Cc \skw (\D \vb{u} - \bm{P}) , \skw(\D \delta \vb{u} -\delta \bm{P}) \rangle  \,\dd X \, , \\
b(\{\vb{u},\bm{P}\}, \delta \bm{D}) &= \int_{\Omega} \langle \Curl\bm{P} , \delta \bm{D} \rangle\, \dd X \, , \\
d(\bm{D},\delta \bm{D}) &= \int_{\Omega} \langle \bm{D} , \delta \bm{D} \rangle\, \dd X \, , \\
l(\delta \vb{u}, \delta \bm{P}) &=  \int_{\Omega}   \langle \delta \vb{u},\, \vb{f} \rangle + \langle \delta \bm{P} , \, \bm{M} \rangle \, \dd X \, ,
\end{align}
\end{subequations}
and the resulting mixed formulation reads: Find $(\{\vb{u},\bm{P}\},\bm{D})\in X\times \HDzdz{, \Omega}$ such that
\begin{subequations}
	\label{eq:mixed_problem}
	\begin{alignat}{3}
	&a(\{\vb{u},\bm{P}\},\{\delta \vb{u},\delta \bm{P}\}) +  b(\{\delta \vb{u}, \delta \bm{P}\},  \bm{D}) &&= l(\delta \vb{u}, \delta \bm{P}) \, , 
	&&\quad \forall\, \{\delta \vb{u},\delta \bm{P}\} \in \X \, ,
	\label{eq:mixed_problem_a} \\
	&b(\{ \vb{u}, \bm{P}\}, \delta \bm{D})- \dfrac{1}{\muma\,\Lc^2}d(\bm{D},\delta \bm{D})  &&= 0 \, , &&\quad \forall\, \delta \bm{D} \in \HDzdz{, \Omega} \, , \label{eq:mixed_problem_b} 
	\end{alignat}
\end{subequations}

where the Lagrange multiplier $\bm{D}$ has the physical meaning of a hyperstress. Notice we now approximate the hyperstress directly with its own variable, therefore recovering lost precision due to differentiation.

For the following proofs we will make use of the Helmholtz decomposition \cite{GR86}, splitting a vector field into a curl and gradient potential. For all $\vb{u}\in [\Le(\Omega)]^3$ there exists $\vb{q}\in \Hcz{, \Omega}$ and $\Psi\in \Hone(\Omega)\backslash\mathbb{R}$ such that
\begin{align}
\label{eq:helmholtz}
    &\vb{u} = \curl \vb{q} + \nabla \Psi\,, \qquad \|\vb{q}\|_{\Hc{}}\leq c\|\vb{u}\|_{\Le} \, , \qquad  \|\Psi\|_{\Hone}\leq c\|\vb{u}\|_{\Le} \, .
\end{align}
Further, if $\vb{u}\in \Hdzdz{,\Omega}$ then there exists $\vb{q}\in \Hcz{, \Omega}$  such that
\begin{align}
\label{eq:helmholtz2}
    \vb{u} = \curl \vb{q}\,,\qquad \|\vb{q}\|_{\Hc{}}\leq c\|\vb{u}\|_{\Le}.
\end{align}
\begin{theorem}
\label{thm:ex_un_mixed1}
	Problem \ref{eq:mixed_problem} is uniquely solvable and there holds the stability estimate
	\begin{align*}
	\|\{ \vb{u},\bm{P} \}\|_{\X}+\|\bm{D}\|_{\HD{}} \leq c \, \|l\|_{\Le},\qquad c\neq c(\Lc).
	\end{align*}
\end{theorem}
\begin{proof}
	We use the extended Brezzi-theorem \cite[Thm. 4.11]{Bra2013}. The continuity of $a(\cdot,\cdot)$, $b(\cdot,\cdot)$, $d(\cdot,\cdot)$ and non-negativity of $a(\cdot,\cdot)$ and $d(\cdot,\cdot)$ are obvious. 
	
	Therefore, we have to prove that $a(\cdot,\cdot)$ is coercive on the kernel of $b(\cdot,\cdot)$
	\begin{align}
	\ker (b) = \{ \{\vb{u}, \bm{P} \} \in \X \; | \; b(\{\vb{u}, \bm{P} \}, \delta \bm{D}) = 0 \, , \, \forall \,  \delta \bm{D} \in \HDzd{, \Omega} \} = \{ \{\vb{u}, \bm{P} \} \in \X \, | \, \Curl \bm{P} = \bm{0} \} \, .
	\end{align}
	The last equality follows from the
	exact property $\Curl\big(\HCz{, \Omega}\big) = \HDzdz{, \Omega}$.
	However, we already know that $a(\{\vb{u},\bm{P}\},\{\delta \vb{u},\delta\bm{P}\}) + \langle \Curl\bm{P}, \Curl\delta\bm{P}\rangle_{\Le}$ is coercive, compare the proof of Theorem~\ref{thm:ex_un_primal}. This leaves us with the Ladyzhenskaya--Babu{\v{s}}ka--Brezzi (LBB) condition to be satisfied
	\begin{align}
	\exists \beta_2 > 0 : \quad \sup_{\{\vb{u}, \bm{P}\} \in \X} \dfrac{b(\{\vb{u}, \bm{P}\},\bm{D})}{ \| \{\vb{u}, \bm{P}\} \|_{\X}} \geq \beta_2 \, \|\bm{D}\|_{\HD{}} \, , \quad \forall\, \bm{D} \in \HDzdz{, \Omega} \, .
	\end{align}
	We choose $\vb{u}=\vb{0}$ and $\bm{P}$ such that $\Curl \bm{P}=\bm{D}$ with $\|\bm{P}\|_{\HC{}}\leq c\|\bm{D}\|_{\Le}$ leading to
	\begin{align}
	\sup_{\{\vb{u}, \bm{P}\} \in \X}\dfrac{b(\{\vb{u}, \bm{P}\},\bm{D})}{ \| \{\vb{u}, \bm{P}\} \|_{\X}} &= \sup_{\{\vb{u}, \bm{P}\} \in \X}\dfrac{\int_\Omega \langle\bm{D}, \Curl\bm{P}\rangle\,\dd X}{ \| \bm{P}\|_{\Le}+\| \Curl\bm{P}\|_{\Le}}\geq c \, \dfrac{\|\bm{D}\|^2_{\Le}}{\|\bm{D}\|_{\Le}} = c\,\|\bm{D}\|_{\HD{}}.
	\end{align}
	The existence of such a potential $\bm{P}$ for $\bm{D}\in \HDzdz{, \Omega}$ follows directly from the theory of stable Helmholtz decompositions \eqref{eq:helmholtz2}.
	
	Thus, with Brezzi's theorem \cite{BBF13}, there exists a unique solution independent of $\Lc$ for $\frac{1}{\muma \Lc^2}\leq 1$ fulfilling the stability estimate
	\begin{align}
	\|\{ \vb{u},\bm{P} \}\|_{\X}+\|\bm{D}\|_{\HD{}} \leq c \, \|l\|_{\Le} \, ,\label{eq:mixed_method_stab_est}
	\end{align}
	where the constant $c$ does not depend on $\Lc$.
\end{proof}

\begin{remark}
	In the proof we made use of the exactness property of the de Rham complex by restricting the space of $\bm{D}$ to be divergence free. Otherwise we cannot prove the LBB condition in the limit $\Lc\to\infty$. Numerically we observed that this restriction is necessary, otherwise the resulting matrix becomes singular.
\end{remark}

The limit case $\lim \Lc \to \infty$ of \cref{eq:mixed_problem} is well-defined, resulting in the problem: Find $(\{\vb{u}_{\infty},\bm{P}_{\infty}\},\bm{D}_{\infty})\in \X\times \HDzdz{, \Omega}$ such that
\begin{subequations}
	\label{eq:mixed_problem_limit}
	\begin{alignat}{3}
	& a(\{\vb{u}_\infty,\bm{P}_\infty\},\{\delta \vb{u},\delta \bm{P}\}) +  b(\{\delta \vb{u}, \delta \bm{P}\},  \bm{D}_\infty) &&= l(\delta \vb{u}, \delta \bm{P}) \, , 
	&&\quad \forall\, \{\delta \vb{u},\delta \bm{P}\} \in \X \, ,\label{eq:mixed_problem_limit_a}
	\\
	& b(\{ \vb{u}_\infty, \bm{P}_\infty\}, \delta \bm{D})  &&= 0 \, , &&\quad \forall\, \delta \bm{D} \in \HDzdz{, \Omega} \, . \label{eq:mixed_problem_limit_b}
	\end{alignat}
	\label{eq:lcinf}
\end{subequations}

Consequently, at the limit $\lim \Lc \to \infty$ we have $\Curl \bm{P} = \bm{0}$, i.e. the rows of $\bm{P}$ are gradient fields. This can be observed by considering that $\delta \bm{D}$ only tests functions in $\range(\Curl)$, where the range is fully given by $\range(\Curl) = \HDzd{}$. As such, by the Helmholtz decomposition only the gradient part of the microdistortion $\Pm$ remains. The latter is also clearly observable through the minimization of the energy function, as $\Pm = \D \bm{\Psi}$ is needed for finite energies. 
Therefore, the consistent coupling condition is crucial, prescribing only gradient fields as Dirichlet boundary conditions for $\bm{P}$. From the theory of mixed methods we obtain, analogously to the 2D case in \cite{Sky2021}, quadratic convergence in $\Lc$ towards the limit case
\begin{align}
\| \{\vb{u}_\infty - \vb{u},\bm{P}_\infty\ - \bm{P} \} \|_{\X} + \| \bm{D}_\infty - \bm{D} \|_{\HD{}}  \leq \dfrac{c}{\Lc^2} \|l \|_{\Le},\quad c\neq c(\Lc) \, .\label{eq:mixed_method_conv_Lc}
\end{align}

\begin{remark}\label{re:tocmicro}
    Note that for the limit $\Lc \to \infty$ one finds $\Pm \to \D \bm{\Psi}$, inducing finite energies due to $\Curl \D \bm{\Psi} = 0$. Since $\Lc$ defines a zoom into the microstructure, the latter can be interpreted as the entire domain being the micro-body. Consequently, setting $\vb{f} = \vb{0}$ yields
    \begin{align}
        -\Di[\Ce \sym (\D \vb{u} - \bm{P}) + \Cc \skw (\D \vb{u} - \bm{P})] &= \vb{0} \,,
        \label{eq:cmicro}
    \end{align}
    and as such, taking the divergence of \cref{eq:strong_form_dP} results in
    \begin{align}
        \Di ( \Cm \sym \D \bm{\Psi}) &= \Di \bm{M} \, .
    \end{align}
    The divergence of the micro-moment $\Di\bm{M}$ can be interpreted as the micro body-force. The latter implies that the limit $\Lc \to \infty$ defines again a classical Cauchy continuum theory with a finite stiffness $\Cm$, representing the upper limit of the stiffness for the relaxed micromorphic continuum. Further, due to the consistent coupling condition and \cref{eq:cmicro} there holds $\D \bm{\Psi} = \D \vb{u}$ (for a thorough derivation see \cite{Barbagallo2017}).
\end{remark}

\subsubsection{Reformulation of the mixed divergence free constraint}
	The construction of finite elements which are exactly divergence free, $\overset{\circ}{V^h}\subset \HDzd{, \Omega}$, is possible for higher polynomials, \cite{Zaglmayr2006}. However, at least for the lowest order shape functions, one has to use an additional Lagrange multiplier $\vb{q}\in [\Lez(\Omega)]^3$, compare Remark~\ref{rem:zero_mean} below, forcing $\Di \bm{D}=\vb{0}$. We now show how the mixed formulation has to be adapted leading to a method which can be directly implemented. By defining
	$\widetilde{\X}=\X\times [\Lez(\Omega)]^3$  we introduce the adapted bilinear forms
	\begin{subequations}
	\begin{align}
	\widetilde{a}(\{\vb{u},\bm{P},\vb{q}\},\{\delta \vb{u},\delta \bm{P},\delta \vb{q}\}) &= a(\{\vb{u},\bm{P}\},\{\delta \vb{u},\delta \bm{P}\}) \, , \\
	\widetilde{b}(\{\vb{u},\bm{P},\vb{q}\}, \delta \bm{D}) &= \int_{\Omega} \langle \Curl\bm{P} , \delta \bm{D} \rangle + \langle\vb{q},\Di \delta \bm{D}\rangle \, \dd X \,
	\end{align}
	\end{subequations}
	and the problem reads: find $(\{\vb{u},\bm{P},\vb{q}\},\bm{D})\in \widetilde{\X}\times \HDz{, \Omega}$ such that
	\begin{subequations}
	\label{eq:mixed_problem2}
		\begin{alignat}{3}
		&\tilde{a}(\{\vb{u},\bm{P},\vb{q}\},\{\delta \vb{u},\delta \bm{P},\delta \vb{q}\}) +  \tilde{b}(\{\delta \vb{u}, \delta \bm{P}, \delta \vb{q}\},  \bm{D}) &&= l(\delta \vb{u}, \delta \bm{P}) \, , 
		&&\quad \forall\, \{\delta \vb{u},\delta \bm{P},\delta \vb{q}\} \in \tilde{\X} \, ,
		\\
		&\tilde{b}(\{ \vb{u}, \bm{P},\vb{q}\}, \delta \bm{D})- \dfrac{1}{\muma\,\Lc^2}d(\bm{D},\delta \bm{D})  &&= 0 \, , &&\quad \forall\, \delta \bm{D} \in \HDz{, \Omega} \, .
		\end{alignat}
	\end{subequations}

\begin{theorem}
\label{thm:ex_un_mixed2}
	Problem~\ref{eq:mixed_problem2} is uniquely solvable and the solution solves also Problem~\ref{eq:mixed_problem}.
\end{theorem}
\begin{proof}
The kernel of the bilinear form $b$ is now given by
	\begin{align}
	\ker (b) = \{ \{\vb{u}, \bm{P},\vb{q} \} \in \tilde{\X} \; | \; b(\{\vb{u}, \bm{P},\vb{q} \}, \delta \bm{D}) = 0 \quad \forall \delta \bm{D} \in \HDz{, \Omega} \} = \{ \{\vb{u}, \bm{P},\vb{q} \} \in \tilde{\X} \; | \; \Curl \bm{P} = \bm{0} \text{ and } \vb{q}=\vb{0} \} \,.
	\end{align}
	The last equality follows by choosing on the one hand $\bm{D}=\Curl \bm{P}\in\HDz{,\Omega}$ yielding $\Curl\bm{P}=\bm{0}$ and on the other hand from the de Rham sequence we find that $\Di \left(\HDz{, \Omega}\right)=\Lez(\Omega)$ such that we can choose $\bm{D}\in \HDz{, \Omega}$ such that $\Di\bm{D}=\vb{q}$. With the same argument as in the proof of Theorem~\ref{thm:ex_un_mixed1} the kernel coercivity follows.
	For the LBB condition we use the Helmholtz decomposition \eqref{eq:helmholtz} $\bm{D}=\Curl\bm{Q}+\D \vb{\Psi}$, $\bm{Q}\in\HCz{,\Omega}$ and $\vb{\Psi}\in[\Hone(\Omega)\backslash\mathbb{R}]^3$, and define $\vb{u}=\vb{0}$, $\bm{P}=\bm{Q}$, and $\vb{q}=\Di\bm{D}-\vb{\Psi}\in [\Lez(\Omega)]^3$. Then there holds
	\begin{align*}
	\sup\limits_{\{\vb{u},\bm{P},\vb{q}\}\in\tilde{\X}}\dfrac{b(\{\vb{u}, \bm{P}, \vb{q}\},\bm{D})}{ \| \{\vb{u}, \bm{P}, \vb{q}\} \|_{\tilde{\X}}} &= \sup\limits_{\{\vb{u},\bm{P},\vb{q}\}\in\tilde{\X}}\dfrac{\int_{\Omega} \langle\bm{D}, \Curl\bm{P}\rangle + \langle\vb{q}, \Di\bm{D}\rangle\,\dd X}{\| \bm{P}\|_{\Le}+\| \Curl\bm{P}\|_{\Le}+\|\vb{q}\|_{\Le}}\\
	&\geq  \dfrac{\int_{\Omega} \langle\Curl\bm{Q}+\D \vb{\Psi},\Curl\bm{Q}\rangle - \langle\vb{\Psi}, \Di\D\vb{\Psi}\rangle+\|\Di \bm{D}\|^2\,\dd X}{ \| \bm{Q}\|_{\Le}+\| \Curl\bm{Q}\|_{\Le}+\|\Di \bm{D}-\vb{\Psi}\|_{\Le}}\\
	&\geq\dfrac{\int_{\Omega} \|\Curl\bm{Q}\|^2 + \|\Di \bm{D}\|^2+\|\D \vb{\Psi}\|^2\,\dd X}{ \| \bm{Q}\|_{\Le}+\| \Curl\bm{Q}\|_{\Le}+\|\Di \bm{D}\|_{\Le}+\|\vb{\Psi}\|_{\Le}}\geq c \|\bm{D}\|_{\HD{}}.
	\end{align*}
	Now, with Brezzi's theorem we conclude that Problem~\ref{eq:mixed_problem2} is uniquely solvable with a stability constant independent of $\Lc$.
	
	Due to the exactness property of the de Rham complex the additional Lagrange multiplier $\vb{q}$ enforces that $\bm{D}\in\HDzdz{,\Omega}$. Therefore,  taking $(\{\vb{u},\bm{P}\},\bm{D})\in \X\times \HDzdz{,\Omega}$ from the solution solves also Problem~\ref{eq:mixed_problem}.
\end{proof}

\begin{remark}
\label{rem:zero_mean}
    In case of prescribed Dirichlet boundary conditions on $\Gamma_D^P$ for $\Pm \in \HC{,\Omega}$, due to the de Rham complex, compatible Dirichlet conditions have to be used for $\bm{D}\in \HD{,\Omega}$. If $\Gamma_D^P=\partial\Omega$ then from $\Pm \in \HCz{,\Omega}$ there follows $\bm{D}\in \HDz{,\Omega}$. Thus, with $\Di(\HDz{,\Omega})=[\Lez(\Omega)]^3 $ the space for $\vb{q}$ has a zero mean. Consequently, another vector valued variable in $\mathbb{R}^3$ is needed for numerics to ensure the zero average over the domain of $\vb{q}$, since elements of $\Le$ cannot be prescribed on the boundary.
\end{remark}

\begin{figure}
	\centering
	\begin{tikzpicture}[scale = 0.6][line cap=round,line join=round,>=triangle 45,x=1.0cm,y=1.0cm]
		\clip(2.9,7) rectangle (24.5,10);
		\draw (2.9,9) node[anchor=north west] {$[\mathit{H}^1(\Omega)]^3$};
		\draw (8.7,9) node[anchor=north west] {$\HC{, \Omega}$};
		\draw [->,line width=1.5pt] (5.7,8.5) -- (8.5,8.5);
		\draw (6.5,9.5) node[anchor=north west] {$\D$};
		\draw [->,line width=1.5pt] (12.2,8.5) -- (15.2,8.5);
		\draw (12.8,9.5) node[anchor=north west] {$\Curl$};
		\draw (15.5,9) node[anchor=north west]  {$\HD{,\Omega}$};
		\draw [->,line width=1.5pt] (18.8,8.5) -- (21.8,8.5);
		\draw (19.4,9.5) node[anchor=north west] {$\Di$};
		\draw (22,9) node[anchor=north west] {$[\Le(\Omega)]^3$};
	\end{tikzpicture}
	\caption{The classical de Rham exact sequence.
		The range of each operator is exactly the kernel of the next operator in the sequence.}
	\label{fig:rham}
\end{figure} 

\subsection{Discrete case}\label{ch:3}
We note that existence and uniqueness in the primal formulation is given by the Lax-Milgram theorem. As a result, the use of a commuting diagram for the relation between the displacement $\vb{u}$ and the microdistortion $\Pm$ is not necessary. However, the relaxed micromorphic model introduces the so called consistent coupling condition on the Dirichlet boundary, which can be satisfied exactly in the general case, if commuting interpolants are employed. Further, the mixed formulation requires the coercivity on the kernel of the bilinear form $b(\{\vb{u}, \Pm\}, \Ds)$ to also be satisfied in the discrete case. Consequently, we rely on the commuting de Rham diagram for the construction of our finite elements. Specifically, for the lower order elements we make use of linear and quadratic tetrahedral Lagrangian elements, linear elements from the first and second N\'ed\'elec spaces and lowest-order Raviart-Thomas elements, see \cref{fig:comm}.

We denote in the following
\begin{align*}
    \vb{u}^h,\delta \vb{u}^h &\in V^h\subset[\Hone(\Omega)]^3\,,& \bm{P}^h,\delta \bm{P}^h&\in U^h\subset \HC{,\Omega}\,,\\
    \bm{D}^h,\delta \bm{D}^h &\in\Sigma^h\subset \HD{,\Omega}\,,& \vb{q}^h,\delta \vb{q}^h &\in Q^h\subset [\Le(\Omega)]^3\,.
\end{align*}
On each tetrahedral element $T$ we define $P^p(T)$ as the set of all polynomials up to order $p$, $P^p(T)=\{ x^iy^jz^l\,\big|\, i+j+l\leq p,\, i,j,l\geq 0 \}$ and denote with $P^p(\Omega)$ the set of piece-wise polynomials of order $p$ on the triangulation of $\Omega$.

\begin{remark}
	The Raviart-Thomas element from the space $\RT^0$ can only produce constant normal projections on an element's outer surface. Further, the dimension of the space is smaller than the full linear space, $\dim \RT^0 = 4 < \dim [\Po^1]^3 = 12$.
	
	The lowest-order Brezzi--Douglas--Marini elements $\BDM^1$ have linear normal components as the full linear polynomial space is considered. On the one hand better $\Le$-estimates are therefore possible, however, on the other hand the range of the divergence operator is not increased. There holds the sequence
	\begin{align}
	   [P^0(\Omega)]^3\subsetneq \RT^0\subsetneq \BDM^1=[P^1(\Omega)]^3 \subsetneq \RT^1 \subsetneq  \BDM^2=[P^2(\Omega)]^3 \subsetneq\dots \,.
	\end{align}
\end{remark}

From Section~\ref{ssec:mixed} we know that the hyperstress field $\bm{D}\in\HD{,\Omega}$ is solenoidal, $\Di \bm{D}= \vb{0}$. Therefore, aside from the lowest order case $\RT^0$, it is desirable to use a basis for $\RT^p$ (or $\BDM^p$) where all higher-order shape functions are divergence-free. We denote the adapted space by $\RT^p_0=\{\vb{d}\in \RT^p\;|\; \vb{d}\in \RT^0 \text{ or } \di \vb{d}=0\}$ and the corresponding interpolation operator as $\Pi^p_d$. With such a construction, we can on the one hand save several redundant degrees of freedom, reducing the computational costs for assembly and solution steps without any loss of accuracy, and on the other hand, only a constant correction term $\vb{q}\in [P^0(\Omega)]^3$ is required to enforce $\Di \bm{D}=\vb{0}$ for the remaining lowest-order $\RT^0$ shape functions in the mixed formulation since
\begin{align*}
    \Di [\RT^p_0]^3 = \Di [\RT^0]^3 = [P^0(\Omega)]^3 \, .
\end{align*} 
In the numerical examples we use NGSolve for the high-order elements where such high-order divergence-free shape functions exist following the construction of \cite{Zaglmayr2006}.

\begin{remark}
	The polynomial order of the sequence can be increased by using the N\'ed\'elec elements of the first type $\Ned^1_{I}$ instead of $\Ned^1_{II}$, increasing the range of the curl differential operator. We note that both types yield a tangential projection of the same polynomial power on the outer boundaries of the element. Similar to the Raviart--Thomas and Brezzi--Douglas--Marini elements there holds
	\begin{align}
	   [P^0(\Omega)]^3\subsetneq \Ned_{I}^0\subsetneq \Ned_{II}^1=[P^1(\Omega)]^3 \subsetneq \Ned_{I}^1 \subsetneq  \Ned_{II}^2=[P^2(\Omega)]^3 \subsetneq\dots \,.
	\end{align}
\end{remark}

\begin{figure}
	\centering
	\begin{subfigure}{1.\linewidth}
	\centering
	    \begin{tikzpicture}[scale = 0.6][line cap=round,line join=round,>=triangle 45,x=1.0cm,y=1.0cm]
		\clip(3,3) rectangle (24.5,10);
		\draw (3,9) node[anchor=north west] {$[\mathit{H}^1(\Omega)]^3$};
		\draw (8.7,9) node[anchor=north west] {$\HC{, \Omega}$};
		\draw [->,line width=1.5pt] (5.7,8.5) -- (8.5,8.5);
		\draw (6.5,9.5) node[anchor=north west] {$\D$};
		\draw [->,line width=1.5pt] (12.2,8.5) -- (15.2,8.5);
		\draw (12.8,9.5) node[anchor=north west] {$\Curl$};
		\draw (15.5,9) node[anchor=north west]  {$\HD{,\Omega}$};
		\draw [->,line width=1.5pt] (18.8,8.5) -- (21.8,8.5);
		\draw (19.4,9.5) node[anchor=north west] {$\Di$};
		\draw (21.8,9) node[anchor=north west] {$[\Le(\Omega)]^3$};
		
		\draw (3,5) node[anchor=north west] {$[\Po^1(\Omega)]^3$};
		\draw (9.,5) node[anchor=north west] {$[\Ned_{I}^0(\Omega)]^3$};
		\draw [->,line width=1.5pt] (5.7,4.5) -- (8.5,4.5);
		\draw (6.5,5.5) node[anchor=north west] {$\D$};
		\draw [->,line width=1.5pt] (12.2,4.5) -- (15.2,4.5);
		\draw (12.8,5.5) node[anchor=north west] {$\Curl$};
		\draw (15.5,5) node[anchor=north west]  {$[\RT^0(\Omega)]^3$};
		\draw [->,line width=1.5pt] (18.8,4.5) -- (21.8,4.5);
		\draw (19.4,5.5) node[anchor=north west] {$\Di$};
		\draw (21.8,5) node[anchor=north west] {$[\Po^0(\Omega)]^3$};
		
		\draw [->,line width=1.5pt] (4.5,7.8) -- (4.5,5.2);
		\draw (4.5,6.5) node[anchor=east] {$\Pi^1_g$};
		\draw [->,line width=1.5pt] (10.5,7.8) -- (10.5,5.2);
		\draw (10.5,6.5) node[anchor=east] {$\Pi^{I,0}_c$};
		\draw [->,line width=1.5pt] (17,7.8) -- (17,5.2);
		\draw (17,6.5) node[anchor=east] {$\Pi^{\RT,0}_d$};
		\draw [->,line width=1.5pt] (23,7.8) -- (23,5.2);
		\draw (23,6.5) node[anchor=east] {$\Pi^0_o$};
	\end{tikzpicture}
	\end{subfigure}
	\begin{subfigure}{1.\linewidth}
	\centering
	    \begin{tikzpicture}[scale = 0.6][line cap=round,line join=round,>=triangle 45,x=1.0cm,y=1.0cm]
		\clip(3,3) rectangle (24.5,10);
		\draw (3,9) node[anchor=north west] {$[\mathit{H}^1(\Omega)]^3$};
		\draw (8.7,9) node[anchor=north west] {$\HC{, \Omega}$};
		\draw [->,line width=1.5pt] (5.7,8.5) -- (8.5,8.5);
		\draw (6.5,9.5) node[anchor=north west] {$\D$};
		\draw [->,line width=1.5pt] (12.2,8.5) -- (15.2,8.5);
		\draw (12.8,9.5) node[anchor=north west] {$\Curl$};
		\draw (15.5,9) node[anchor=north west]  {$\HD{,\Omega}$};
		\draw [->,line width=1.5pt] (18.8,8.5) -- (21.8,8.5);
		\draw (19.4,9.5) node[anchor=north west] {$\Di$};
		\draw (21.8,9) node[anchor=north west] {$[\Le(\Omega)]^3$};
		
		\draw (3,5) node[anchor=north west] {$[\Po^2(\Omega)]^3$};
		\draw (9.,5) node[anchor=north west] {$[\Ned_{II}^1(\Omega)]^3$};
		\draw [->,line width=1.5pt] (5.7,4.5) -- (8.5,4.5);
		\draw (6.5,5.5) node[anchor=north west] {$\D$};
		\draw [->,line width=1.5pt] (12.2,4.5) -- (15.2,4.5);
		\draw (12.8,5.5) node[anchor=north west] {$\Curl$};
		\draw (15.5,5) node[anchor=north west]  {$[\RT^0(\Omega)]^3$};
		\draw [->,line width=1.5pt] (18.8,4.5) -- (21.8,4.5);
		\draw (19.4,5.5) node[anchor=north west] {$\Di$};
		\draw (21.8,5) node[anchor=north west] {$[\Po^0(\Omega)]^3$};
		
		\draw [->,line width=1.5pt] (4.5,7.8) -- (4.5,5.2);
		\draw (4.5,6.5) node[anchor=east] {$\Pi^2_g$};
		\draw [->,line width=1.5pt] (10.5,7.8) -- (10.5,5.2);
		\draw (10.5,6.5) node[anchor=east] {$\Pi^{II,1}_c$};
		\draw [->,line width=1.5pt] (17,7.8) -- (17,5.2);
		\draw (17,6.5) node[anchor=east] {$\Pi^{\RT,0}_d$};
		\draw [->,line width=1.5pt] (23,7.8) -- (23,5.2);
		\draw (23,6.5) node[anchor=east] {$\Pi^0_o$};
	\end{tikzpicture}
	\end{subfigure}
	\caption{The de Rham complex for linear (upper) and quadratic (lower) sequences. The differential and interpolation operators commute between the continuous and the discrete spaces.}
	\label{fig:comm}
\end{figure} 

In the Lax--Milgram setting we define $\X^h=V^h\times U^h$. Due to $\X^h\subset \X$ the solvability of the discrete problem follows directly from the continuous case \cite{Bochev,Ciarl78}. We use Cea's lemma for the quasi-best approximation as a basis for the convergence estimate
\begin{align*}
    \|\{\vb{u},\bm{P}\}-\{\vb{u}^h,\bm{P}^h\}\|_{\X}^2 \leq c(\Lc,\muma,\Ce,\Cc,\Cm)\inf\limits_{\{\delta\vb{u}^h,\delta\bm{P}^h\}\in\X^h }\|\{\vb{u},\bm{P}\}-\{\delta\vb{u}^h,\delta\bm{P}^h\}\|_{\X}^2.
\end{align*}

\begin{remark}
\label{rem:projection_ops}
    When using the canonical interpolation operators in the de Rham complex depicted in Figure~\ref{fig:comm} additional smoothness of the function spaces has to be assumed, as e.g., point evaluation is not well-defined for functions $u\in \Hone(\Omega)$, $\Omega\subset\mathbb{R}^d$ in dimension $d>1$. First commuting, but non-local projections without additional regularity assumptions were constructed in \cite{Schoe05} and \cite{CW08}. Very recently, local $\Le$-bounded and commuting projection operators from the function spaces into the finite element spaces have been established in \cite{AG21}. Therefore, in the theoretical proofs we will make use of these novel projection operators, whereas for the finite element implementations the usual canonical operators are considered.
\end{remark}
\begin{lemma}
\label{lem:conv_primal1}
Assume that the exact solution $\{\vb{u},\bm{P}\}$ is in $\mathit{H}^{p+1}(\Omega)\times \mathit{H}^p(\Curl, \Omega)$, where $\mathit{H}^p(\Curl, \Omega)=\{\bm{P}\in [\mathit{H}^p(\Omega)]^{3\times 3}\,|\, \Curl\bm{P}\in [\mathit{H}^p(\Omega)]^{3\times 3}\}$. If  $[P^p(\Omega)]^3\subset V^h$ and $U^h=[\Ned_{I}^{p-1}]^3$ then the discrete solution $\{\vb{u}^h,\bm{P}^h\}\in\X^h$ converges with optimal rate
\begin{align}
\|\vb{u}-\vb{u}^h\|_{\Hone}+\|\bm{P}-\bm{P}^h\|_{\HC{}}\leq c(\Lc,\muma,\Ce,\Cc,\Cm) \, h^p.
\end{align}
\end{lemma}
\begin{proof}
During the proof $\Pi_c^{I,p-1}$ denotes, with an abuse of notation, the projection operator without extra regularity assumptions, see Remark~\ref{rem:projection_ops}. With Cea's lemma and the approximation properties of the interpolation operators \cite{Ciarl78,BBF13} we obtain
\begin{align*}
    \|\{\vb{u},\bm{P}\}-\{\vb{u}^h,\bm{P}^h\}\|_{\X}^2 &\leq c\inf\limits_{\{\delta\vb{u}^h,\delta\bm{P}^h\}\in\X^h }\|\{\vb{u},\bm{P}\}-\{\delta\vb{u}^h,\delta\bm{P}^h\}\|_{\X}^2\\
    &\leq c\Big(\|\vb{u}-\Pi_g^p\vb{u}\|_{\Hone}^2 + \|\bm{P} - \Pi_c^{I,p-1}\bm{P}\|_{\Le}^2 + \|\Curl\bm{P} - \Curl\Pi_c^{I,p-1}\bm{P}\|_{\Le}^2  \Big)\\
    &= c\Big(\|\vb{u}-\Pi_g^p\vb{u}\|_{\Hone}^2 + \|\bm{P} - \Pi_c^{I,p-1}\bm{P}\|_{\Le}^2 + \|(\mathrm{id}-\Pi_d^{p-1})(\Curl\bm{P})\|_{\Le}^2  \Big)\\
    &\leq c\, h^{2p}\Big( |\vb{u}|^2_{H^{p+1}} + |\bm{P}|^2_{H^{p}} + |\Curl\bm{P}|^2_{H^{p}} \Big),
\end{align*}
where $|\cdot|_{H^p}$ denotes the standard Sobolev semi-norm.
\end{proof}
When using the N\'ed\'elec elements of second type $\Ned_{II}^{p-1}$ instead of $\Ned_{I}^{p-1}$, we lose one order of convergence for the microdistortion tensor due to the decreased range of the curl. This should also lead
to sub-optimal convergence of the displacement. For example, the quadratic sequence depicted in Figure~\ref{fig:comm} would only have linear convergence for $\vb{u}$ although quadratic elements are used. In the extreme case of vanishing characteristic length $\Lc=0$ the optimal convergence rates are achieved since only the $\Le$-norm of $\bm{P}$ is considered, compare Section~\ref{sec:lczero}. On the subspace $U_0=\{\vb{v}\in\Hc{, \Omega}\,|\, \curl\vb{v}=0\}$ the N\'ed\'elec elements of first and second kind coincide, $\Ned_I^p = \Ned_{II}^p$, leading to optimal rates, which corresponds to the limit case $\Lc=\infty$. However, in the numerical experiments we continued to observe improved convergence rates for the $\Le$-norm of $\bm{P}$ and $\Hone$-norm of $\vb{u}$ for $\Lc\in(0,\infty)$. Under the assumption of an $s$-regular problem we can prove improved estimates for these norms with an Aubin-Nitsche technique. We call Problem~\ref{eq:primal_problem} $s$-regular if there holds for the solution $\{\vb{u},\bm{P}\}\in X$ with $s\in (0,1]$
\begin{align}
|\vb{u}|_{H^{1+s}} + \|\bm{P}\|_{H^s(\mathrm{Curl})} \leq \big(\|\vb{f}\|_{\Le} + \|\bm{M}\|_{\Le}\big)\, ,
\end{align}
where $\|\bm{P}\|^2_{H^s(\mathrm{Curl})}=\|\bm{P}\|^2_{H^s}+\|\Curl\bm{P}\|^2_{H^s}$.

\begin{corollary}
\label{cor:aubin_nitsche}
Adopt the assumptions from Lemma~\ref{lem:conv_primal1} only changing $U^h=[\Ned_{II}^{p-1}]^3$ and $p>1$. Further assume that Problem~\ref{eq:primal_problem} is $s$-regular with $s\in (0,1]$. Then
\begin{align}
\|\vb{u}-\vb{u}^h\|_{\Hone}+\|\bm{P}-\bm{P}^h\|_{\Le}\leq c\, (\Lc,\muma,\Ce,\Cc,\Cm) \, h^{p-1+s} \, .
\end{align}
\end{corollary}
\begin{proof}
We solve the following dual problem, where the differences $\vb{u}-\vb{u}^h$ and $\bm{P}-\bm{P}^h$ are used as right-hand side: find $\{\vb{w},\bm{Q}\}\in X$ such that
\begin{align*}
    a(\{\delta \vb{u},\delta \bm{P}\},\{\vb{w},\bm{Q}\}) = \int_{\Omega} \langle \vb{u}-\vb{u}^h,\delta \vb{u}\rangle + \langle\bm{P}-\bm{P}^h,\delta\bm{P}\rangle\,\dd X\qquad \forall \{\delta \vb{u},\delta \bm{P}\}\in \X \, .
\end{align*}
By using the test-function $\{\delta \vb{u},\delta \bm{P}\} = \{\vb{u}-\vb{u}^h, \bm{P}-\bm{P}^h\}$ and the Galerkin-orthogonality $a(\{\vb{u}-\vb{u}^h, \bm{P}-\bm{P}^h\},\{\vb{v}^h,\bm{Q}^h\})=0$ for all $\{\vb{v}^h,\bm{Q}^h\}\in X^h$ we can insert the corresponding natural interpolation operators $\Pi_g^p\vb{w}$ and $\Pi_c^{p-1}\bm{Q}$ (Remark~\ref{rem:projection_ops}) leading to
\begin{align*}
\|\vb{u}-\vb{u}^h\|_{\Le}^2+\|\bm{P}-\bm{P}^h\|_{\Le}^2 &= a(\{\vb{u}-\vb{u}^h, \bm{P}-\bm{P}^h\},\{\vb{w}-\Pi_g^p\vb{w},\bm{Q}-\Pi_c^{p-1}\bm{Q}\})\\
&\leq c\|\{\vb{u}-\vb{u}^h, \bm{P}-\bm{P}^h\}\|_X (\|\vb{w}-\Pi_g^p\vb{w}\|_{\Hone}+\|\bm{Q}-\Pi_c^{p-1}\bm{Q}\|_{\HC{}})\\
&\leq c\|\{\vb{u}-\vb{u}^h, \bm{P}-\bm{P}^h\}\|_Xh^s(|\vb{w}|_{H^{1+s}}+\|\bm{Q}\|_{H^s(\mathrm{Curl})})\\
&\leq ch^s\|\{\vb{u}-\vb{u}^h, \bm{P}-\bm{P}^h\}\|_X(\|\vb{u}-\vb{u}^h\|_{\Le}+\|\bm{P}-\bm{P}^h\|_{\Le})\, .
\end{align*}
Using Young's inequality on the left side, dividing through $\|\vb{u}-\vb{u}^h\|_{\Le}+\|\bm{P}-\bm{P}^h\|_{\Le}$, and using that $\|\{\vb{u}-\vb{u}^h, \bm{P}-\bm{P}^h\}\|_X\leq ch^{p-1}$ for $\bm{P}^h\in [\Ned_{II}^{p-1}]^3$ yields the claim for the $\Le$-norm of $\vb{u}$ and $\bm{P}$. For the $\Hone$-norm we first note that with the Galerkin-orthogonality, adding and subtracting $\Pi_g^p\vb{u}$, and the specific structure of the bilinear form we obtain
\begin{align*}
0 &= a(\{\vb{u}-\vb{u}^h, \bm{P}-\bm{P}^h\},\{\vb{v}^h,\bm{0}\}) = \int_{\Omega}\langle(\Ce+\Cc)\big( \D\,(\vb{u}-\vb{u}^h)-(\bm{P}-\bm{P}^h)\big),\D\,\vb{v}^h\rangle\,\dd X\\
&=a(\{\vb{u}-\Pi_g^p\vb{u}, \bm{P}-\bm{P}^h\},\{\vb{v}^h,\bm{0}\}) + \int_{\Omega}\langle(\Ce+\Cc)\D\,(\Pi_g^p\vb{u}-\vb{u}^h),\D\,\vb{v}^h\rangle\,\dd X \, .
\end{align*}
Thus, we can estimate
\begin{align*}
    \sqrt{\int_{\Omega}\langle(\Ce+\Cc) \D\,(\Pi_g^{p}\vb{u}-\vb{u}^h),\D\,(\Pi_g^{p}\vb{u}-\vb{u}^h)\rangle\,\dd X} &= \sup\limits_{\vb{v}^h\in V^h}\frac{\int_{\Omega} \langle(\Ce+\Cc)\D\,(\Pi_g^{p}\vb{u}-\vb{u}^h),\D\vb{v}^h\rangle\,\dd X}{\sqrt{\int_{\Omega}\langle(\Ce+\Cc) \D\,\vb{v}^h,\D\,\vb{v}^h\rangle\,\dd X}}\\
    &= \sup\limits_{\vb{v}^h\in V^h} \frac{-a(\{\vb{u}-\Pi_g^{p}\vb{u},\bm{P}-\bm{P}^h\},\{\vb{v}^h,\bm{0}\})}{\sqrt{\int_{\Omega}\langle(\Ce+\Cc) \D\,\vb{v}^h,\D\,\vb{v}^h\rangle\,\dd X}}\\
    &\leq c(\|\vb{u}-\Pi_g^{p}\vb{u}\|_{\Hone}+\|\bm{P}-\bm{P}^h\|_{\Le})\leq ch^{p-1+s}\, ,
\end{align*}
where we used Cauchy-Schwarz and the explicit structure of the bilinear form. With the triangle inequality $\|\D\,(\vb{u}-\vb{u}^h)\|_{\Le}\leq  \|\D\,(\Pi_g^{p}\vb{u}-\vb{u}^h)\|_{\Le}+\|\D\,(\vb{u}-\Pi_g^{p}\vb{u})\|_{\Le}$ and positive definiteness of $\Ce+\Cc$ (in combination with the generalized Korn's inequality for incompatible tensor fields if $\Cc=0$) the claim follows.
\end{proof}

\begin{remark}
We do not discuss the $s$-regularity of Problem~\ref{eq:primal_problem} in this work, but mention that for the Laplace as well as for Maxwell's equations there holds $s=1$ for convex domains $\Omega$.

Corollary~\ref{cor:aubin_nitsche} also shows that we cannot expect improved convergence for $\vb{u}$ in the $\Le$-norm compared to the $\Hone$-norm, which is confirmed by the numerics.
\end{remark}

Unfortunately, the approximation constant depends a priori on $\Lc$. As we are interested in the limit $\Lc\to\infty$ robust estimates with respect to $\Lc$ are desirable. We show such an estimate with the help of the equivalent mixed formulation of the problem.


Normally, in contrast to the primal formulation, the kernel coercivity as well as the LBB-condition for the discrete case are not inherited from the continuous proof. However, thanks to the de Rham complex and a discrete Helmholtz decomposition \cite{AFW97}, these properties hold also in the discrete case as long as the appropriate conforming finite element spaces are employed \cite{BBF13}.
Consequently, Cea's lemma yields
\begin{align*}
    \|\{\vb{u},\bm{P}, \vb{q}\}-\{\vb{u}^h,\bm{P}^h,\vb{q}^h\}\|_{\tilde{\X}}^2 + \|\bm{D}-\bm{D}^h\|_{\HD{}}^2 \leq  c\inf\limits_{ (\{\delta\vb{u}^h,\delta\bm{P}^h,\delta\vb{q}^h\},\delta\bm{D})\in \tilde{X}\times \HD{}} \begin{aligned}
        (&\|\{\vb{u},\bm{P}, \vb{q}\}-\{\delta\vb{u}^h,\delta\bm{P}^h,\delta\vb{q}^h\}\|_{\tilde{\X}}^2\\
        &+ \|\bm{D}-\delta\bm{D}^h\|_{\HD{}}^2)
    \end{aligned}
\end{align*}

\begin{lemma}
\label{lem:mixed_fe_conv}
Assume that the exact solution $\{\vb{u},\bm{P},\bm{D},\vb{q}\}$ is in $\mathit{H}^{p+1}(\Omega)\times \mathit{H}^p(\Curl, \Omega)\times \mathit{H}^p(\Omega)\times \Hone(\Omega)$. If  $[P^p(\Omega)]^3\subset V^h$, $U^h=[\Ned_{I}^{p-1}]^3$,  $\Sigma^h=[\RT_0^{p-1}]^3$, and $[P^{0}(\Omega)]^3= Q^h$ then the discrete solution $(\{\vb{u}^h,\bm{P}^h,\vb{q}^h\},\bm{D}^h)\in\tilde{\X}^h\times \Sigma^h$ converges with optimal rate
\begin{align}
\label{eq:conv_rate_mixed}
\|\vb{u}-\vb{u}^h\|_{\Hone}+\|\bm{P}-\bm{P}^h\|_{\HC{}}+\|\bm{D}-\bm{D}^h\|_{\HD{}}+ h^{p-1}\|\vb{q}-\vb{q}^h\|_{\Le}\leq c \,  h^p\qquad c\neq c(\Lc) \, .
\end{align}
Additionally, with $\{\vb{u}_{\infty},\bm{P}_{\infty},\bm{D}_{\infty},\vb{q}_{\infty}\}$ for the (smooth) solution of the limit problem one obtains
	\begin{align}
	\label{eq:conv_lc_h}
	& \|\vb{u}_{\infty}-\vb{u}^h\|_{\Hone} + \|\bm{P}_{\infty}-\bm{P}^h\|_{\HC{}} + \|\bm{D}_{\infty}-\bm{D}^h\|_{\HD{}}\leq \dfrac{c_1}{\Lc^2} + c_2 \, h^p\,,\quad \|\vb{q}_{\infty}-\vb{q}^h\|_{\Le}\leq \dfrac{c_1}{\Lc^2} + c_2 \, h \, .
	\end{align}
\end{lemma}
\begin{proof}
The proof of the first claim starts by following the same lines as in the primal case. The estimates for $\bm{D}$ and $\vb{q}$ are shown in detail
\begin{align*}
    \|\bm{D}-\Pi^{p-1}_d\bm{D}\|^2_{\HD{}} + \|\vb{q}-\Pi^{0}_o\vb{q}\|^2_{\Le} &=  \|\bm{D}-\Pi^{p-1}_d\bm{D}\|^2_{\Le} + \|(\mathrm{id}-\Pi^{p-1}_o)\underbrace{\Di\bm{D}}_{=0}\|^2_{\Le} + \|\vb{q}-\Pi^{0}_o\vb{q}\|^2_{\Le}\\
    &\leq c( h^{2p}|\bm{D}|^2_{H^{p}} + h^2|\vb{q}|^2_{H^{1}} ),
\end{align*}
where we used that the exact hyperstress $\bm{D}$ is divergence free. This estimate, however, would only yield linear convergence for all fields. By noting that on the subspace of divergence free hyperstresses the solution is independent of $\vb{q}$ we obtain the optimal rates for $\vb{u}^h$, $\bm{P}^h$, and $\bm{D}^h$.

For the second inequality we add and subtract the exact solution, use the triangle inequality, and \eqref{eq:mixed_method_conv_Lc} and \eqref{eq:conv_rate_mixed}.
\end{proof}

\begin{remark}
Note that only linear convergence is obtained for the correction term $\vb{q}$. If one is interested in this quantity, a cheap post-processing step can be applied element-wise to obtain a high-order approximation \cite{LS15}. For $U^h=[\Ned_{II}^{p-1}]^3$ the results of Lemma~\ref{lem:mixed_fe_conv} can directly be adapted by following the same lines as in Corollary~\ref{cor:aubin_nitsche}. 

Inequalities \eqref{eq:conv_lc_h} show that besides the model error in $\Lc$, also the approximation error has to be considered for convergence studies in the limit $\Lc\to\infty$.
\end{remark}

\section{Finite element formulations} \label{ch:4}
In this work we employ tetrahedral finite elements. The construction of $\Hone$, $\Hd{}$, $\Hc{}$, and $\Le$-conforming finite elements to obtain the lowest order combination is presented in the following in detail. Higher order elements are used from the open source finite element software NETGEN/NGSolve\footnote{www.ngsolve.org} \cite{Sch1997,Sch2014} and we refer to \cite{Zaglmayr2006} for the construction of higher order finite elements.

The elements are mapped from the reference element to the physical element using the barycentric base functions
\begin{align}
    b_1(\xi, \, \eta, \, \zeta) &= 1 - \xi - \eta - \zeta \, , \qquad  b_2(\xi, \, \eta, \, \zeta) = \xi \, , \qquad  b_3(\xi, \, \eta, \, \zeta) = \eta \, , \qquad b_4(\xi, \, \eta, \, \zeta) = \zeta \, , \label{eq:bary} \\[1ex]
	\vb{x}(\xi, \, \eta, \, \zeta) &= (1-\xi -\eta -\zeta) \, \vb{x}_1 + \xi \, \vb{x}_2 + \eta \, \vb{x}_3 + \zeta \, \vb{x}_4 \, , \qquad \bm{J} = \begin{bmatrix}
		\vb{x}_2 - \vb{x}_1 & \vb{x}_3 - \vb{x}_1 & \vb{x}_4 - \vb{x}_1
	\end{bmatrix} \, ,
	\label{eq:jacobi}
\end{align} 
where $\vb{x}_i$ are the vertex coordinates of each tetrahedron on the physical domain and $\bm{J}$ is the corresponding Jacobi matrix, see \cref{fig:map}.
The entire domain is given by the union
\begin{align}
	\Omega = \bigcup_{e=1}^n \Omega_e \subset \mathbb{R}^3 \, .
\end{align}

\begin{figure}
	\centering
	\definecolor{asl}{rgb}{0.4980392156862745,0.,1.}
	\definecolor{asb}{rgb}{0.,0.4,0.6}
	\begin{tikzpicture}
		\begin{axis}
			[
			width=30cm,height=17cm,
			view={50}{15},
			enlargelimits=true,
			xmin=-1,xmax=2,
			ymin=-1,ymax=2,
			zmin=-1,zmax=2,
			domain=-10:10,
			axis equal,
			hide axis
			]
			\draw[->, line width=1.pt, color=black](0., 0., 0.)--(1.5,0.,0.);
			\draw[color=black] (1.6,0,0) node[] {$\xi$};
			\draw[->, line width=1.pt, color=black](0., 0., 0.)--(0.,1.5,0.);
			\draw[color=black] (0.,1.6,0) node[] {$\eta$};
			\draw[->, line width=1.pt, color=black](0., 0., 0.)--(0.,0.,1.5);
			\draw[color=black] (0.,0.,1.6) node[] {$\zeta$};
			\addplot3[color=asb][line width=0.6pt,mark=*]
			coordinates {(0,0,0)(1,0,0)(0,1,0)(0,0,0)};
			\addplot3[color=asb][line width=1.pt,mark=*]
			coordinates {(0,0,0)(0,0,1)};
			\addplot3[color=asb][line width=0.6pt]coordinates {(1,0,0)(0,0,1)};
			\addplot3[color=asb][line width=0.6pt]coordinates {(0,1,0)(0,0,1)};
			\draw[color=asb] (0.2,0.2,0.15) node[anchor=south east] {$\Xi$};
			\draw[color=asb] (0,0,0) node[anchor=south east] {$_{v_{1}}$};
			\draw[color=asb] (1,0,0) node[anchor=north east] {$_{v_{2}}$};
			\draw[color=asb] (0,1,0) node[anchor=south west] {$_{v_{3}}$};
			\draw[color=asb] (0,0,1) node[anchor=north east] {$_{v_{4}}$};
			\fill[opacity=0.1, asb] (axis cs: 0,0,0) -- (axis cs: 1,0,0) -- (axis cs: 0,1,0) -- (axis cs: 0,0,1) -- cycle;
			\draw[->, line width=1.pt, color=asl](0.8,0.2,0)--(0.2,0.8,0);
			\draw[color=asl] (0.55,0.5,0.) node[anchor=west] {$\bm{\varsigma}$};
			\draw[->, line width=1.pt, color=asl](0.33,0.33,0.33)--(0.5,0.5,0.5);
			\draw[color=asl] (0.5,0.5,0.5) node[anchor=west] {$\bm{\varrho}$};
			
			\addplot3[color=asb][line width=0.6pt,mark=*]
			coordinates {(2.6,2,0.5)(2.6,3,0.5)(3,2,0.2)(2,3,1.1)};
		    \addplot3[color=asb][line width=0.6pt]coordinates {(2.6,2,0.5)(3,2,0.2)};
		    \addplot3[color=asb][line width=0.6pt]coordinates {(2.6,2,0.5)(2,3,1.1)};
		    \addplot3[color=asb][line width=0.6pt]coordinates {(2.6,3,0.5)(2,3,1.1)};
			\fill[opacity=0.1, asb] (axis cs: 2.6,2,0.5) -- (axis cs: 3,2,0.2) -- (axis cs: 2.6,3,0.5) -- (axis cs: 2,3,1.1) -- cycle;
			\draw[color=asb] (axis cs: 2.25,3.25,0.5) node[anchor=east] {$\Omega_e$};
			\draw[color=asb] (axis cs: 2.6,2,0.5) node[anchor=east] {$\vb{x}_4$};
			\draw[color=asb] (axis cs: 2.6,3,0.5) node[anchor=west] {$\vb{x}_1$};
			\draw[color=asb] (axis cs: 3,2,0.2) node[anchor=east] {$\vb{x}_2$};
			\draw[color=asb] (axis cs: 2,3,1.1) node[anchor=east] {$\vb{x}_3$};
			\draw[->, line width=1.pt, color=black](1.5, 1.5, 0.)--(2.5, 1.5,0.);
			\draw[color=black] (2.6, 1.5,0.) node[] {$x$};
			\draw[->, line width=1.pt, color=black](1.5, 1.5, 0.)--(1.5, 2.5,0.);
			\draw[color=black] (1.5, 2.6,0.) node[] {$y$};
			\draw[->, line width=1.pt, color=black](1.5, 1.5, 0.)--(1.5, 1.5,1.);
			\draw[color=black] (1.5, 1.5,1.1) node[] {$z$};
			
			\draw[->, line width=1.pt, color=asl](2.8 , 2.2 , 0.38)--(2.2 , 2.8 , 0.92);
			\draw[color=asl] (2.16 , 2.76 , 0.5) node[anchor=west] {$\bm{\tau}$};
			\draw[->, line width=1.pt, color=asl](2.508, 2.31 , 0.594)--(2.3, 2. , 0.7);
			\draw[color=asl] (2.3, 2. , 0.7) node[anchor=east] {$\bm{\nu}$};
			
			\addplot3[smooth,color=black][->, line width=1.pt]coordinates {(0.7,0.7,1.2)(1.2,1.2,1.4)(1.7,1.7,1.2)};
			\draw[color=black] (1.2,1.2,1.4) node[anchor=south] {$\vb{x}(\xi, \, \eta , \, \zeta)$};
		\end{axis}
	\end{tikzpicture}
	\caption{Barycentric mapping from the reference element to the physical domain.}
	\label{fig:map}
\end{figure}
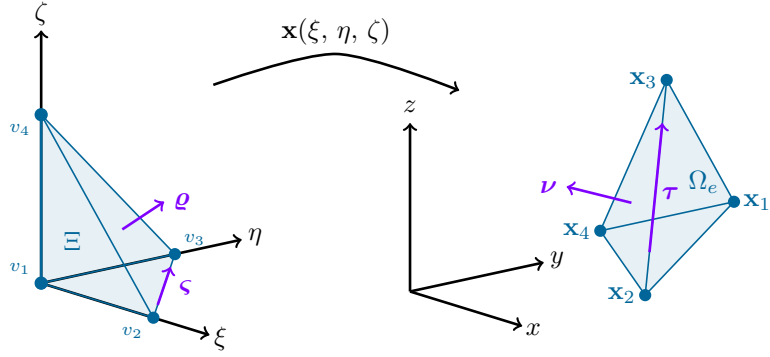

\subsection{Lagrangian base}
For the primal formulation we make use of Lagrangian base functions. These have the nodal degrees of freedom
\begin{align}
	l_{ij}(u) = \delta_{ij} \, u \at_{x_j} \, ,   
\end{align}
where $\delta_{ij}$ is the Kronecker delta.
Defining the reference element to be the unit tetrahedron
\begin{align}
    T = \{ \xi, \, \eta , \, \zeta \in [0, \, 1] \; | \; \xi + \eta + \zeta  \leq 1 \} \, ,
\end{align}
the linear Lagrangian element is given by the barycentric base functions \cref{eq:bary}.
For the quadratic polynomial space $\Po^2$, one finds the following ten base functions
\begin{subequations}
\begin{align}
	n_1(\xi,\eta,\zeta) &= 2\left(\eta + \xi + \zeta - 0.5\right) \left(\eta + \xi + \zeta - 1\right) \, , & n_2(\xi,\eta,\zeta) &= 2 \,\xi \left(\xi - 0.5\right) \, , \\
	n_3(\xi,\eta,\zeta) &= 2 \,\eta \left(\eta - 0.5\right) \, , & n_4(\xi,\eta,\zeta) &= 2\, \zeta \left(\zeta - 0.5\right) \, , \\[2ex] 
	n_5(\xi,\eta,\zeta) &=  4 \,\xi \left(1- \eta - \xi - \zeta \right) \, , & n_6(\xi,\eta,\zeta) &= 4 \,\eta \,\xi \, , \\
	n_7(\xi,\eta,\zeta) &= 4 \,\eta \left(1- \eta - \xi - \zeta \right) \, , & n_8(\xi,\eta,\zeta) &= 4\, \zeta \left(1- \eta - \xi - \zeta \right) \, , \\
	n_9(\xi,\eta,\zeta) &= 4 \,\xi \,\zeta \, , & n_{10}(\xi,\eta,\zeta) &= 4 \,\eta \,\zeta \, , 
\end{align}
\end{subequations}
where the first four are vertex base functions and the next six are edge base functions defined on the edge midpoint, see \cref{fig:lag}. 
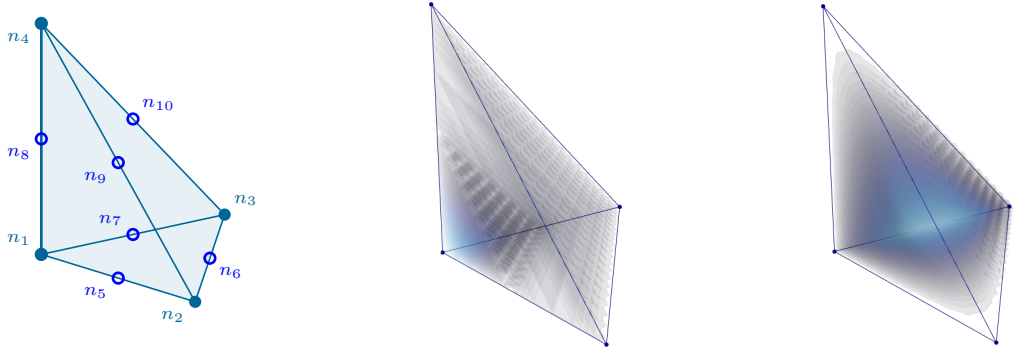
\begin{figure}
	\centering
	\begin{subfigure}{0.3\linewidth}
		\centering
		\definecolor{asl}{rgb}{0.4980392156862745,0.,1.}
		\definecolor{asb}{rgb}{0.,0.4,0.6}
		\begin{tikzpicture}
			\begin{axis}
				[
				width=30cm,height=20cm,
				view={50}{15},
				enlargelimits=true,
				xmin=-1,xmax=2,
				ymin=-1,ymax=2,
				zmin=-1,zmax=2,
				domain=-10:10,
				axis equal,
				hide axis
				]
				\addplot3[color=asb][
				line width=0.6pt,
				mark=*
				]
				coordinates {
					(0,0,0)(1,0,0)(0,1,0)(0,0,0)
				};
				\addplot3[color=asb][
				line width=1.pt,
				mark=*
				]
				coordinates {
					(0,0,0)(0,0,1)
				};
			    \addplot3[color=blue][line width=1.pt,mark=o]
			    coordinates {(0.5,0,0)};
			    \addplot3[color=blue][line width=1.pt,mark=o]
			    coordinates {(0,0.5,0)};
			    \addplot3[color=blue][line width=1.pt,mark=o]
			    coordinates {(0,0,0.5)};
			    \addplot3[color=blue][line width=1.pt,mark=o]
			    coordinates {(0.5,0.5,0)};
			    \addplot3[color=blue][line width=1.pt,mark=o]
			    coordinates {(0,0.5,0.5)};
			    \addplot3[color=blue][line width=1.pt,mark=o]
			    coordinates {(0.5,0,0.5)};
				\addplot3[color=asb][line width=0.6pt]coordinates {(1,0,0)(0,0,1)};
				\addplot3[color=asb][line width=0.6pt]coordinates {(0,1,0)(0,0,1)};
				\draw[color=asb] (0,0,0) node[anchor=south east] {$_{n_1}$};
				\draw[color=asb] (1,0,0) node[anchor=north east] {$_{n_2}$};
				\draw[color=asb] (0,1,0) node[anchor=south west] {$_{n_3}$};
				\draw[color=asb] (0,0,1) node[anchor=north east] {$_{n_4}$};
				\draw[color=blue] (0.5,0,0) node[anchor=north east] {$_{n_5}$};
				\draw[color=blue] (0.5,0.5,0) node[anchor=north west] {$_{n_6}$};
				\draw[color=blue] (0,0.5,0) node[anchor=south east] {$_{n_7}$};
				\draw[color=blue] (0,0,0.5) node[anchor=north east] {$_{n_8}$};
				\draw[color=blue] (0.5,0,0.5) node[anchor=north east] {$_{n_9}$};
				\draw[color=blue] (0,0.5,0.5) node[anchor=south west] {$_{n_{10}}$};
				\fill[opacity=0.1, asb] (axis cs: 0,0,0) -- (axis cs: 1,0,0) -- (axis cs: 0,1,0) -- (axis cs: 0,0,1) -- cycle;
			\end{axis}
		\end{tikzpicture}
	\end{subfigure}
	\begin{subfigure}{0.3\linewidth}
		\centering
		\includegraphics[width=0.5\linewidth]{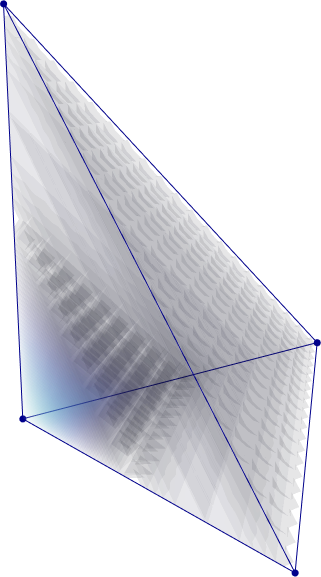}
	\end{subfigure}
	\begin{subfigure}{0.3\linewidth}
		\centering
		\includegraphics[width=0.5\linewidth]{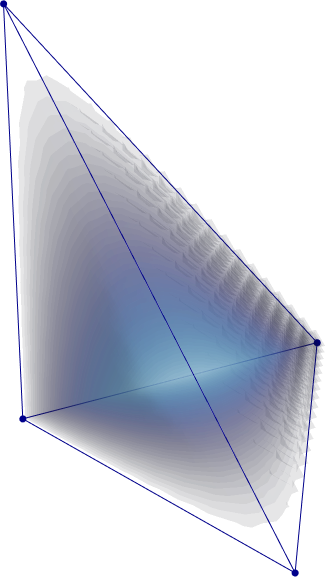}
	\end{subfigure}
	\caption{Vertex and edge nodes on the reference element,  vertex base function $n_1$ and edge base function $n_7$ of the Lagrangian basis.}
	\label{fig:lag}
\end{figure}
Gradients with respect to the physical space are given by the standard chain rule
\begin{align}
	\nabla_x n_i = \bm{J}^{-T} \nabla_\xi n_i \, .
\end{align}
The interpolation of the displacement field is generated using the $\bm{N}$-matrix and the union operator
\begin{align}
	\vb{u} &= \bigcup_{e=1}^n \bm{N} \vb{u}_e \, , 
	\\ \bm{N}_1 &= \begin{bmatrix}
		b_1 \one & b_2 \one & b_3 \one & b_4 \one 
	\end{bmatrix} \in \mathbb{R}^{3 \times 12} \, ,
	\\[1ex] \bm{N}_2 &= \begin{bmatrix}
		n_1 \one & n_2 \one & n_3 \one & n_4 \one & n_5 \one & n_6 \one & n_7 \one & n_8 \one & n_9 \one & n_{10} \one 
	\end{bmatrix} \in \mathbb{R}^{3 \times 30} \, ,
\end{align}
where $\one$ is the three-dimensional identity matrix.
Consequently, in each element, the displacement vector is interpolated by $12$ Lagrangian base functions in the linear case and $30$ in the quadratic case.

\subsection{N\'ed\'elec base}
For the interpolation of the microdistortion tensor $\Pm$ we make use of N\'ed\'elec base functions of the first and second types \cite{Nedelec1980,Ned2,Zaglmayr2006}. 
The edge degrees of freedom are defined by the functionals 
\begin{align}
	l_{ij} (\vb{p}) = \int_{s_{i}} q_j \langle \vb{p} , \, \bm{\tau} \rangle \, \dd s \qquad \forall \, q_j \in \left \{ \begin{matrix}
	    \Po^{p-1}(s_i) & \text{for} & \Ned_I  \\
	    \Po^p(s_i) & \text{for} & \Ned_{II}
	\end{matrix} \right . \, ,
	\label{eq:neddof}
\end{align}
where $s_i$ is the curve of an edge on the element and $q_j$ is a test function. In accordance with the de Rham complex, the appropriate polynomial space for the linear sequence is
\begin{align}
	&\mathit{R}^1 = [\Po^{0}]^3 \oplus \mathit{S}^1 \, , && \mathit{S}^1 = \left \{ \vb{p} \in \left[\widetilde{\Po}^1 \right ]^3 \; | \; \langle \vb{p} , \, \bm{\xi} \rangle = 0 \right \} \, , && \dim \mathit{R}^1 = 6 \, ,
\end{align} 
where $\widetilde{\Po}$ is the space of homogeneous polynomials. The corresponding test function $q_j = 1$ yields the base functions on the reference element (see \cref{fig:nedI})
\begin{figure}
	\centering
	\begin{subfigure}{0.15\linewidth}
		\includegraphics[width=1.0\linewidth]{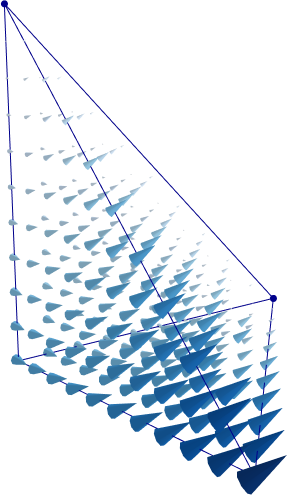}
		\caption{$\bm{\vartheta}_1$} 
	\end{subfigure}
	\begin{subfigure}{0.15\linewidth}
		\includegraphics[width=1.0\linewidth]{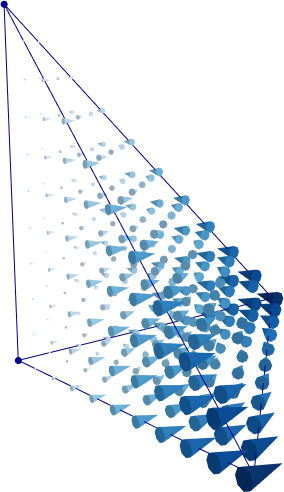}
		\caption{$\bm{\vartheta}_2$}
	\end{subfigure}
	\begin{subfigure}{0.15\linewidth}
		\includegraphics[width=1.0\linewidth]{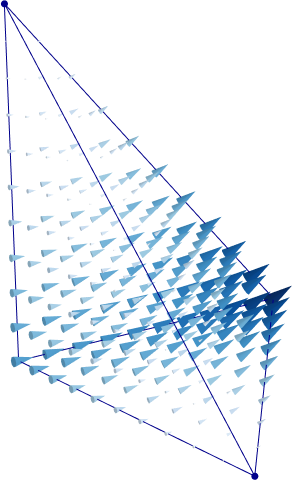}
		\caption{$\bm{\vartheta}_3$}
	\end{subfigure}
	\begin{subfigure}{0.15\linewidth}
		\includegraphics[width=1.0\linewidth]{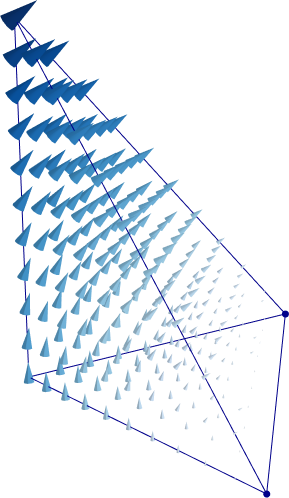}
		\caption{$\bm{\vartheta}_4$}
	\end{subfigure}
	\begin{subfigure}{0.15\linewidth}
		\includegraphics[width=1.0\linewidth]{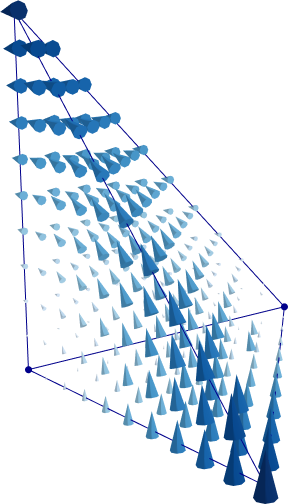}
		\caption{$\bm{\vartheta}_5$}
	\end{subfigure}
    \begin{subfigure}{0.15\linewidth}
    	\includegraphics[width=1.0\linewidth]{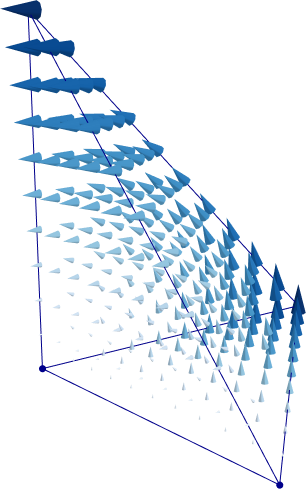}
    	\caption{$\bm{\vartheta}_6$}
    \end{subfigure}
	\caption{Lowest order N\'ed\'edelec base functions of the first type on the reference tetrahedron.}
	\label{fig:nedI}
\end{figure}
\begin{align}
	&\bm{\vartheta}_1 = \begin{bmatrix}
		1 - \eta - \zeta \\
		\xi \\
		\xi
	\end{bmatrix} \, , &&
	\bm{\vartheta}_2 = \begin{bmatrix}
		- \eta \\
		\xi \\
		0
	\end{bmatrix} \, , &&
	\bm{\vartheta}_3 = \begin{bmatrix}
		\eta \\
		1-\xi-\zeta \\
		\eta
	\end{bmatrix} \, , \notag \\
	&\bm{\vartheta}_4 = \begin{bmatrix}
		\zeta \\
		\zeta \\
		1 - \xi - \eta
	\end{bmatrix} \, , &&
	\bm{\vartheta}_5 = \begin{bmatrix}
		- \zeta \\
		0 \\
		\xi
	\end{bmatrix} \, , &&
	\bm{\vartheta}_6 = \begin{bmatrix}
		0 \\
		-\zeta \\
		\eta
	\end{bmatrix} \, ,
\end{align}
representing the base of the lowest order N\'ed\'elec element of the first type $\Ned_I^0$.
For the sequence of the quadratic element we employ the linear N\'ed\'elec elements of the second type $\Ned_{II}^1$. The corresponding polynomial space reads
\begin{align}
	&[\Po^1]^3 = [\spa\{ 1, \,\xi, \,\eta, \,\zeta \}]^3 \, , && \dim [\Po^1]^3 = 12 \, .
\end{align}
In other words, the space is given by $12$ N\'ed\'elec functions.
Using the ansatz
\begin{align}
	\bm{\vartheta} = \begin{bmatrix}
		a_0 + a_1 \xi + a_2 \eta + a_3 \zeta \\
		b_0 + b_1 \xi + b_2 \eta + b_3 \zeta \\
		c_0 + c_1 \xi + c_2 \eta + c_3 \zeta 
	\end{bmatrix} \in [\Po^1]^3 \, ,
\end{align}
and the degrees of freedom from \cref{eq:neddof} for each edge on the reference element
\begin{align}
	l_{ij} (\bm{\vartheta}) &= \int_{\mu_{i}} q_j \langle \bm{\vartheta} , \, \bm{\varsigma} \rangle \, \dd \mu = \delta_{ij} \, , && \, q_1(\mu_i) = \dfrac{d-\mu_i}{d} \, , && q_2(\mu_i) = \mu_i \, , && d = \int_{\mu_{i}} \dd \mu \, ,
\end{align}
where $\mu_i$ is the parameter for the corresponding edge,
we find the base functions (see \cref{fig:nedfns})
\begin{subequations}
\begin{align}
	\bm{\vartheta}_1 &= \left[\begin{matrix}1 - \xi - \eta  - \zeta \\0\\0\end{matrix}\right] \, , & \bm{\vartheta}_2 &= \left[\begin{matrix}\xi\\\xi\\\xi\end{matrix}\right] \, , &
	\bm{\vartheta}_3 &=\left[\begin{matrix}0\\\xi\\0\end{matrix}\right] \, , \\
	\bm{\vartheta}_4 &= \left[\begin{matrix}- \eta\\0\\0\end{matrix}\right] \, , &
	\bm{\vartheta}_5 &=\left[\begin{matrix}0\\1- \xi - \eta  - \zeta\\0\end{matrix}\right] \, , &\bm{\vartheta}_6 &=
	\left[\begin{matrix}\eta\\\eta\\\eta\end{matrix}\right] \, , \\
	\bm{\vartheta}_7  &= \left[\begin{matrix}0\\0\\1- \xi - \eta  - \zeta\end{matrix}\right] \, , &\bm{\vartheta}_8 &=
	\left[\begin{matrix}\zeta\\\zeta\\\zeta\end{matrix}\right] \, , &\bm{\vartheta}_9 &=
	\left[\begin{matrix}0\\0\\\xi\end{matrix}\right] \, , \\
	\bm{\vartheta}_{10} &= \left[\begin{matrix}- \zeta\\0\\0\end{matrix}\right] \, , &\bm{\vartheta}_{11} &=
	\left[\begin{matrix}0\\0\\\eta\end{matrix}\right] \, , &\bm{\vartheta}_{12} &=
	\left[\begin{matrix}0\\- \zeta\\0\end{matrix}\right] \, .	
\end{align}
\end{subequations}
\begin{figure}
	\centering
	\begin{subfigure}{0.15\linewidth}
		\centering
		\includegraphics[width=1.0\linewidth]{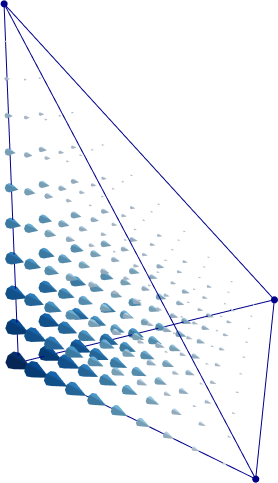}
		\caption{$\bm{\vartheta}_{1}$}
	\end{subfigure}
	\begin{subfigure}{0.15\linewidth}
		\centering
	\includegraphics[width=1.0\linewidth]{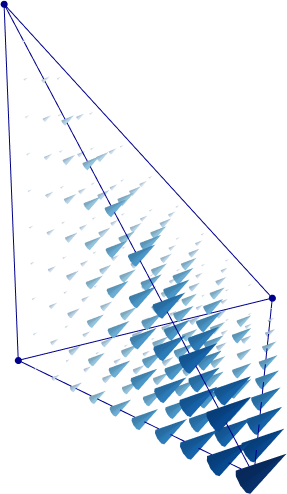}
	\caption{$\bm{\vartheta}_{2}$}
\end{subfigure}
	\begin{subfigure}{0.15\linewidth}
		\centering
	\includegraphics[width=1.0\linewidth]{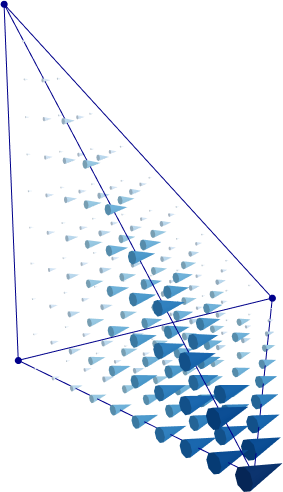}
	\caption{$\bm{\vartheta}_{3}$}
\end{subfigure}
	\begin{subfigure}{0.15\linewidth}
		\centering
	\includegraphics[width=1.0\linewidth]{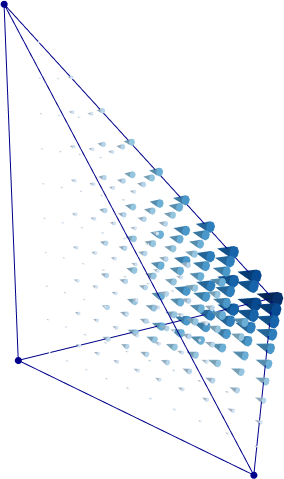}
	\caption{$\bm{\vartheta}_{4}$}
\end{subfigure}
	\begin{subfigure}{0.15\linewidth}
		\centering
	\includegraphics[width=1.0\linewidth]{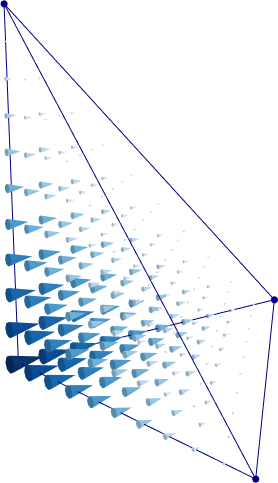}
	\caption{$\bm{\vartheta}_{5}$}
\end{subfigure}
	\begin{subfigure}{0.15\linewidth}
		\centering
	\includegraphics[width=1.0\linewidth]{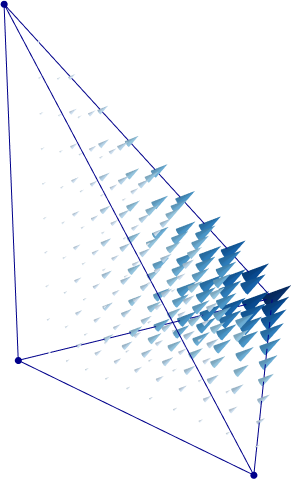}
	\caption{$\bm{\vartheta}_{6}$}
\end{subfigure}
	\begin{subfigure}{0.15\linewidth}
		\centering
	\includegraphics[width=1.0\linewidth]{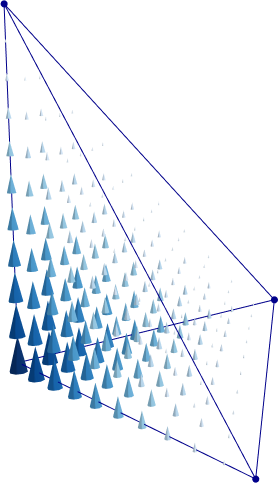}
	\caption{$\bm{\vartheta}_{7}$}
\end{subfigure}
	\begin{subfigure}{0.15\linewidth}
		\centering
	\includegraphics[width=1.0\linewidth]{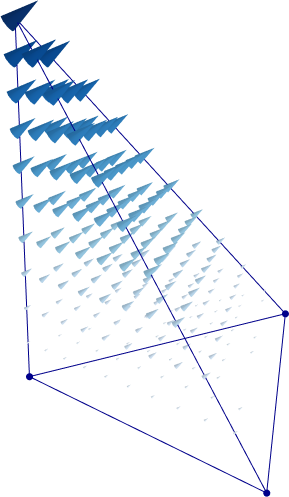}
	\caption{$\bm{\vartheta}_{8}$}
\end{subfigure}
	\begin{subfigure}{0.15\linewidth}
		\centering
	\includegraphics[width=1.0\linewidth]{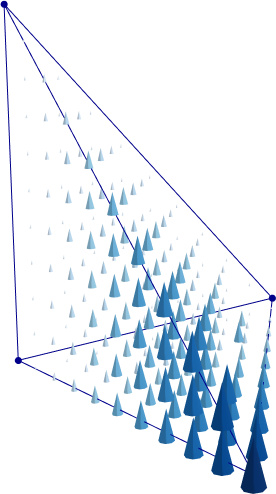}
	\caption{$\bm{\vartheta}_{9}$}
\end{subfigure}
	\begin{subfigure}{0.15\linewidth}
		\centering
	\includegraphics[width=1.0\linewidth]{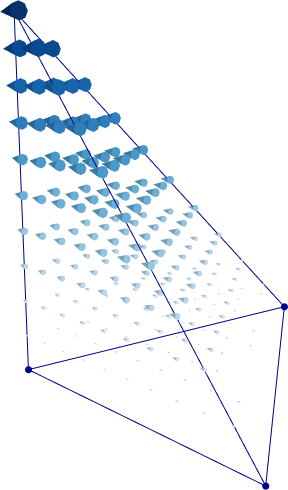}
	\caption{$\bm{\vartheta}_{10}$}
\end{subfigure}
	\begin{subfigure}{0.15\linewidth}
		\centering
	\includegraphics[width=1.0\linewidth]{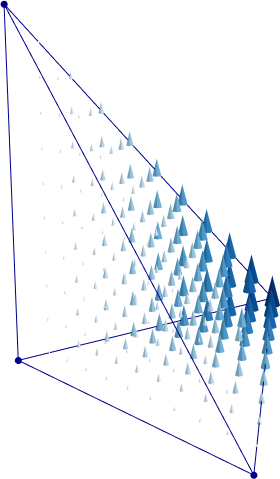}
	\caption{$\bm{\vartheta}_{11}$}
\end{subfigure}
	\begin{subfigure}{0.15\linewidth}
		\centering
	\includegraphics[width=1.0\linewidth]{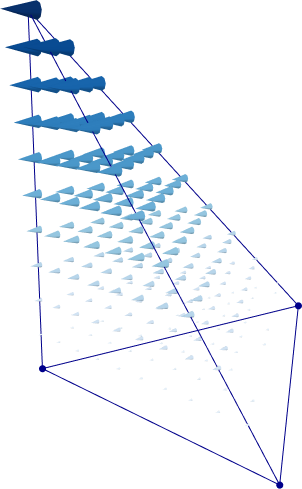}
	\caption{$\bm{\vartheta}_{12}$}
\end{subfigure}
	\caption{N\'ed\'elec $\Ned_{II}^1$-base functions on the reference element.}
	\label{fig:nedfns}
\end{figure}
\begin{remark}
    The resulting N\'ed\'elec base functions are not unique in the sense that they are determined by the choice of the test functions $q_j$. In our case we chose the linear Lagrangian base on $\mu$, leading to a N\'ed\'elec-Lagrangian base.
\end{remark}
\begin{remark}
    Note that the decrease in polynomial order of the second discrete sequence (see \cref{fig:rham}) can be alleviated by employing first order N\'ed\'elec elements of the first type $\Ned_{I}^1$. This increases the interpolation power of the solenoidal part of the microdistortion field $\Pm$, but not of the irrotational part. 
\end{remark}
The base functions are defined on the reference element. In order to preserve their tangential projection on edges in the physical domain
\begin{align}
	&\int_{s_{i}} q_j \langle \bm{\theta} , \, \bm{\tau} \rangle \, \dd s = \int_{\mu_{i}} \hat{q}_j \langle \bm{\vartheta} , \, \bm{\varsigma} \rangle \, \dd \mu \, , && \hat{q}_j = q_j \circ \vb{x} \, ,  
\end{align} 
we make use of the transformation of curves via the Jacobi matrix (see \cref{eq:jacobi})
\begin{align}
	\dd \vb{s} = \bm{\tau} \, \dd s = \bm{J} \, \dd \bm{\mu} = \bm{J} \, \bm{\varsigma} \, \dd \mu \, ,
	\label{eq:tang}
\end{align}
to find the covariant Piola transformation \cite{Mon03}
\begin{align}
	\langle \bm{\theta} , \, \dd \vb{s} \rangle  = \langle \bm{\theta} , \, \bm{J} \, \dd \bm{\mu} \rangle =  \langle \bm{\vartheta} , \, \dd \bm{\mu} \rangle \quad  \Longleftrightarrow \quad \bm{\theta} = \bm{J}^{-T} \bm{\vartheta} \, .
	\label{eq:copiola}
\end{align}
Further, for the transformation of the curl we find
\begin{align}
	\curl_x \bm{\theta} = \nabla_x \times \bm{\theta} = (\bm{J}^{-T} \nabla_\xi) \times (\bm{J}^{-T} \bm{\vartheta}) = \cof(\bm{J}^{-T})\curl_\xi \bm{\vartheta} = \dfrac{1}{\det\bm{J}} \bm{J} \, \curl_\xi \bm{\vartheta} \, ,
	\label{eq:conta}
\end{align}
which is the so called contravariant Piola transformation. The formula is derived using the identity\footnote{$\nabla_x \times \bm{J}^{-T} = (\mathrm{Anti}\nabla_x) (\D_x \bm{\xi})^T = 0$, where $\mathrm{Anti}(\cdot)$ maps the vector to its corresponding anti-symmetric matrix.} ${\nabla_x \times \bm{J}^{-T} = 0}$. 

The microdistortion can now be interpolated by
\begin{align}
    \Pm &= \bigcup_{e=1}^n \bm{\Theta} \Pm_e \, , \\ 
    \bm{\Theta}_I &= \begin{bmatrix}
        \bm{\theta}_1 & \vb{o} & \vb{o} & \bm{\theta}_2 & \vb{o} & \vb{o} & & \bm{\theta}_{6} & \vb{o} & \vb{o} \\
         \vb{o} & \bm{\theta}_1 & \vb{o} & \vb{o} & \bm{\theta}_2 & \vb{o} & \cdots & \vb{o} & \bm{\theta}_{6} & \vb{o}  \\
        \vb{o} & \vb{o} & \bm{\theta}_1 & \vb{o} & \vb{o} & \bm{\theta}_2 & & \vb{o} & \vb{o} & \bm{\theta}_{6} 
    \end{bmatrix} \in \mathbb{R}^{9 \times 18} \, , \\
    \bm{\Theta}_{II} &= \begin{bmatrix}
        \bm{\theta}_1 & \vb{o} & \vb{o} & \bm{\theta}_2 & \vb{o} & \vb{o} & & \bm{\theta}_{12} & \vb{o} & \vb{o} \\
         \vb{o} & \bm{\theta}_1 & \vb{o} & \vb{o} & \bm{\theta}_2 & \vb{o} & \cdots & \vb{o} & \bm{\theta}_{12} & \vb{o}  \\
        \vb{o} & \vb{o} & \bm{\theta}_1 & \vb{o} & \vb{o} & \bm{\theta}_2 & & \vb{o} & \vb{o} & \bm{\theta}_{12} 
    \end{bmatrix} \in \mathbb{R}^{9 \times 36} \, .
\end{align}
In other words, the microdistortion is interpolated by $18$ N\'ed\'elec base functions in the linear case and $36$ in the quadratic case on each element.

\subsection{Raviart-Thomas base}\label{ch:5}
For the mixed formulation we also require Raviart-Thomas \cite{Zaglmayr2006,Raviart} elements.
The lowest order face degrees of freedom of Raviart-Thomas elements are defined as
\begin{align}
	l_i(\vb{p}) = \int_{A_i} \langle \vb{p} , \, \bm{\nu} \rangle \, \dd \Gamma \, .
	\label{eq:dofrt}
\end{align}
The polynomial space for the construction of base functions is given by
\begin{align}
	\RT^p = [\Po^p]^3 \oplus \bm{\xi} \, \widetilde{\Po}^p \, ,
\end{align}
where $\widetilde{\Po}$ is the space of homogeneous polynomials. In the lowest order one finds
\begin{align}
	&\RT^0 = \mathbb{R}^3 \oplus \bm{\xi} \, \mathbb{R} = \spa \left\{ \begin{bmatrix}
		1 \\ 0 \\ 0
	\end{bmatrix} , \,
    \begin{bmatrix}
    	0 \\ 1 \\ 0
    \end{bmatrix} , \, 
    \begin{bmatrix}
    	0 \\ 0 \\ 1
    \end{bmatrix} , \,
    \begin{bmatrix}
    	\xi \\ \eta \\ \zeta
    \end{bmatrix} \right\} \, , && \dim \RT^0 = 4 \, .
\end{align}
Using the ansatz 
\begin{align}
	\bm{\phi} = \begin{bmatrix}
		a_0 + a_1 \, \xi \\
		b_0 + b_1 \, \eta \\
		c_0 + c_1 \, \zeta \\
	\end{bmatrix} \in \RT^0 \, ,
\end{align}
and the degrees of freedom from \cref{eq:dofrt} for each face on the reference tetrahedron
\begin{align}
	l_i(\bm{\phi}_j) = \int_{\Gamma_i} \langle \bm{\phi}_j , \, \bm{\varrho} \rangle \, \dd \Lambda = \delta_{ij} \, ,
\end{align}
where $\Lambda$ denotes a face on the reference tetrahedron, we find the base functions (see \cref{fig:rtfns})
\begin{align}
	&\bm{\phi}_1 = \begin{bmatrix}
		-\xi \\ -\eta \\ 1 - \zeta
	\end{bmatrix} \, , &&
    \bm{\phi}_2 = \begin{bmatrix}
    	\xi \\ \eta - 1 \\ \zeta 
    \end{bmatrix} \, , &&
    \bm{\phi}_3 = \begin{bmatrix}
    	\xi \\ \eta \\ \zeta 
    \end{bmatrix} \, , &&
    \bm{\phi}_4 = \begin{bmatrix}
    	1 - \xi \\ -\eta \\ -\zeta 
    \end{bmatrix} \, . 
\end{align}
\begin{figure}
	\centering
	\begin{subfigure}{0.24\linewidth}
		\centering
		\includegraphics[width=0.65\linewidth]{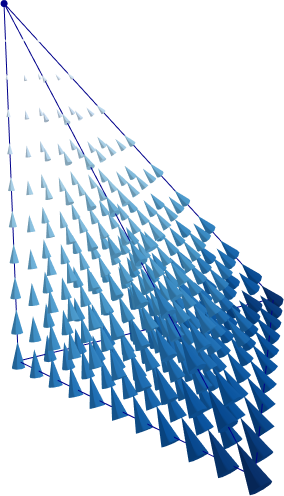}
		\caption{$\bm{\phi}_{1}$}
	\end{subfigure}
	\begin{subfigure}{0.24\linewidth}
		\centering
		\includegraphics[width=0.65\linewidth]{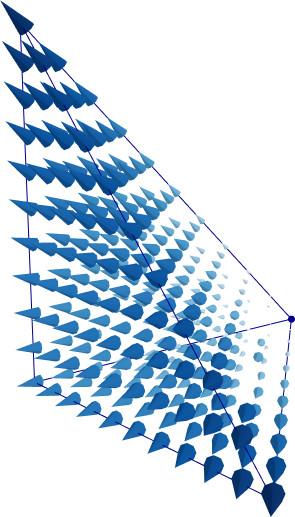}
		\caption{$\bm{\phi}_{2}$}
	\end{subfigure}
	\begin{subfigure}{0.24\linewidth}
		\centering
		\includegraphics[width=0.65\linewidth]{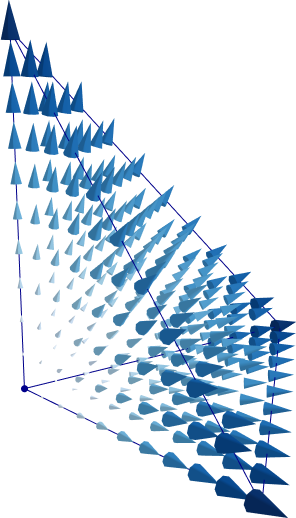}
		\caption{$\bm{\phi}_{3}$}
	\end{subfigure}
	\begin{subfigure}{0.24\linewidth}
		\centering
		\includegraphics[width=0.65\linewidth]{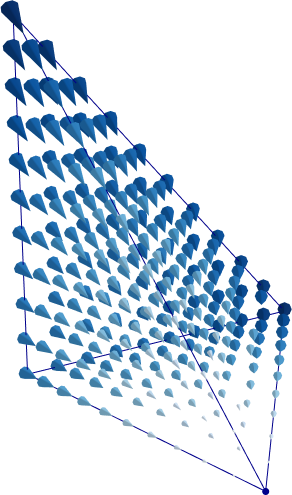}
		\caption{$\bm{\phi}_{4}$}
	\end{subfigure}
	\caption{Raviart-Thomas $\RT^0$-base functions on the reference element.}
	\label{fig:rtfns}
\end{figure}
Raviart-Thomas base functions are defined using the normal projections on the element's faces. In order to preserve the normal projection
\begin{align}
	\int_{A_i} \langle \bm{\varphi}_j , \, \bm{\nu} \rangle \, \dd \Gamma  = \int_{\Gamma_i} \langle \bm{\phi}_j , \, \bm{\varrho} \rangle \, \dd \Lambda \, ,
\end{align}
we consider the transformation of surfaces
\begin{align}
	\dd \bm{\Gamma} = \bm{\nu} \, \dd \Gamma = (\cof\bm{J}) \bm{\varrho} \, \dd \Lambda = (\cof\bm{J}) \dd \bm{\Lambda} \, .
\end{align}
Consequently, the base functions are mapped using the contravariant Piola transformation
\begin{align}
	\langle \bm{\varphi}_j , \, \dd \bm{\Gamma} \rangle = \langle \bm{\varphi}_j , \, (\cof\bm{J}) \, \dd \bm{\Lambda}\rangle = \langle \bm{\phi}_j , \, \dd \bm{\Lambda} \rangle
	\quad  \Longleftrightarrow \quad \bm{\varphi}_j  = \dfrac{1}{\det \bm{J}} \bm{J} \, \bm{\phi}_j \, 
\end{align}
\begin{remark}
    Note that this the same transformation as in \cref{eq:conta} since $\curl \bm{\theta} \in \Hd{} \, , \; \bm{\theta} \in \Hc{}$.
\end{remark}
Considering the divergence of functions undergoing a contravariant Piola transformation we observe
\begin{align}
	\int_\Omega q \, \di_x\bm{\varphi} \, \dd  X &= 
	\oint_{\partial \Omega} q \, \langle \bm{\varphi} , \, \bm{\nu} \rangle \, \dd \Gamma - \int_\Omega \langle \D_x q , \, \bm{\varphi} \rangle \, \dd X \notag \\
	&= \oint_{\partial \Xi} \hat{q} \, \langle \dfrac{1}{\det \bm{J}} \bm{J} \, \bm{\phi} , \, \det(\bm{J}) \, \bm{J}^{-T} \bm{\varrho} \rangle \, \dd \Lambda - \int_\Xi \langle \bm{J}^{-T} \D_\xi \hat{q} , \, \dfrac{1}{\det \bm{J}} \bm{J} \, \bm{\phi} \rangle \, \det \bm{J} \, \dd \Xi \notag \\
	&= \oint_{\partial \Xi} \hat{q} \, \langle \bm{\phi} , \, \bm{\varrho} \rangle \, \dd \Lambda - \int_\Xi \langle \D_\xi \hat{q} , \, \bm{\phi} \rangle \, \dd \Xi \notag \\
	&= \int_\Xi \hat{q} \, \di_\xi \bm{\phi} \, \dd \Xi = \int_\Omega q \, \di_\xi(\bm{\phi}) \, \dfrac{1}{\det \bm{J}} \dd X \qquad \forall\, q \, \in \mathit{C}^\infty(\Omega) \, ,
\end{align}
where $\hat{q} = q \circ \vb{x}$.
As a result, one finds
\begin{align}
	\di_x\bm{\varphi}= 
	\dfrac{1}{\det \bm{J}}\, \di_\xi \bm{\phi} \, .
\end{align}
Thus, the hyperstress field is interpolated by
\begin{align}
    &\bm{D} = \bigcup_{e=1}^n \bm{\Phi} \bm{D}_e \, , && \bm{\Phi} = \begin{bmatrix}
        \bm{\varphi}_1 & \vb{o} & \vb{o} & \bm{\varphi}_2 & \vb{o} & \vb{o} & & \bm{\varphi}_{4} & \vb{o} & \vb{o} \\
         \vb{o} & \bm{\varphi}_1 & \vb{o} & \vb{o} & \bm{\varphi}_2 & \vb{o} & \cdots & \vb{o} & \bm{\varphi}_{4} & \vb{o}  \\
        \vb{o} & \vb{o} & \bm{\varphi}_1 & \vb{o} & \vb{o} & \bm{\varphi}_2 & & \vb{o} & \vb{o} & \bm{\varphi}_{4} 
    \end{bmatrix} \in \mathbb{R}^{9 \times 12} \, ,
\end{align}
such that, on each element $12$ Raviart-Thomas base functions define the hyperstress.  
\begin{remark}
    An increase of the N\'ed\'elec element in the second discrete sequence to $\Ned_I^1$ would allow to employ either first order Raviart-Thomas elements $\RT^1$ or linear Brezzi-Douglas-Marini elements $\BDM^1$ \cite{BDM85}. 
\end{remark}

\subsection{Discontinuous basis}\label{ch:dis}
Finally, for the discontinuous elements in $\Le$ we employ piece-wise constants 
\begin{align}
    &q_i = 1 \, , && q_i \in \Po^0 \, .
\end{align}
The entire space has the dimension $\dim [\Po^0]^3 = 3$ and the interpolation reads
\begin{align}
    &\vb{q} = \bigcup_{e = 1}^n \one \vb{q}_e \, ,
\end{align}
resulting in three discontinuous base functions on each element.

\subsection{The orientation problem}
The co- and contravariant Piola transformations do not suffice to assert the consistent orientation of the tangential or normal projections of the N\'ed\'elec and Raviart-Thomas base functions, respectively. The transformations control the size of the projections, but not whether these are parallel or anti-parallel with respect to neighbouring elements. 
Consistent projections is a key requirement in ensuring no jumps occur in the trace of the respective space and as such, there exist various methods for dealing with this so called orientation problem \cite{Zaglmayr2006,Ainsworth2003,Anjam2015,Sky2021}. In this work we present a solution based on the sequencing of vertices and the separation of orientational data.
We define the following rule for the orientation of edges
\begin{align}
	e = \{v_i, \, v_j \} \qquad \text{s.t.} \quad i < j \, .
\end{align}
This means each edge starts at the lower vertex index and ends at the higher vertex index. 
This definition determines the orientation of the edge tangent vector (see \cref{fig:map}) and consequently, the tangential projection of the N\'ed\'elec base functions. 
Analogously, for surfaces we define
\begin{align}
	f = \{v_i, \, v_j , \, v_k\}  \qquad \text{s.t.} \quad i < j < k  \, ,
\end{align}
such that each surface is given by a sequence of increasing vertex indices. The orientation of the surface normal is given according to the left-hand rule. In other words, the direction of the normal is determined by the cross product of the vectors arising from the edges $\{v_i , \, v_j\}$ and $\{v_i, \, v_k\}$
\begin{align}
	\bm{\nu}_{ijk} \parallel \bm{\tau}_{ij} \times \bm{\tau}_{ik} \, . 
\end{align}
The orientation of the N\'ed\'elec- and Raviart-Thomas base functions is according to these rules. Consequently, in order to map each tetrahedron in the mesh to this orientation, we define each element as an increasing vertex-index sequence (see \cref{fig:consist})
\begin{align}
	T = \{ v_i , \, v_j , \, v_k , \, v_l \} \qquad \text{s.t.} \quad i < j < k < l \, .
\end{align}
The latter ensures the consistent orientation of the base functions, since they are all mapped from the same reference domain. However, integration in the reference element is determined by the determinant of the Jacobi matrix
\begin{align}
	\int_{\Omega_e} \dd X = \int_{\Xi} \det \bm{J} \, \dd \Xi \, ,
\end{align}
which may be negative due to a reflection of the element in the mapping from the reference to the physical domain. We correct for the error by taking only the absolute value of the determinant
\begin{align}
	\int_{\Omega_e} \dd X = \int_{\Xi} |\det \bm{J}| \, \dd \Xi \, .
\end{align}
Consequently, consistency is guaranteed by mapping from a single reference element and the use of correction functions or considerations of neighbouring elements are circumvented. 
\begin{remark}
	The absolute value of $\det \bm{J}$ is only used for the integration over the element. In all other use-cases, the information of the sign is necessary.
\end{remark}
   
\begin{figure}
	\centering
	\definecolor{asl}{rgb}{0.4980392156862745,0.,1.}
	\definecolor{asb}{rgb}{0.,0.4,0.6}
	\begin{subfigure}{0.45\linewidth}
		\centering
		\begin{tikzpicture}
			\begin{axis}
				[
				width=30cm,height=30cm,
				view={20}{5},
				enlargelimits=true,
				xmin=-1,xmax=2,
				ymin=-1,ymax=2,
				zmin=-1,zmax=2,
				domain=-10:10,
				axis equal,
				hide axis
				]
				\addplot3[color=asb][line width=0.6pt,mark=*]
				coordinates {(0.5,0,0)(0,0.5,0)};
				\addplot3[color=asb][line width=0.6pt, densely dashed,mark=*]
				coordinates {(0,0.5,0)(0.25,0.25,0.5)};
				\addplot3[color=asb][line width=0.6pt,mark=*]
				coordinates {(0.25,0.25,0.5)(0.5,0,0)};
				\addplot3[color=asb][line width=0.6pt,mark=*]
				coordinates {(0.5,0,0)(0,0,0.25)};
				\addplot3[color=asb][line width=0.6pt,mark=*]
				coordinates {(0.25,0.25,0.5)(0,0,0.25)};
				\addplot3[color=asb][line width=0.6pt,mark=*]
				coordinates {(0,0.5,0)(0,0,0.25)};
				\fill[opacity=0.1, asb] (axis cs: 0.5,0,0) -- (axis cs: 0,0.5,0) -- (axis cs: 0,0,0.25) -- (axis cs: 0.25,0.25,0.5) -- cycle;
				\draw[color=asb] (0.5,0,0) node[anchor=north east] {$_{v_{1}}$};
				\draw[color=asb] (0,0.5,0) node[anchor=north east] {$_{v_{3}}$};
				\draw[color=asb] (0,0,0.25) node[anchor=north east] {$_{v_{4}}$};
				\draw[color=asb] (0.25,0.25,0.5) node[anchor=south east] {$_{v_{2}}$};
				\draw[->, line width=1.pt, color=asl](0.35,0.15,0)--(0.15,0.35,0);
				\draw[->, line width=1.pt, color=asl, dashed](0.175,0.325,0.35)--(0.075,0.425,0.15);
				\draw[->, line width=1.pt, color=asl](0.425,0.075,0.15)--(0.325,0.175,0.35);
				\draw[->, line width=1.pt, color=asl, dashed](0.25,0.25,1/6)--(0.15,7/30-0.1,1/6);
				
				\addplot3[color=asb][line width=0.6pt,mark=*]
				coordinates {(0.8,0.1,0)(0.3,0.6,0)(0.55,0.35,0.5)(0.8,0.1,0)};	
				\addplot3[color=asb][line width=0.6pt,mark=*]
				coordinates {(0.8,0.1,0)(1.0,0.3,0.25)};
				\addplot3[color=asb][line width=0.6pt,densely dashed,mark=*]
				coordinates {(0.3,0.6,0)(1.0,0.3,0.25)};
				\addplot3[color=asb][line width=0.6pt,mark=*]
				coordinates {(0.55,0.35,0.5)(1.0,0.3,0.25)};
				\addplot3[color=black][line width=0.6pt,dotted]
				coordinates {(0.5,0,0)(0.8,0.1,0)};
				\addplot3[color=black][line width=0.6pt,dotted]
				coordinates {(0,0.5,0)(0.3,0.6,0)};
				\addplot3[color=black][line width=0.6pt,dotted]
				coordinates {(0.25,0.25,0.5)(0.55,0.35,0.5)};
				\fill[opacity=0.1, asb] (axis cs: 0.3,0.6,0) -- (axis cs: 0.8,0.1,0) -- (axis cs: 1.0,0.3,0.25) -- (axis cs: 0.55,0.35,0.5) -- cycle;
				\draw[color=asb] (0.29,0.6,0.02) node[anchor=south] {$_{v_{3}}$};
				\draw[color=asb] (0.8,0.1,0) node[anchor=north east] {$_{v_{1}}$};
				\draw[color=asb] (1.0,0.3,0.25) node[anchor=north west] {$_{v_{5}}$};
				\draw[color=asb] (0.55,0.35,0.5) node[anchor=south east] {$_{v_{2}}$};
				\draw[->, line width=1.pt, color=asl](0.65,0.25,0)--(0.45,0.45,0);
				\draw[->, line width=1.pt, color=asl, ](0.475,0.425,0.35)--(0.375,0.525,0.15);
				\draw[->, line width=1.pt, color=asl](0.725,0.175,0.15)--(0.625,0.275,0.35);
				\draw[->, line width=1.pt, color=asl](0.55,0.35,1/6)--(0.45,7/30,1/6);
			\end{axis}
		\end{tikzpicture}
	\end{subfigure}
    \begin{subfigure}{0.45\linewidth}
    	\centering
    	\textcolor{asb}{\begin{align}
    			&T_1 = \{v_1, \,v_2, \,v_3, \,v_4\} \notag \\[2ex]
    			&T_2 = \{v_1, \,v_2, \,v_3, \,v_5\} \notag 
    	\end{align}}
    \end{subfigure}
	\caption{Consistent orientations using vertex sequences. Edges are oriented from the lower to the higher vertex and faces according to the left-hand rule starting from the lowest vertex across the middle to the highest.}
	\label{fig:consist}
\end{figure}
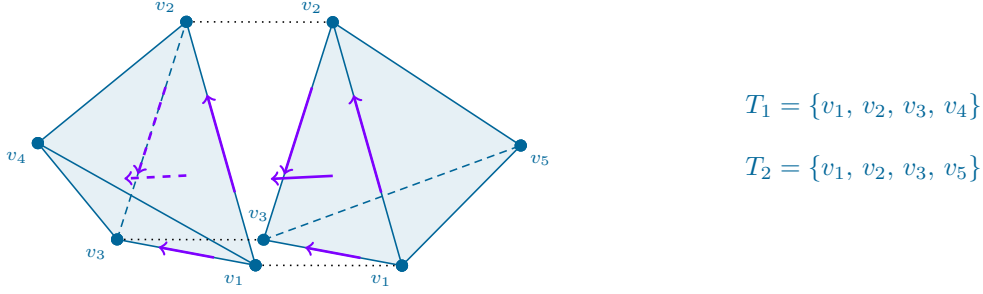

\subsection{The discrete consistent coupling condition}\label{ch:consist}
In order to exactly satisfy the consistent coupling, one may use the degrees of freedom from \cite{Demkowicz2000} to set the boundary conditions, ensuring a commuting projection. However, this requires the solution of a Dirichlet type problem for every edge on $\Gamma_D$. Alternatively, one can first find the Lagrangian interpolation of the boundary condition of the displacement on each edge and derive the boundary condition for the microdistortion directly from it.
We observe that on each element's edge, the displacement is given by
\begin{align}
    \Pi_g^1 u(\mu)\at_s &= \widetilde{u} \at_{v_1} n_1(\mu) + \widetilde{u} \at_{v_2} n_2(\mu) = \widetilde{u} \at_{v_1}(1-\mu) + \widetilde{u} \at_{v_2}\mu  \, , \\[1ex]
    \Pi_g^2 u(\mu)\at_s &= \widetilde{u} \at_{v_1} n_1(\mu) + \widetilde{u} \at_{m}  n_2(\mu) + \widetilde{u} \at_{v_2} n_3(\mu) = \widetilde{u} \at_{v_1}(2 \mu -1)(\mu-1) + \widetilde{u} \at_{m} 4 \mu \, (1 - \mu) + \widetilde{u} \at_{v_2}\mu \, (2  \mu - 1) \, , \notag
\end{align}
where $n_i$ are the Lagrangian base functions on the reference edge and $m$ is the mid-point of the edge.
The consistent coupling condition on a curve transforms according to the covariant Piola transformation
\begin{align}
    \langle \D_x \Pi_g u , \, \bm{\tau} \rangle = \langle \bm{J}^{-T} \D_\xi \Pi_g u , \, \bm{J} \, \bm{\varsigma} \rangle = \langle \D_\xi \Pi_g  u , \, \bm{\varsigma} \rangle = 
    \langle \Pi_c \vb{p} , \, \bm{\varsigma} \rangle \, ,
\end{align}
where $\vb{p}$ represents a row of the microdistortion $\Pm$ and $\langle \Pi_c \vb{p} , \, \bm{\varsigma} \rangle$ results from the transformations of the tangent vector \cref{eq:tang} and the N\'ed\'elec base functions \cref{eq:copiola}.
The gradient of the Lagrangian interpolation of the displacement $\Pi_g u$ can be reformulated in dependence of the curve parameter $\mu \in [0, \, 1]$
\begin{align}
    \langle \D_\xi \Pi_g u ,\, \bm{\varsigma} \rangle = (\Pi_g u)_{,\mu} \, \langle \D_\xi \mu ,\, \bm{\varsigma} \rangle \, . 
\end{align}
We note that in order to maintain an invariant transformation, the tangent vectors on the reference domain are given by
\begin{align}
    \bm{\varsigma} = \bm{\xi}_{,\mu} = \dfrac{\dd}{\dd \mu} [(1 - \mu) \, \bm{\xi}_{v_1} + \mu \, \bm{\xi}_{v_2}] = \bm{\xi}_{v_2} - \bm{\xi}_{v_1} \, ,
\end{align}
and are not necessarily unit vectors, consider edges $\{v_2, \, v_3\}, \, \{v_2, \, v_4\} , \, \{v_3, \, v_4\}$ in \cref{fig:map}.
Consequently, one finds
\begin{align}
    (\Pi_g u)_{,\mu} \, \langle \D_\xi \mu ,\, \bm{\varsigma} \rangle = (\Pi_g u)_{,\mu} \, \langle \D_\xi \mu ,\, \bm{\xi}_{, \mu} \rangle = (\Pi_g u)_{, \mu} \, .
\end{align}
As a result, we can simplify the consistent coupling condition to
\begin{align}
    \langle \Pi_c^{I,0} \vb{p} , \, \bm{\varsigma} \rangle &= (\Pi_g^1 u)_{, \mu} = \widetilde{u} \at_{v_2}  - \widetilde{u} \at_{v_1} \, , \\[1ex]
    \langle \Pi_c^{II,1} \vb{p} , \, \bm{\varsigma} \rangle &= (\Pi_g^2 u)_{, \mu} = \widetilde{u} \at_{v_1}  (4\mu - 3) + \widetilde{u} \at_{m} (4 - 8 \mu)  + \widetilde{u} \at_{v_2}(4 \mu - 1) \, .
\end{align}
For the linear element, each edge has one N\'ed\'elec base function of the first type with a constant tangential projection
\begin{align}
    \langle \bm{\vartheta}_{i} , \, \bm{\varsigma} \rangle \at_{\mu_j} = \delta_{ij} \, , 
\end{align}
allowing to easily embed the consistent coupling condition
\begin{align}
    &p_i: \, &&  \langle \Pi_c^{I,0} \vb{p} , \, \bm{\varsigma} \rangle \at_{\mu_i} = p_i = (\Pi^1_g u)_{,\mu} \at_{\mu_i} = \widetilde{u} \at_{v_2}  - \widetilde{u} \at_{v_1} \, .
\end{align}
Consequently, on each boundary edge the interpolation of one row of the microdistortion is given by
\begin{align}
    \Pi_c^{I,0} \vb{p} \at_{s_i} = p_i \, \bm{\theta}_i \, .
\end{align}
In the quadratic sequence, each edge is associated with two N\'ed\'elec base functions of the second type  $\bm{\vartheta}_{v_1}, \, \bm{\vartheta}_{v_2}$. Their projections on the tangent vector at their respective vertices is one and zero on the other vertex 
\begin{align}
    \langle \bm{\vartheta}_{v_i} , \, \bm{\varsigma} \rangle \at_{v_j} = \delta_{ij} \, .  
\end{align}
Consequently, it suffices to control the tangential projection of each base function at its respective vertex in order to impose the boundary condition
\begin{align}
    &p_{v_1} \, : && \langle \Pi_c^{II,1} \vb{p} , \, \bm{\varsigma} \rangle \at_{v_1} = p_{v_1} = (\Pi_g u)_{,\mu} \at_{v_1} = (\Pi_g u)_{,\mu} \at_{\mu = 0} = -3 \, \widetilde{u} \at_{v_1} + 4 \, \widetilde{u} \at_{m}  - \widetilde{u} \at_{v_2} \, , \notag \\[2ex]
    &p_{v_2} \, : && \langle \Pi_c^{II,1} \vb{p} , \, \bm{\varsigma} \rangle \at_{v_2} = p_{v_2} = (\Pi_g u)_{,\mu} \at_{v_2} = (\Pi_g u)_{,\mu} \at_{\mu = 1} = \, \widetilde{u} \at_{v_1} - 4 \, \widetilde{u} \at_{m}  + 3 \, \widetilde{u} \at_{v_2} \, .
\end{align}
The resulting interpolation of one row of the microdistortion at the boundary edge reads
\begin{align}
    \Pi_c^{II,1} \vb{p} \at_{s_i} = p_{v_1} \bm{\theta}_{v_1} + p_{v_2} \bm{\theta}_{v_2} \, .
\end{align}
For both sequences the above demonstrated methodology satisfies the discrete consistent coupling condition exactly by construction.

\subsection{Element stiffness matrices}
The matrices $\bm{N}$, $\bm{\Theta}$, $\bm{\Phi}$, and $\one$ define the interpolation for each element. 
The resulting finite element stiffness matrices for the primal formulation read
\begin{align}
    \bm{K}_{\vb{u}\vb{u}} &= \int_{\Omega_e} (\D\bm{N})^T (\Ce + \Cc) \D \bm{N} \, \dd X \, , &  
    \bm{K}_{\vb{u}\bm{P}} &= \int_{\Omega_e} (\D\bm{N})^T (\Ce + \Cc)  \, \bm{\Theta} \, \dd X \, ,
    \notag\\[2ex]
    \bm{K}_{\bm{P}\vb{u}} &= \int_{\Omega_e} \bm{\Theta}^T (\Ce + \Cc) \D\bm{N} \, \dd X \, , &  
    \bm{K}_{\bm{P}\bm{P}} &= \int_{\Omega_e} \bm{\Theta}^T (\Ce + \Cc + \Cm) \, \bm{\Theta} + \muma \Lc^2 \, (\Curl \bm{\Theta})^T \Curl \bm{\Theta} \, \dd X \, , \notag \\
    \bm{K}_e &= \begin{bmatrix}
        \bm{K}_{\vb{u}\vb{u}} & \bm{K}_{\vb{u}\bm{P}}\\
        \bm{K}_{\bm{P}\vb{u}} & \bm{K}_{\bm{P}\bm{P}}
    \end{bmatrix} \, .
\end{align}
The linear element's stiffness matrix has the dimensions $\bm{K}_e \in \mathbb{R}^{30 \times 30}$ and the matrix of the quadratic element the dimensions $\bm{K}_e \in \mathbb{R}^{66 \times 66}$.
The load vector of the primal formulation is given by
\begin{align}
    \vb{f}_e = \int_{\Omega_e} \begin{bmatrix}
        \bm{N}^T \vb{f} \\
        \bm{\Theta}^T \bm{M}
    \end{bmatrix} \, \dd X \, , 
\end{align}
with the dimensions $\vb{f}_e \in \mathbb{R}^{30}$ for the linear sequence and $\vb{f}_e \in \mathbb{R}^{66}$ for the quadratic sequence, respectively.

The mixed formulation induces the following additional stiffness matrices 
\begin{align}
    \bm{K}_{\bm{P}\bm{P}} &= \int_{\Omega_e}  \bm{\Theta}^T \Cm \, \bm{\Theta} \, \dd X \, , & \bm{K}_{\bm{P}\bm{D}} &= \int_{\Omega_e}  (\Curl\bm{\Theta})^T \bm{\Phi} \, \dd X \, ,
    \notag\\[2ex]
    \bm{K}_{\bm{D}\bm{P}} &= \int_{\Omega_e}   \bm{\Phi}^T \Curl\bm{\Theta} \, \dd X \, , & \bm{K}_{\bm{D}\bm{D}} &=- \dfrac{1}{\muma \Lc^2} \int_{\Omega_e}   \bm{\Phi}^T \bm{\Phi} \, \dd X \, ,
    \notag\\[2ex]
    \bm{K}_{\bm{D}\vb{q}} &= \int_{\Omega_e} (\Di \bm{\Phi})^T \one \, \dd X = \int_{\Omega_e} (\Di \bm{\Phi})^T \, \dd X \, , &
    \bm{K}_{\vb{q}\bm{D}} &= \int_{\Omega_e} \one \Di \bm{\Phi} \, \dd X = \int_{\Omega_e} \Di \bm{\Phi} \, \dd X \, , 
    \notag\\[2ex] 
    \bm{K}_e &= \begin{bmatrix}
        \bm{K}_{\vb{u}\vb{u}} & \bm{K}_{\vb{u}\bm{P}} &
        0 & 0
        \\
        \bm{K}_{\bm{P}\vb{u}} & \bm{K}_{\bm{P}\bm{P}} 
        & \bm{K}_{\bm{P}\bm{D}} & 0 \\
        0 & \bm{K}_{\bm{D}\bm{P}} & \bm{K}_{\bm{D}\bm{D}} & \bm{K}_{\bm{D}\vb{q}} \\
        0 & 0 & \bm{K}_{\vb{q}\bm{D}} & 0
    \end{bmatrix} \, ,
\end{align}
where $\bm{K}_{\vb{u}\vb{u}}$, $\bm{K}_{\vb{u}\bm{P}}$, and $\bm{K}_{\bm{P}\vb{u}}$ remain unchanged with respect to the primal formulation.
The linear sequence yields the matrix dimensions $\bm{K}_e \in \mathbb{R}^{45 \times 45}$, whereas the quadratic sequence results in $\bm{K}_e \in \mathbb{R}^{81 \times 81}$.
The load vector changes its dimensions in the mixed formulation, but not its content
\begin{align}
    \vb{f}_e = \int_{\Omega_e} \begin{bmatrix}
        \bm{N}^T \vb{f} \\
        \bm{\Theta}^T \bm{M} \\
        0
    \end{bmatrix} \, \dd X \, , 
\end{align} 
with dimensions corresponding to the element stiffness matrix. 
\begin{remark}
    Note that the symmetry and anti-symmetry operators have been dropped. This is because the material tensors are defined as $\Ce, \, \Cm: \mathbb{R}^{3 \times 3} \mapsto \Sym(3)$ and $\Cc: \mathbb{R}^{3\times 3} \mapsto \so(3)$, implying the former. For materials where this is not guaranteed, the operators can be recovered by modifying the material tensors 
    \begin{align}
        \Ce \sym \Pm = \Ce \mathbb{S} \Pm = \widetilde{\mathbb{C}}_\mathrm{e} \Pm \, ,  &&  \Cm \sym \Pm = \Cm \mathbb{S} \Pm = \widetilde{\mathbb{C}}_\mathrm{micro} \Pm \, , && \Cc \skw \Pm = \Cc \mathbb{A} \Pm = \widetilde{\mathbb{C}}_\mathrm{c} \Pm \, .
    \end{align}
\end{remark}

\section{Numerical examples}\label{ch:7}
In the following we test the convergence rates of the model. The behaviour with respect to the characteristic length $\Lc$ is crucial, as this parameter governs the impact of the micro-structure on the displacement field $\vb{u}$ and the relation to the Cauchy continuum. 

In order to study convergence we construct artificial analytical solutions for the displacement field $\widetilde{u}$ and the microdistortion field $\widetilde{P}$ (see \cref{ap:A}) and compare them with the numerical approximation of the model.
Errors are measured in the $\Le$-norm
\begin{align}
    &\| \widetilde{\vb{u}} - \vb{u} \|_{\Le} = \sqrt{\sum_{e = 1}^n \int_{\Omega_e} \| \widetilde{\vb{u}} - \vb{u} \|^2 \, \dd X} \, , && \| \widetilde{\bm{P}} - \bm{P} \|_{\Le} = \sqrt{\sum_{e = 1}^n \int_{\Omega_e} \| \widetilde{\bm{P}} - \bm{P} \|^2 \, \dd X} \, .
\end{align}

\subsection{Convergence for $\Lc \to 0$}
In this benchmark the domain is defined as the axis-symmetric cube $\Omega = [-1, \, 1]^3$. For simplicity the material constants in the isotropic case are set to 
\begin{align}
    \lambda_\mathrm{e}, \, \mu_\mathrm{e}, \, \mu_\mathrm{c} , \, \lambda_\mathrm{micro}, \, \mu_\mathrm{micro}, \, \mu_\mathrm{macro} = 1 \, .
\end{align}
The characteristic length is varied between $\Lc \in (0, \, 1]$ in order to test the stability of the formulations.
The prescribed fields are defined as 
\begin{align}
    \widetilde{\vb{u}} = \begin{bmatrix}
        0 \\ 0 \\ (1 - x)^2 \, (1 + x)^2
    \end{bmatrix} \, , && \widetilde{\bm{P}} = (1 - x) \, (1 + x) \begin{bmatrix}
        -y - z & x & x \\
        -y - z & x & x \\
        -y - z & x & x 
    \end{bmatrix} \, ,
\end{align}
being of higher polynomial order than the quadratic element. The entire boundary is set to $\partial \Omega = \Gamma_D$.
By applying the strong form one finds the force 
\begin{align}
    \vb{f} = \left[\begin{matrix}- 6 \, x^{2} + 6 \, x \, y + 6 \, x \, z + 2\\ x^{2} + 4 \, x \, y + 4 \, x \, z - 1\\- 23 \, x^{2} + 4 \, x \, y + 4 \, x \, z + 7\end{matrix}\right] \, , 
\end{align}
and micro-moment
\begin{align}
    \bm{M} =& \left[\begin{matrix}\left(x - 1\right) \left(x + 1\right) \left(- 4 x + 6 y + 6 z\right) & \left(x - 1\right) \left(x + 1\right) \left(- 3 x + y + z\right) & \left(x - 1\right) \left(x + 1\right) \left(- 3 x + 1 y + 1 z\right)\\\left(x - 1\right) \left(x + 1\right) \left(- 1 x + 3 y + 3 z\right) & \left(x - 1\right) \left(x + 1\right) \left(- 8 x + 2 y + 2 z\right) & 4 x \left(1 - x^{2}\right)\\- 9 x^{3} + 3 x^{2} y + 3 x^{2} z + 9 x - 3 y - 3 z & 4 x \left(1 - x^{2}\right) & \left(x - 1\right) \left(x + 1\right) \left(- 8 x + 2 y + 2 z\right)\end{matrix}\right] \notag \\
    &+ \Lc^2 \left[\begin{matrix}0 & 8 x & 8 x\\0 & 8 x & 8 x\\0 & 8 x & 8 x\end{matrix}\right] \, . 
\end{align}
Reducing the characteristic length $\Lc$ or setting it to zero does not disturb the stability of the computation, even though control of the Curl of $\Pm$ is lost, as proven by the convergence rates in \cref{fig:conv_lc0_muc1}. This is a direct result of the derivation in \cref{sec:lczero}, asserting stable computations for $\Pm \in [\Le(\Omega)]^{3 \times 3}$.
\begin{figure}
	\centering
	\begin{subfigure}{0.48\linewidth}
		\begin{tikzpicture}
		\definecolor{npurple}{rgb}{0.4980392156862745,0.,1.}
			\begin{loglogaxis}[
				/pgf/number format/1000 sep={},
				axis lines = left,
				xlabel={Degrees of freedom (Log)},
				ylabel={$\| \widetilde{\vb{u}} - \vb{u}\|_{\Le}$ (Log)},
				xmin=100, xmax=1e5,
				ymin=1e-3, ymax=100,
				xtick={100,1000,10000,1e5},
				ytick={1e-3,1e-2,1e-1,1, 10,100},
				legend pos= north east,
				ymajorgrids=true,
				grid style=dotted,
				]
				\addplot[
				color=npurple,
				mark=triangle,
				]
				coordinates {
					( 165 ,  0.2821303631805278 )
				    ( 2241 ,  0.03460126169470074 )
				    ( 9069 ,  0.010569364916892884 )
				    ( 23385 ,  0.00454690987524941 )
				};
				\addlegendentry{$\Lc = 10^{-6}$}
				\addplot[
				color=violet,
				mark=o,
				]
				coordinates {
					( 165 ,  0.2821303631807766 )
					( 2241 ,  0.034601261694742085 )
					( 9069 ,  0.010569364916887026 )
					( 23385 ,  0.004546909875231032 )
				};
				\addlegendentry{$\Lc = 10^{-7}$}
				\addplot[
				color=teal,
				mark=square,
				]
				coordinates {
				    ( 165 ,  0.28213036318077944 )
				    ( 2241 ,  0.03460126169474249 )
				    ( 9069 ,  0.010569364916886906 )
				    ( 23385 ,  0.004546909875230665 )
				};
				\addlegendentry{$\Lc = 10^{-8}$}
				\addplot[
				color=blue,
				mark=pentagon,
				]
				coordinates {
				    ( 165 ,  0.28213036318077966 )
				    ( 2241 ,  0.034601261694742536 )
				    ( 9069 ,  0.01056936491688737 )
				    ( 23385 ,  0.004546909875230773 )
				};
				\addlegendentry{$\Lc = 10^{-9}$}
				\addplot[dashed,color=black, mark=none]
				coordinates {
					(250, 0.02)
					(10000, 0.0017099759466766972)
				};
			\end{loglogaxis}
			\draw (2.5,1.4) node[anchor=north west]{$\mathcal{O}(h^2)$};
		\end{tikzpicture}
	\end{subfigure}
	\begin{subfigure}{0.48\linewidth}
		\begin{tikzpicture}
		\definecolor{npurple}{rgb}{0.4980392156862745,0.,1.}
			\begin{loglogaxis}[
				/pgf/number format/1000 sep={},
				axis lines = left,
				xlabel={Degrees of freedom (Log)},
				ylabel={$\| \widetilde{\bm{P}} - \bm{P}\|_{\Le}$ (Log)},
				xmin=100, xmax=1e5,
				ymin=1e-3, ymax=100,
				xtick={100,1000,10000,1e5},
				ytick={1e-3,1e-2,1e-1,1, 10,100},
				legend pos= south west,
				ymajorgrids=true,
				grid style=dotted,
				]
				\addplot[
				color=npurple,
				mark=triangle,
				]
				coordinates {
				    ( 165 ,  2.251032611571998 )
				    ( 2241 ,  0.4750945567336475 )
				    ( 9069 ,  0.19229746152151736 )
				    ( 23385 ,  0.10137962849435002 )
				};
				\addlegendentry{$\Lc = 10^{-6}$}
				\addplot[
				color=violet,
				mark=o,
				]
				coordinates {
				    ( 165 ,  2.2510326115732937 )
				    ( 2241 ,  0.47509455673478185 )
				    ( 9069 ,  0.192297461522773 )
				    ( 23385 ,  0.10137962849571058 )
				};
				\addlegendentry{$\Lc = 10^{-7}$}
				\addplot[
				color=teal,
				mark=square,
				]
				coordinates {
				    ( 165 ,  2.2510326115733066 )
				    ( 2241 ,  0.47509455673479356 )
				    ( 9069 ,  0.19229746152278476 )
				    ( 23385 ,  0.10137962849572427 )
				};
				\addlegendentry{$\Lc = 10^{-8}$}
				\addplot[
				color=blue,
				mark=pentagon,
				]
				coordinates {
				    ( 165 ,  2.2510326115733066 )
				    ( 2241 ,  0.4750945567347937 )
				    ( 9069 ,  0.1922974615227841 )
				    ( 23385 ,  0.10137962849572534 )
				};
				\addlegendentry{$\Lc = 10^{-9}$}
				\addplot[dashed,color=black, mark=none]
				coordinates {
					(500, 5)
					(25000, 0.36840314986403866)
				};
			\end{loglogaxis}
			\draw (3.2,4.2) node[anchor=north west]{$\mathcal{O}(h^2)$};
		\end{tikzpicture}
	\end{subfigure}
	\caption{Convergence rates for $\Lc \to 0$ and $\mu_\mathrm{c} = 1$ for the quadratic sequence. The convergence remains stable for $\Lc \to 0$.}
	\label{fig:conv_lc0_muc1}
\end{figure}
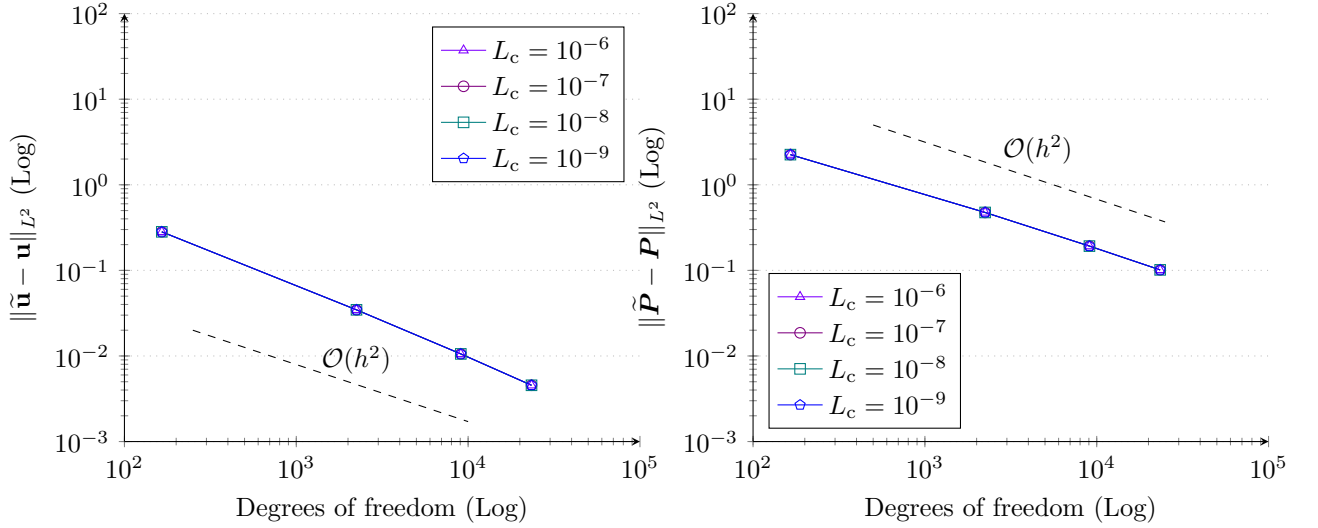

In an alternative example we take the Cosserat couple modulus $\mu_\mathrm{c} = 0$ (i.e., $\Cc = 0$), resulting in the force
\begin{align}
    \vb{f} = \left[\begin{matrix}x \left(- 4\, x + 6 \,y + 6 \,z\right)\\2 x \left(- x + y + z\right)\\- 14\, x^{2} + 2 \,x\, y + 2 \,x \,z + 4\end{matrix}\right] \, ,
\end{align}
and the micro-moment
\begin{align}
    \bm{M} =& \left[\begin{matrix}\left(x - 1\right) \left(x + 1\right) \left(- 4 x + 6 y + 6 z\right) & 2 \left(x - 1\right) \left(x + 1\right) \left(- x + y + z\right) & - 6 x^{3} + 2 x^{2} y + 2 x^{2} z + 6 x - 2 y - 2 z\\2 \left(x - 1\right) \left(x + 1\right) \left(- x + y + z\right) & \left(x - 1\right) \left(x + 1\right) \left(- 8 x + 2 y + 2 z\right) & 4 x \left(1 - x^{2}\right)\\- 6 x^{3} + 2 x^{2} y + 2 x^{2} z + 6 x - 2 y - 2 z & 4 x \left(1 - x^{2}\right) & \left(x - 1\right) \left(x + 1\right) \left(- 8 x + 2 y + 2 z\right)\end{matrix}\right] \notag \\
    &+ \Lc^2 \left[\begin{matrix}0 & 8 x & 8 x\\0 & 8 x & 8 x\\0 & 8 x & 8 x\end{matrix}\right] \, .
\end{align}
The primal formulation maintains its stability up to small $\Lc$-values coupled with small elements, see \cref{fig:conv_lc0_muc0}. For $\Lc = 10^{-10}$ the formulation becomes completely unstable and the matrix singular. The deterioration in the convergence is clearly depicted in \cref{fig:lc0}. Since $\muc = 0$ and $\Lc \to 0$, no term controls the skew-symmetric part of the microdistortion and the formulation only converges for the displacement $\vb{u}$, up to $\Lc = 10^{-10}$, where the stiffness matrix becomes singular.

\begin{figure}
	\centering
	\begin{subfigure}{0.48\linewidth}
		\begin{tikzpicture}
		\definecolor{npurple}{rgb}{0.4980392156862745,0.,1.}
			\begin{loglogaxis}[
				/pgf/number format/1000 sep={},
				axis lines = left,
				xlabel={Degrees of freedom (Log)},
				ylabel={$\| \widetilde{\vb{u}} - \vb{u}\|_{\Le}$ (Log)},
				xmin=100, xmax=1e5,
				ymin=1e-4, ymax=100000,
				xtick={100,1000,10000,1e5},
				ytick={1e-4,1e-2,1, 100,10000},
				legend pos= north east,
				ymajorgrids=true,
				grid style=dotted,
				]
				\addplot[
				color=npurple,
				mark=triangle,
				]
				coordinates {
				    ( 165 ,  0.1739641654817233 )
				    ( 2241 ,  0.02516444965396782 )
				    ( 9069 ,  0.007581667692057104 )
				    ( 23385 ,  0.003213616641034855 )
				};
				\addlegendentry{$\Lc = 10^{-6}$}
				\addplot[
				color=violet,
				mark=o,
				]
				coordinates {
				    ( 165 ,  0.17396416548172358 )
				    ( 2241 ,  0.02516444965396797 )
				    ( 9069 ,  0.007581667692057311 )
				    ( 23385 ,  0.0032136166410345886 )
				};
				\addlegendentry{$\Lc = 10^{-7}$}
				\addplot[
				color=teal,
				mark=square,
				]
				coordinates {
				    ( 165 ,  0.17396416548172391 )
				    ( 2241 ,  0.02516444965396804 )
				    ( 9069 ,  0.00758166769205726 )
				    ( 23385 ,  0.003213616641034271 )
				};
				\addlegendentry{$\Lc = 10^{-8}$}
				\addplot[
				color=blue,
				mark=pentagon,
				]
				coordinates {
				    ( 165 ,  0.17396416548172308 )
				    ( 2241 ,  0.025164449653951825 )
				    ( 9069 ,  0.007581667692059966 )
				    ( 23385 ,  0.0032136166410418423 )
				};
				\addlegendentry{$\Lc = 10^{-9}$}
				\addplot[dashed,color=black, mark=none]
				coordinates {
					(250, 0.02)
					(10000, 0.0017099759466766972)
				};
			\end{loglogaxis}
			\draw (2.3,1.1) node[anchor=north west]{$\mathcal{O}(h^2)$};
		\end{tikzpicture}
	\end{subfigure}
	\begin{subfigure}{0.48\linewidth}
		\begin{tikzpicture}
		\definecolor{npurple}{rgb}{0.4980392156862745,0.,1.}
			\begin{loglogaxis}[
				/pgf/number format/1000 sep={},
				axis lines = left,
				xlabel={Degrees of freedom (Log)},
				ylabel={$\| \widetilde{\bm{P}} - \bm{P}\|_{\Le}$ (Log)},
				xmin=100, xmax=1e5,
				ymin=1e-4, ymax=100000,
				xtick={100,1000,10000,1e5},
				ytick={1e-4,1e-2,1, 100,10000},
				legend pos= south west,
				ymajorgrids=true,
				grid style=dotted,
				]
				\addplot[
				color=npurple,
				mark=triangle,
				]
				coordinates {
					( 165 ,  2.723198237156818 )
				    ( 2241 ,  0.5600411872231676 )
				    ( 9069 ,  0.23192083303998345 )
				    ( 23385 ,  0.12464011268338643 )
				};
				\addlegendentry{$\Lc = 10^{-6}$}
				\addplot[
				color=violet,
				mark=o,
				]
				coordinates {
					( 165 ,  2.723369097038974 )
					( 2241 ,  0.5616709846575639 )
					( 9069 ,  0.2511482476071055 )
					( 23385 ,  0.1328120159447346 )
				};
				\addlegendentry{$\Lc = 10^{-7}$}
				\addplot[
				color=teal,
				mark=square,
				]
				coordinates {
				    ( 165 ,  12.348226893262659 )
				    ( 2241 ,  2.4015224041833645 )
				    ( 9069 ,  5.066991735986115 )
				    ( 23385 ,  2063.645285311393 )
				};
				\addlegendentry{$\Lc = 10^{-8}$}
				\addplot[
				color=blue,
				mark=pentagon,
				]
				coordinates {
				    ( 165 ,  12.290839082077994 )
				    ( 2241 ,  11056.047723407932 )
				    ( 9069 ,  3217.171962414453 )
				    ( 23385 ,  23644.50176321061 )
				};
				\addlegendentry{$\Lc = 10^{-9}$}
				\addplot[dashed,color=black, mark=none]
				coordinates {
					(500, 5)
					(25000, 0.36840314986403866)
				};
			\end{loglogaxis}
			\draw (4.4,2.95) node[anchor=north west]{$\mathcal{O}(h^2)$};
		\end{tikzpicture}
	\end{subfigure}
	\caption{Convergence rates for $\Lc \to 0$ and $\mu_\mathrm{c} = 0$ for the quadratic sequence. Convergence is lost for small characteristic length $\Lc$ values.}
	\label{fig:conv_lc0_muc0}
\end{figure}
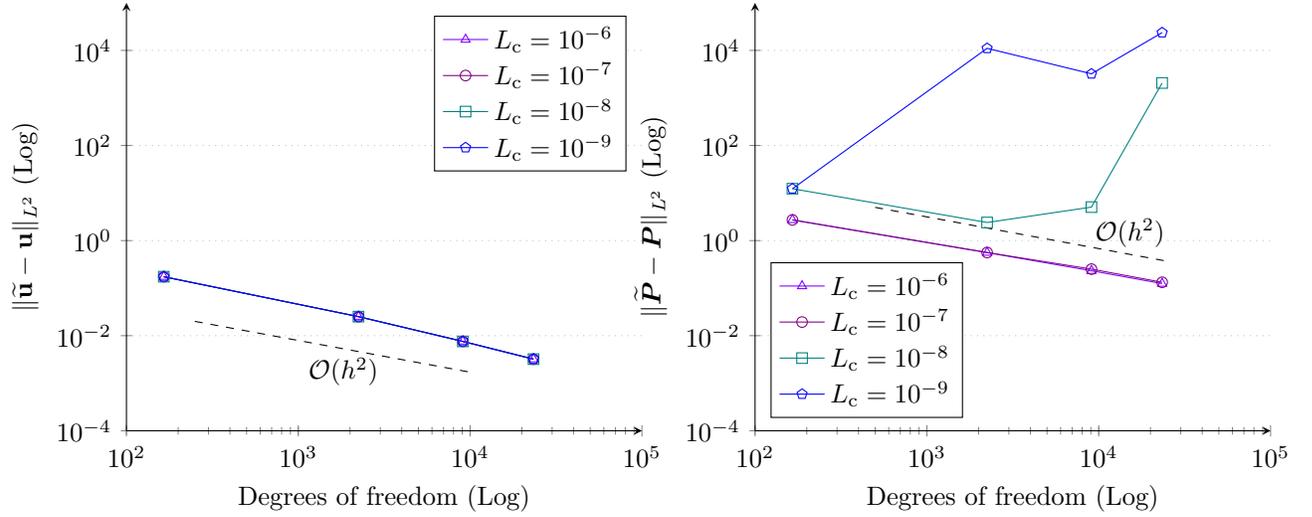
\begin{figure}
	\centering
	\begin{subfigure}{0.24\linewidth}
	\centering
		\includegraphics[width=1.0\linewidth]{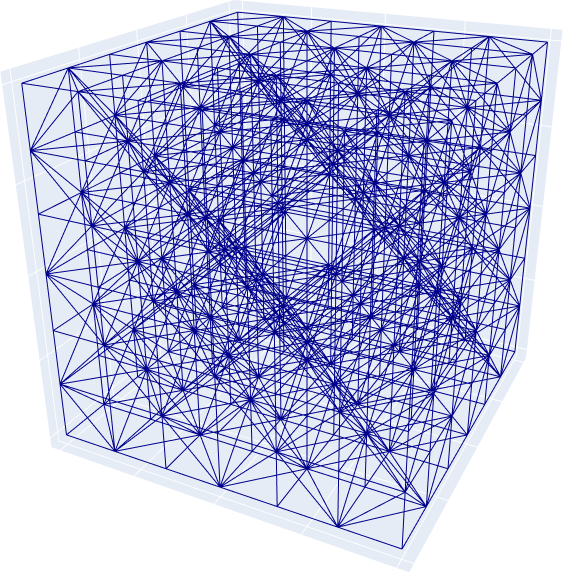}
		\caption{} 
	\end{subfigure}
	\begin{subfigure}{0.24\linewidth}
	\centering
		\includegraphics[width=1.0\linewidth]{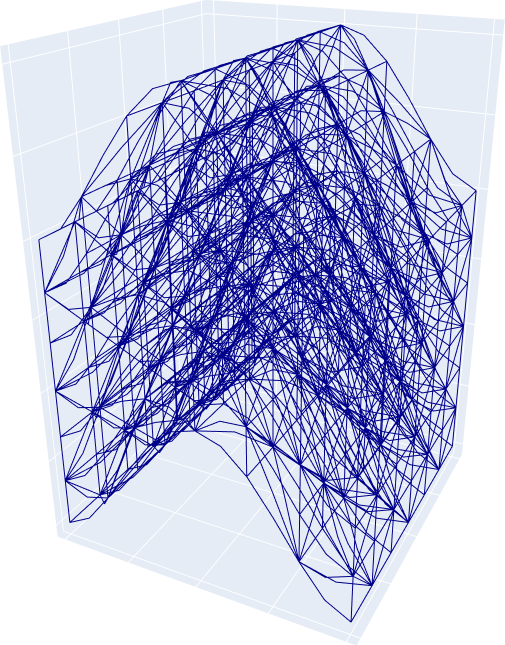}
		\caption{}
	\end{subfigure}
	\begin{subfigure}{0.24\linewidth}
	\centering
		\includegraphics[width=1.0\linewidth]{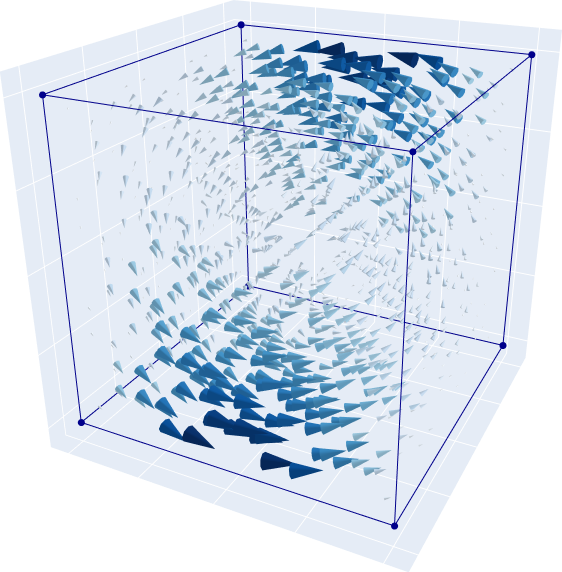}
		\caption{}
	\end{subfigure}
	\begin{subfigure}{0.24\linewidth}
	\centering
		\includegraphics[width=1.0\linewidth]{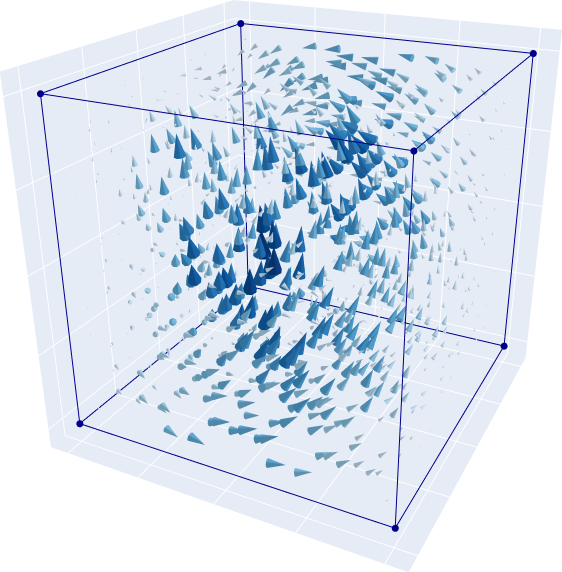}
		\caption{}
	\end{subfigure}
	\caption{(a) Initial geometry with 1080 elements, (b) displacement field, (c) first row of the microdistortion field with $\Lc = 10^{-6}$, (d) first row of the microdistortion field with $\Lc = 10^{-8}$.}
	\label{fig:lc0}
\end{figure}

\subsection{Robustness in $\Lc$}\label{ch:robust}
To test for robustness of the formulation in the characteristic length $\Lc$ we again consider $\Omega=[-1,\,1]^3$, the parameters
\begin{align}
    \lambda_\mathrm{e}= \mu_\mathrm{e}= \mu_\mathrm{c}= \lambda_\mathrm{micro}= \mu_\mathrm{micro}= \mu_\mathrm{macro} = 1 \, \quad \text{and} \quad \muc = 0,
\end{align}
and the following exact solution fields
\begin{align}
\label{eq:ex_sol_robustness_lc}
    \widetilde{\vb{u}} = \begin{bmatrix}
        0 \\ 0 \\ (1 - x)^2 \, (1 + x)
    \end{bmatrix} \, , && \widetilde{\bm{P}} = \begin{bmatrix}
        x(y^2-1) & y(x^2-1) & 0\\
        0 & y(z^2-1) & z(y^2-1) \\
        x(z^2-1) & 0 & z(x^2-1) 
    \end{bmatrix} + \frac{10}{\Lc^2}(1-x)(1-y)(1-z)\begin{bmatrix}
        -y & x & 0 \\
        0 & -z & y \\
        z & 0 & -x
    \end{bmatrix}\, ,
\end{align}
where the first part of $\widetilde{\bm{P}}$ is curl-free. The second term is scaled by $\Lc^{-2}$ to obtain a well-defined limit for $\Lc\to\infty$. The corresponding right-hand sides read
\begin{align*}
   \vb{f}=\left[\begin{matrix} x^{2} + 4 x z + 3 y^{2} - 4 \\4 x y + y^{2} + 3 z^{2} - 4\\- 9 x^{2} + 4y z + z^{2} \end{matrix}\right] +  \dfrac{1}{\Lc^2} \left[\begin{matrix}- 10 x^{2} z + 10 x^{2} - 10 x y + 10 x z + 30 y^{2} z - 30 y^{2} + 10 y z^{2} - 30 y z + 30 y - 10 z^{2}\\10 x^{2} z - 10 x^{2} - 10 x y^{2} + 10 x y + 30 x z^{2} - 30 x z + 10 y^{2} - 10 y z - 30 z^{2} + 30 z\\30 x^{2} y - 30 x^{2} + 10 x y^{2} - 30 x y - 10 x z + 30 x - 10 y^{2} - 10 y z^{2} + 10 y z + 10 z^{2}\end{matrix}\right],
\end{align*}
\begin{align*}
\bm{M} =& 2\left[\begin{matrix}x^2 z + 3 x y^2 + y z^2 & x^2 y& -2 x^3 + x z^2 \\ x^2 y  & x^2 z + x y^2 + 3 y z^2 &  y^2 z \\ -2 x^3 + x z^2 &  y^2 z & 3 x^2 z + x y^2 + y z^2 \end{matrix}\right] \\
    & +2\left[\begin{matrix}- 20 x z + 17 x - y + 14 z - 15 &20 y z - 21 y - 15 z + 15 & - 5 x^2+ 6 x + 5 y^2 - 5 y \\ - 5 y^2 + 4 y + 5 z^2 - 5 z & - 20 x y + 14 x + 17 y - z - 15 & 20 x z - 15 x- 21 z + 15 \\ 20 x y- 19 x - 15 y + 15 & 5 x^2 - 5 x- 5 z^2 + 4 z & - x - 20 y z + 14 y + 17 z - 15\end{matrix}\right]\\
&+ \dfrac{(x - 1) (y - 1) (z - 1)}{\Lc^2} \left[\begin{matrix}20  (x + 3 y + z) & - 20 x  & - 20 z  \\ - 20 x  & 20  (x + y + 3 z) &  - 20 y  \\ - 20 z  & - 20 y &20  (3 x + y + z)\end{matrix}\right].
\end{align*}

We use NGSolve in combination with the direct solver UMFPACK \cite{DD97} for inverting the arising stiffness matrix for the primal and mixed method. As can be clearly observed in Figure~\ref{fig:conv_robust_lc} for larger values of $\Lc\geq 10^6$ the primal method becomes unstable, whereas the mixed method still converges optimally.

\begin{figure}
	\centering
	\begin{subfigure}{0.48\linewidth}
		\begin{tikzpicture}
		\definecolor{npurple}{rgb}{0.4980392156862745,0.,1.}
			\begin{loglogaxis}[
				/pgf/number format/1000 sep={},
				axis lines = left,
				xlabel={Degrees of freedom (Log)},
				ylabel={$\| \bm{P} - \widetilde{\bm{P}}\|_{\Le}/\|\widetilde{\bm{P}}\|_{\Le}$ (Log)},
				xmin=0.2, xmax=3e6,
				ymin=8e-4, ymax=500,
				xtick={1,100,10000,1e6,1e8,1e10},
				ytick={1e-3,1e-2,1e-1,1, 10},
				legend pos= north west,
				ymajorgrids=true,
				grid style=dotted,
				]
\addplot[color=npurple, mark=triangle] coordinates {
(195, 1.9915419952627995)
(963, 0.480754609392379)
(5811, 0.11541240257863444)
(39843, 0.02826423509725097)
(293955, 0.007019310338891732)
};

\addlegendentry{$\Lc=10^{0}$}

\addplot[color=violet, mark=o] coordinates {
(195, 0.8708756676569545)
(963, 0.24013674618471556)
(5811, 0.062260074849321466)
(39843, 0.015738298230094882)
(293955, 0.0039466957822219175)
};
\addlegendentry{$\Lc=10^{1}$}
\addplot[color=teal, mark=square] coordinates {
(195, 0.7808998282887822)
(963, 0.2169988077695519)
(5811, 0.0562591184171324)
(39843, 0.014210620977286072)
(293955, 0.00356248391583371)
};
\addlegendentry{$\Lc=10^{2}$}
\addplot[color=blue, mark=pentagon] coordinates {
(195, 0.7809796479090094)
(963, 0.21702236519474136)
(5811, 0.05626143907495484)
(39843, 0.014210689141122822)
(293955, 0.0035624574716154704)
};
\addlegendentry{$\Lc=10^{3}$}
\addplot[color=cyan, mark=diamond] coordinates {
(195, 0.7809805541832068)
(963, 0.21702262892946128)
(5811, 0.05626146923046596)
(39843, 0.01421069154209374)
(293955, 0.0035624576352376314)
};
\addlegendentry{$\Lc=10^{4}$}
\addplot[color=red, mark=+] coordinates {
(195, 0.7809805632567549)
(963, 0.2170226318810565)
(5811, 0.05626146959573117)
(39843, 0.014210691506921587)
(293955, 0.0035624576203609846)
};
\addlegendentry{$\Lc=10^{5}$}
\addplot[color=magenta, mark=x] coordinates {
(195, 0.7809805633474917)
(963, 0.2170225739018566)
(5811, 0.056261471619002835)
(39843, 0.014210693989964203)
(293955, 0.035159003642851044)
};
\addlegendentry{$\Lc=10^{6}$}
\addplot[color=olive, mark=square*] coordinates {
(195, 0.7809805633483989)
(963, 0.21702898744037147)
(5811, 0.06170910596000579)
(39843, 134.11258212183648)
(293955, 4.455390279798243)
};
\addlegendentry{$\Lc=10^{7}$}
\addplot[color=pink, mark=*] coordinates {
(195, 0.7809805633484078)
(963, 0.217793077144513)
(5811, 0.05635278276362786)
(39843, 0.014214924477280863)
(293955, 773.3361951257252)
};
\addlegendentry{$\Lc=10^{8}$}
\addplot[color=darkgray, mark=diamond*] coordinates {
(195, 0.7809805633484082)
(963, 0.21799818538283436)
(5811, 0.0563460314288825)
(39843, 0.01421499960483926)
(293955, 8.23898178335574)
};

\addlegendentry{$\Lc=10^{9}$}

                \addplot[dashed,color=black, mark=none]
				coordinates {
					(200, 0.10234)
					(40000, 0.0029925)
				};
				
		\end{loglogaxis}
			\draw (3,1.4) node[anchor=north west]{$\mathcal{O}(h^2)$};
		\end{tikzpicture}
	\end{subfigure}
	\begin{subfigure}{0.48\linewidth}
		\begin{tikzpicture}
		\definecolor{npurple}{rgb}{0.4980392156862745,0.,1.}
			\begin{loglogaxis}[
				/pgf/number format/1000 sep={},
				axis lines = left,
				xlabel={Degrees of freedom (Log)},
				ylabel={$\| \bm{P} - \widetilde{\bm{P}}\|_{\Le}/\|\widetilde{\bm{P}}\|_{\Le}$ (Log)},
				xmin=0.2, xmax=3e6,
				ymin=8e-4, ymax=500,
				xtick={1,100,10000,1e6},
				ytick={1e-3,1e-2,1e-1,1, 10},
				legend pos= north west,
				ymajorgrids=true,
				grid style=dotted,
				]
		\addplot[color=npurple, mark=triangle] coordinates {
(270, 1.9915419952627988)
(1470, 0.4807546093923815)
(9558, 0.11541240257863805)
(68646, 0.028264235097255)
(519750, 0.007019310338895832)
};
\addlegendentry{$\Lc=10^{0}$}

\addplot[color=violet, mark=o] coordinates {
(270, 0.8708756676569542)
(1470, 0.240136746184716)
(9558, 0.06226007484932211)
(68646, 0.01573829823009569)
(519750, 0.003946695782222642)
};
\addlegendentry{$\Lc=10^{1}$}

\addplot[color=teal, mark=square] coordinates {
(270, 0.7808998282887825)
(1470, 0.2169988077695516)
(9558, 0.056259118417132235)
(68646, 0.014210620977286125)
(519750, 0.003562483915833717)
};
\addlegendentry{$\Lc=10^{2}$}

\addplot[color=blue, mark=pentagon] coordinates {
(270, 0.7809796479090094)
(1470, 0.21702236519480556)
(9558, 0.0562614390749509)
(68646, 0.014210689141119497)
(519750, 0.003562457471619781)
};
\addlegendentry{$\Lc=10^{3}$}

\addplot[color=cyan, mark=diamond] coordinates {
(270, 0.7809805541832068)
(1470, 0.21702262893101515)
(9558, 0.05626146923033806)
(68646, 0.014210691541984757)
(519750, 0.003562457635391576)
};
\addlegendentry{$\Lc=10^{4}$}

\addplot[color=red, mark=+] coordinates {
(270, 0.7809805632567549)
(1470, 0.217022631571193)
(9558, 0.056261469532586816)
(68646, 0.014210691566165331)
(519750, 0.0035624576370720993)
};
\addlegendentry{$\Lc=10^{5}$}

\addplot[color=magenta, mark=x] coordinates {
(270, 0.7809805633474917)
(1470, 0.2170226315975952)
(9558, 0.056261469535609336)
(68646, 0.014210691566407176)
(519750, 0.003562457637088906)
};
\addlegendentry{$\Lc=10^{6}$}

\addplot[color=olive, mark=square*] coordinates {
(270, 0.7809805633483989)
(1470, 0.21702263159785928)
(9558, 0.05626146953563964)
(68646, 0.014210691566409616)
(519750, 0.003562457637089072)
};
\addlegendentry{$\Lc=10^{7}$}

\addplot[color=pink, mark=*] coordinates {
(270, 0.7809805633484079)
(1470, 0.21702263159786195)
(9558, 0.05626146953563994)
(68646, 0.014210691566409613)
(519750, 0.003562457637089065)
};
\addlegendentry{$\Lc=10^{8}$}

\addplot[color=darkgray, mark=diamond*] coordinates {
(270, 0.7809805633484082)
(1470, 0.2170226315978619)
(9558, 0.056261469535639985)
(68646, 0.014210691566409585)
(519750, 0.0035624576370890705)
};

\addlegendentry{$\Lc=10^{9}$}

\addplot[dashed,color=black, mark=none]
				coordinates {
					(200, 0.10234)
					(40000, 0.0029925)
				};
				
		\end{loglogaxis}
			\draw (3,1.4) node[anchor=north west]{$\mathcal{O}(h^2)$};
		\end{tikzpicture}
	\end{subfigure}
	
	\caption{Convergence behavior for the quadratic sequence with primal (left) and mixed (right) method for $\Lc=10^0$, $10^1$, $\dots$, $10^9$.}
	\label{fig:conv_robust_lc}
\end{figure}
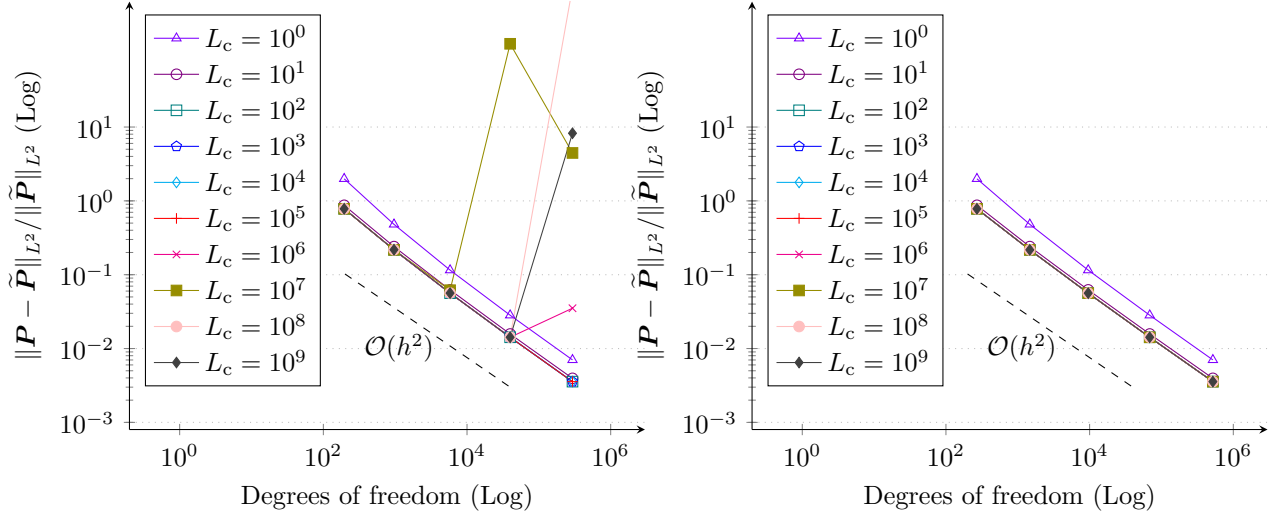

\subsection{Convergence for $\Lc \to \infty$}
Next, we consider the convergence to the limit solution $\Lc\to\infty$. Rewriting the exact solution \eqref{eq:ex_sol_robustness_lc} in the mixed formulation \eqref{eq:mixed_problem_limit} and taking the limit $\Lc\to\infty$ yields the following displacement, microdistortion, and hyperstress fields
\begin{align}
\begin{split}
    \widetilde{\vb{u}} &= \begin{bmatrix}
        0 \\ 0 \\ (1 - x)^2 \, (1 + x)
    \end{bmatrix} \, , \qquad \widetilde{\bm{P}} = \begin{bmatrix}
        x(y^2-1) & y(x^2-1) & 0\\
        0 & y(z^2-1) & z(y^2-1) \\
        x(z^2-1) & 0 & z(x^2-1) 
    \end{bmatrix}\,,\\ 
    \widetilde{\bm{D}}&=\Curl(10(1-x)(1-y)(1-z)\begin{bmatrix}
        -y & x & 0 \\
        0 & -z & y \\
        z & 0 & -x
    \end{bmatrix}) \\
    &=10\begin{bmatrix}
        x(1-x)(1-y) & y(1-x)(1-y) & -(z-1)(4xy - 3x - 3y + 2) \\
        -(x-1)(4yz-3y-3z+2) & y(1-y)(1-z) & z(1-y)(1-z) \\
        x(1-x)(1-z) & -(y-1)(4xz - 3x - 3z + 2) & z(1-x)(1-z)
    \end{bmatrix}\,.
\end{split}
\end{align}
Note, that $\Curl\widetilde{\bm{P}}=\bm{0}$ and $\Di\widetilde{\bm{D}}=\vb{0}$. The hyperstress field $\widetilde{\bm{D}}$ acts as Lagrange multiplier forcing the microdistortion tensor to be curl-free. By taking the limit $\Lc\to\infty$ we obtain directly the forces
\begin{align*}
\vb{f}=\left[\begin{matrix}x^{2} + 4 x z + 3 y^{2} - 4\\4 x y + y^{2} + 3 z^{2} - 4\\- 9 x^{2} + 4 y z + z^{2}\end{matrix}\right],
\end{align*}
\begin{align*}
    \bm{M}&=2\left[\begin{matrix}x^2 z + 3 x y^2 + y z^2 & x^2 y& -2 x^3 + x z^2 \\ x^2 y  & x^2 z + x y^2 + 3 y z^2 &  y^2 z \\ -2 x^3 + x z^2 &  y^2 z & 3 x^2 z + x y^2 + y z^2 \end{matrix}\right] \\
    &\quad +2\left[\begin{matrix}- 20 x z + 17 x - y + 14 z - 15 &20 y z - 21 y - 15 z + 15 & - 5 x^2+ 6 x + 5 y^2 - 5 y \\ - 5 y^2 + 4 y + 5 z^2 - 5 z & - 20 x y + 14 x + 17 y - z - 15 & 20 x z - 15 x- 21 z + 15 \\ 20 x y- 19 x - 15 y + 15 & 5 x^2 - 5 x- 5 z^2 + 4 z & - x - 20 y z + 14 y + 17 z - 15\end{matrix}\right].
\end{align*}

We consider five different structured tetrahedral grids consisting of 6, 48, 384, 3072, and 24576 elements and solve the primal and mixed formulation with NGSolve and UMFPACK using different values for $\Lc$. As depicted in Figure~\ref{fig:conv_lc_inf} the expected quadratic convergence rate in $\Lc$ \eqref{eq:mixed_method_conv_Lc} is observed until the discretization error is reached \eqref{eq:conv_lc_h}. For finer grids we can reduce this error. However, the primal method again becomes unstable for large values of $\Lc$. The mixed method, where also the limit problem with $\Lc=\infty$ is well defined, is again completely stable with respect to $\Lc$.

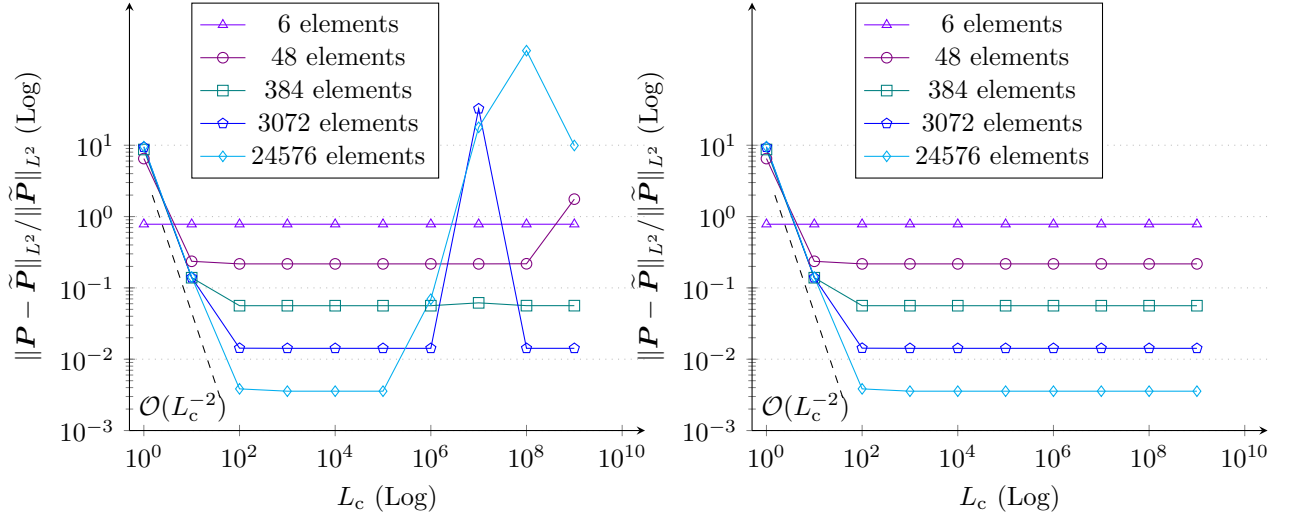
\begin{figure}
	\centering
	\begin{subfigure}{0.48\linewidth}
		\begin{tikzpicture}
		\definecolor{npurple}{rgb}{0.4980392156862745,0.,1.}
			\begin{loglogaxis}[
				/pgf/number format/1000 sep={},
				axis lines = left,
				xlabel={$\Lc$ (Log)},
				ylabel={$\| \bm{P} - \widetilde{\bm{P}}\|_{\Le}/\|\widetilde{\bm{P}}\|_{\Le}$ (Log)},
				xmin=0.5, xmax=3e10,
				ymin=1e-3, ymax=1000,
				xtick={1,100,10000,1e6,1e8,1e10},
				ytick={1e-3,1e-2,1e-1,1, 10},
				legend style={at={(0.12,1)},anchor= north west},
				ymajorgrids=true,
				grid style=dotted,
				]
 \addplot[color=npurple, mark=triangle] coordinates {
(1.0, 0.7809805633484082)
(10.0, 0.7809805633484082)
(100.0, 0.7809805633484082)
(1000.0, 0.7809805633484082)
(10000.0, 0.7809805633484082)
(100000.0, 0.7809805633484082)
(1000000.0, 0.7809805633484082)
(10000000.0, 0.7809805633484082)
(100000000.0, 0.7809805633484082)
(1000000000.0, 0.7809805633484082)
};
\addlegendentry{6 elements}

\addplot[color=violet, mark=o] coordinates {
(1.0, 6.500914561778541)
(10.0, 0.2359666797289093)
(100.0, 0.21703307018556786)
(1000.0, 0.2170227170319103)
(10000.0, 0.21702263244761005)
(100000.0, 0.217022631941475)
(1000000.0, 0.21702259200512808)
(10000000.0, 0.2170307188758028)
(100000000.0, 0.21779832206586303)
(1000000000.0, 1.7553209187973509)
};
\addlegendentry{48 elements}

\addplot[color=teal, mark=square] coordinates {
(1.0, 8.714126019947647)
(10.0, 0.13889519687670826)
(100.0, 0.05628020414721062)
(1000.0, 0.05626151321034957)
(10000.0, 0.0562614699569357)
(100000.0, 0.0562614695636524)
(1000000.0, 0.05626147370147688)
(10000000.0, 0.06170910596090014)
(100000000.0, 0.056352782763627714)
(1000000000.0, 0.05634603142888711)
};
\addlegendentry{384 elements}

\addplot[color=blue, mark=pentagon] coordinates {
(1.0, 9.39338974325679)
(10.0, 0.13932228788605225)
(100.0, 0.014280520965228012)
(1000.0, 0.014210711223422406)
(10000.0, 0.01421069169502626)
(100000.0, 0.014210691515335175)
(1000000.0, 0.01421069463028081)
(10000000.0, 32.49333292681344)
(100000000.0, 0.014214890608183837)
(1000000000.0, 0.014214910906482829)
};
\addlegendentry{3072 elements}

\addplot[color=cyan, mark=diamond] coordinates {
(1.0, 9.569771087088784)
(10.0, 0.14175687049251734)
(100.0, 0.003838814167414234)
(1000.0, 0.0035624896381953905)
(10000.0, 0.003562457673030267)
(100000.0, 0.0035624576232439523)
(1000000.0, 0.06941871562498571)
(10000000.0, 17.855184677065807)
(100000000.0, 211.91354461588156)
(1000000000.0, 9.917069002209198)
};
\addlegendentry{24576 elements}

                    \addplot[dashed,color=black, mark=none]
				coordinates {
					(1.5, 2)
					(40, 0.0028125)
				};
				
		\end{loglogaxis}
			\draw (0.0,0.0) node[anchor=south west]{$\mathcal{O}(\Lc^{-2})$};
		\end{tikzpicture}
	\end{subfigure}
	\begin{subfigure}{0.48\linewidth}
		\begin{tikzpicture}
		\definecolor{npurple}{rgb}{0.4980392156862745,0.,1.}
			\begin{loglogaxis}[
				/pgf/number format/1000 sep={},
				axis lines = left,
				xlabel={$\Lc$ (Log)},
				ylabel={$\| \bm{P} - \widetilde{\bm{P}}\|_{\Le}/\|\widetilde{\bm{P}}\|_{\Le}$ (Log)},
				xmin=0.5, xmax=3e10,
				ymin=1e-3, ymax=1000,
				xtick={1,100,10000,1e6,1e8,1e10},
				ytick={1e-3,1e-2,1e-1,1, 10},
				legend style={at={(0.2,1)},anchor= north west},
				ymajorgrids=true,
				grid style=dotted,
				]
				\addplot[color=npurple, mark=triangle] coordinates {
(1.0, 0.7809805633484082)
(10.0, 0.7809805633484082)
(100.0, 0.7809805633484082)
(1000.0, 0.7809805633484082)
(10000.0, 0.7809805633484082)
(100000.0, 0.7809805633484082)
(1000000.0, 0.7809805633484082)
(10000000.0, 0.7809805633484082)
(100000000.0, 0.7809805633484082)
(1000000000.0, 0.7809805633484082)
};
\addlegendentry{6 elements}
\addplot[color=violet, mark=o] coordinates {
(1.0, 6.5009145617784325)
(10.0, 0.2359666797289086)
(100.0, 0.2170330701855676)
(1000.0, 0.2170227170319554)
(10000.0, 0.21702263245030734)
(100000.0, 0.21702263160638607)
(1000000.0, 0.21702263159794716)
(10000000.0, 0.21702263159786275)
(100000000.0, 0.21702263159786195)
(1000000000.0, 0.2170226315978619)
};
\addlegendentry{48 elements}
\addplot[color=teal, mark=square] coordinates {
(1.0, 8.7141260199475)
(10.0, 0.13889519687670365)
(100.0, 0.05628020414721067)
(1000.0, 0.05626151321035158)
(10000.0, 0.0562614699580147)
(100000.0, 0.05626146953986233)
(1000000.0, 0.05626146953568219)
(10000000.0, 0.0562614695356404)
(100000000.0, 0.05626146953563997)
(1000000000.0, 0.05626146953563999)
};
\addlegendentry{384 elements}
\addplot[color=blue, mark=pentagon] coordinates {
(1.0, 9.393389743256703)
(10.0, 0.139322287886051)
(100.0, 0.014280520965230262)
(1000.0, 0.014210711223420084)
(10000.0, 0.014210691694932826)
(100000.0, 0.014210691567688038)
(1000000.0, 0.014210691566422368)
(10000000.0, 0.014210691566409715)
(100000000.0, 0.014210691566409597)
(1000000000.0, 0.014210691566409575)
};
\addlegendentry{3072 elements}
\addplot[color=cyan, mark=diamond] coordinates {
(1.0, 9.56977108708868)
(10.0, 0.1417568704924836)
(100.0, 0.0038388141674141627)
(1000.0, 0.003562489638201365)
(10000.0, 0.0035624576731761444)
(100000.0, 0.003562457637421535)
(1000000.0, 0.003562457637092395)
(10000000.0, 0.003562457637089097)
(100000000.0, 0.0035624576370890757)
(1000000000.0, 0.003562457637089065)
};
\addlegendentry{24576 elements}

\addplot[dashed,color=black, mark=none]
				coordinates {
					(1.5, 2)
					(40, 0.0028125)
				};
				
		\end{loglogaxis}
			\draw (0.0,0.0) node[anchor=south west]{$\mathcal{O}(\Lc^{-2})$};
		\end{tikzpicture}
	\end{subfigure}
	
	\caption{Convergence rates for $\Lc \to \infty$ with quadratic sequence for primal (left) and mixed (right) method with $1\times1\times1$, $2\times2\times2$, $\dots$, $16\times16\times16$ grids consisting of 6, 48, 384, 3072, and 24576 elements, respectively.}
	\label{fig:conv_lc_inf}
\end{figure}

\subsection{Comparison to the Cauchy continuum}\label{ssec:cauchy}
Considering \cref{eq:lame}, we recognize the meso parameters $\mathbb{C}_\mathrm{e}$ can be retrieved from the micro and macro parameters
\begin{align}
    \mue = \dfrac{\mumi\, \muma}{\mumi- \muma} \, , && 2\mue + 3 \lambda_\mathrm{e} = \dfrac{(2\mumi + 3\lambda_\mathrm{micro})(2\muma + 3\lambda_\mathrm{macro})}{(2\mumi + 3\lambda_\mathrm{micro})-(2\muma + 3\lambda_\mathrm{macro})} \, .
\end{align}
Assuming micro to be always stiffer than macro, we set the macro and micro parameters
\begin{align}
    \lambda_\mathrm{macro} = 115.4 \, , && \muma = 76.9 \, , && \lambda_\mathrm{micro} = 10 \, \lambda_\mathrm{macro} = 1154 \, , && \mumi = 10 \muma = 769 \, , 
\end{align}
and retrieve the meso parameters
\begin{align}
    \lambda_\mathrm{e} = 128.2 \, , && \mue = 85.4 \, .
\end{align}
Finally, we set $\muc = \mue$. The characteristic length $\Lc$ acts as a scaling parameter between highly homogeneous materials and materials with a pronounced micro-structure. Namely, for $\Lc = 0$ the continuum yields the result of classical linear elasticity with the macro parameters. 
We define the domain $\Omega = [-3,\,3] \times [-1,\,1]^2$, illustrating a bending beam (see \cref{fig:cauchy2}), and apply the constant volume load $\vb{f} = \begin{bmatrix}
    0 & 0 & -10
\end{bmatrix}^T$. The Dirichlet boundary conditions  
\begin{align}
    \widetilde{\vb{u}} = \vb{0} \, , &&  \widetilde{\Pm} \times \bm{\nu} = \D \widetilde{\vb{u}} \times \bm{\nu} = \bm{0} \, ,
\end{align}
are applied at $x = -3$ and $x = 3$.
The characteristic length $\Lc$ is varied between $10^3$ and $10^{-3}$ and the deviation from the linear elastic response of the Cauchy continuum is measured by $\|\vb{u}_\mathrm{Cauchy} - \vb{u}_\mathrm{relaxed}\|_{\Le}$ and the internal energy of each formulation. The internal energy of the relaxed micromorphic continuum is extracted from \cref{eq:1} and reads
\begin{align}
     I_\mathrm{relaxed} = \dfrac{1}{2} \int_{\Omega} &\langle \Ce \sym(\D \vb{u} - \bm{P}) , \, \sym(\D \vb{u} - \bm{P}) \rangle
		+  \langle \Cm \sym\bm{P} , \, \sym\bm{P} \rangle \notag \\ 
		& + \langle \Cc \skw(\D\vb{u} - \bm{P}) , \, \skw (\D \vb{u} - \bm{P}) \rangle
		+ \muma \Lc^2 \, \| \text{Curl}\bm{P} \|^2 \, \dd X \, .
\end{align}
For isotropic linear elasticity one finds
\begin{align}
    I_\mathrm{Cauchy} = \int_\Omega \dfrac{1}{2} \langle \mathbb{C}_\mathrm{macro} \, \sym \D \vb{u} , \, \sym \D \vb{u} \rangle \, \dd X \, , && \mathbb{C}_\mathrm{macro} = 2 \mu_{\mathrm{macro}}\Sy + \lambda_{\mathrm{macro}} \one \otimes \one \, .
\end{align}
The behaviour of the deviation in both the displacement and the energy in \cref{fig:cauchy} and \cref{fig:cauchy2} demonstrates the derivation in \cref{eq:tocauchy}. Further, it is apparent that this characteristic strongly depends on the approximation capacity of the finite subspace. In other words, the use of p-refinement greatly influences the result. Decreasing the element size (h-refinement) can also be used to alleviate the error. For example taking $19264$ elements with $87742$ degrees of freedom in the linear sequence with $\Lc = 10^{-3}$ results in
\begin{align}
    \dfrac{\| \vb{u}_\mathrm{Cauchy} - \vb{u}_\mathrm{relaxed}\|_{\Le}}{\| \vb{u}_\mathrm{Cauchy}\|_{\Le}}   = 0.095 \, , && I_\mathrm{Cauchy} = 81.549 \, , && I_\mathrm{relaxed} = 73.333 \, ,
\end{align}
improving over the result with $1031$ elements and $5890$ degrees of freedom
\begin{align}
    \dfrac{\| \vb{u}_\mathrm{Cauchy} - \vb{u}_\mathrm{relaxed}\|_{\Le}}{\| \vb{u}_\mathrm{Cauchy}\|_{\Le}}  = 0.289 \, , && I_\mathrm{Cauchy} = 72.922 \, , && I_\mathrm{relaxed} = 51.588 \, .
\end{align}
However, this does not compare to the order of improvement achieved with $1031$ elements with $15688$ degrees of freedom of the quadratic sequence
\begin{align}
    \dfrac{\| \vb{u}_\mathrm{Cauchy} - \vb{u}_\mathrm{relaxed}\|_{\Le}}{\| \vb{u}_\mathrm{Cauchy}\|_{\Le}}  = 0.035 \, , && I_\mathrm{Cauchy} = 82.932 \, , && I_\mathrm{relaxed} = 79.62 \, .
\end{align}
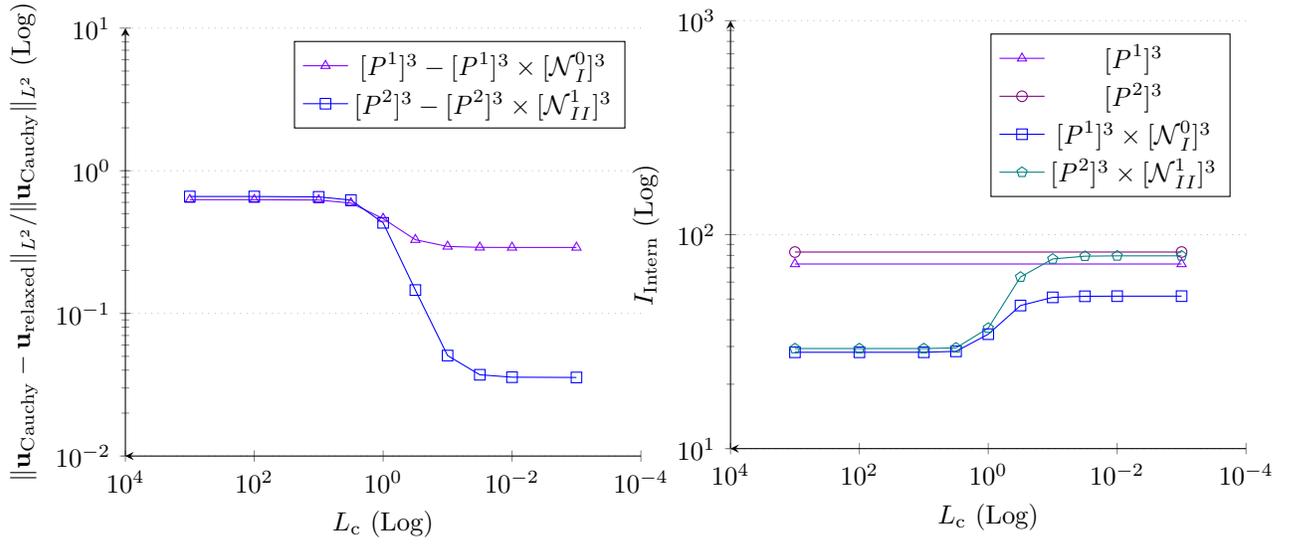
\begin{figure}
	\centering
	\begin{subfigure}{0.48\linewidth}
		\begin{tikzpicture}
		\definecolor{npurple}{rgb}{0.4980392156862745,0.,1.}
			\begin{loglogaxis}[
				/pgf/number format/1000 sep={},
				axis lines = left,
				xlabel={$\Lc$ (Log)},
				ylabel={$\| \vb{u}_\mathrm{Cauchy} - \vb{u}_\mathrm{relaxed}\|_{\Le} / \| \vb{u}_\mathrm{Cauchy}\|_{\Le}$ (Log)},
				xmin=1e-4, xmax=1e+4,
				ymin=1e-2, ymax=10,
				x dir=reverse,
				xtick={1e-4, 1e-2, 1, 1e+2, 1e+4},
				ytick={1e-2,1e-1,1, 10},
				legend pos= north east,
				ymajorgrids=true,
				grid style=dotted,
				]
				\addplot[
				color=npurple,
				mark=triangle,
				]
				coordinates {
					(1000, 2.120475577330363/ 3.3818612460148003)
					(100, 2.1203473397469788/ 3.3818612460148003)
					(10, 2.1077130530787125/ 3.3818612460148003)
					(3.1622776601683795, 2.0073775294829765/ 3.3818612460148003)
					(1, 1.5611046978271974/ 3.3818612460148003)
					(0.31622776601683794, 1.1107075851042856/ 3.3818612460148003)
					(1e-1, 0.9963905812264311/ 3.3818612460148003)
					(0.03162277660168379, 0.9819082282620919/ 3.3818612460148003)
					(1e-2, 0.9804104786991784/ 3.3818612460148003)
					(1e-3, 0.9802451336003354 / 3.3818612460148003)
				};
				\addlegendentry{$[\Po^1]^3 - [\Po^1]^3 \times [\Ned^0_{I}]^3$}
				\addplot[
				color=blue,
				mark=square,
				]
				coordinates {
					(1000, 2.5498689887854273/ 3.863220865924102)
					(100, 2.549714880877904/ 3.863220865924102)
					(10, 2.534432935238876/ 3.863220865924102)
					(3.1622776601683795, 2.4060727600330396/ 3.863220865924102)
					(1, 1.6665134891626427/ 3.863220865924102)
					(0.31622776601683794, 0.5630689696259054/ 3.863220865924102)
					(1e-1, 0.19564449196996236/ 3.863220865924102)
					(0.03162277660168379, 0.14345267469496642/ 3.863220865924102)
					(1e-2, 0.13794251994805096/ 3.863220865924102)
					(1e-3, 0.1373328751199808 / 3.863220865924102)
				};
				\addlegendentry{$[\Po^2]^3 - [\Po^2]^3 \times [\Ned^1_{II}]^3$}
			\end{loglogaxis}
		\end{tikzpicture}
	\end{subfigure}
	\begin{subfigure}{0.48\linewidth}
		\begin{tikzpicture}
		\definecolor{npurple}{rgb}{0.4980392156862745,0.,1.}
			\begin{loglogaxis}[
				/pgf/number format/1000 sep={},
				axis lines = left,
				xlabel={$\Lc$ (Log)},
				ylabel={$I_\mathrm{Intern}$ (Log)},
				xmin=1e-4, xmax=1e+4,
				ymin=10, ymax=1000,
				x dir=reverse,
				xtick={1e-4, 1e-2, 1, 1e+2, 1e+4},
				ytick={10,100, 1000},
				legend pos= north east,
				ymajorgrids=true,
				grid style=dotted,
				]
				\addplot[
				color=npurple,
				mark=triangle,
				]
				coordinates {
				    (1000, 72.92221764486239)
					(1e-3, 72.92221764486239)
				};
				\addlegendentry{$[\Po^1]^3$}
				\addplot[
				color=violet,
				mark=o,
				]
				coordinates {
					(1000, 82.93296273215226)
					(1e-3, 82.93296273215226)
				};
				\addlegendentry{$[\Po^2]^3$}
				\addplot[
				color=blue,
				mark=square,
				]
				coordinates {
				    (1000, 28.1969)
					(100, 28.1969)
					(10, 28.2007)
					(3.1622776601683795, 28.4785)
					(1, 34.2654)
					(0.31622776601683794, 46.6365)
					(1e-1, 50.9192)
					(0.03162277660168379, 51.5188)
					(1e-2, 51.5819)
					(1e-3, 51.5889)
				};
				\addlegendentry{$[\Po^1]^3 \times [\Ned^0_{I}]^3$}
				\addplot[
				color=teal,
				mark=pentagon,
				]
				coordinates {
				    (1000, 29.3596)
					(100, 29.3596)
					(10, 29.3622)
					(3.1622776601683795, 29.5733)
					(1, 36.6003)
					(0.31622776601683794, 63.3885)
					(1e-1, 76.9858)
					(0.03162277660168379, 79.3285)
					(1e-2, 79.5911)
					(1e-3, 79.6204)
				};
				\addlegendentry{$[\Po^2]^3 \times [\Ned^1_{II}]^3$}
			\end{loglogaxis}
		\end{tikzpicture}
	\end{subfigure}
	\caption{Displacement and energy comparison with the Cauchy continuum model with 1031 elements.}
	\label{fig:cauchy}
\end{figure}
\begin{figure}
	\centering
	\begin{subfigure}{0.3\linewidth}
	\centering
		\includegraphics[width=1.0\linewidth]{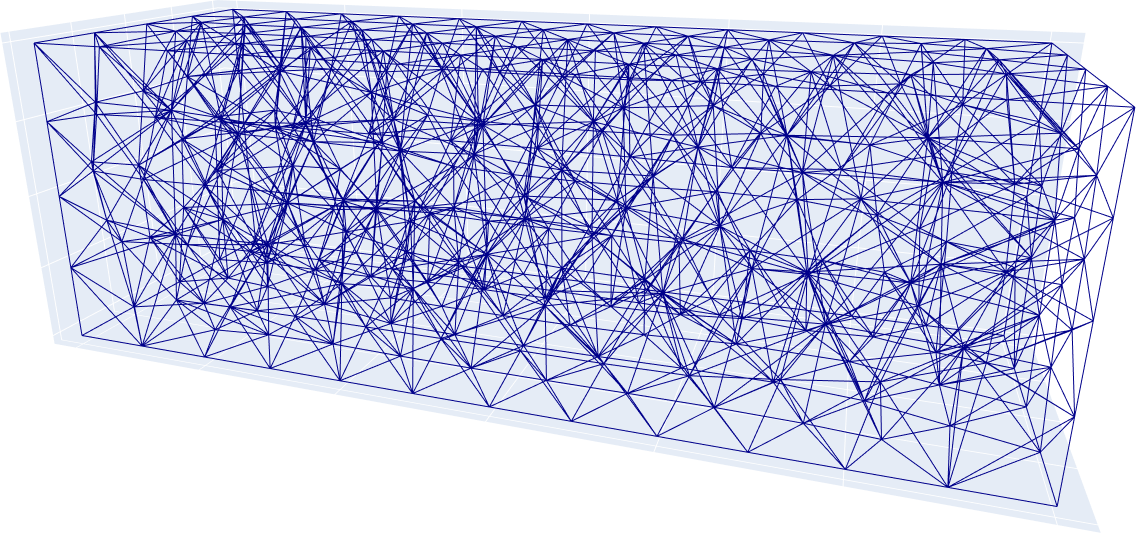}
		\caption{} 
	\end{subfigure}
	\begin{subfigure}{0.3\linewidth}
	\centering
		\includegraphics[width=1.0\linewidth]{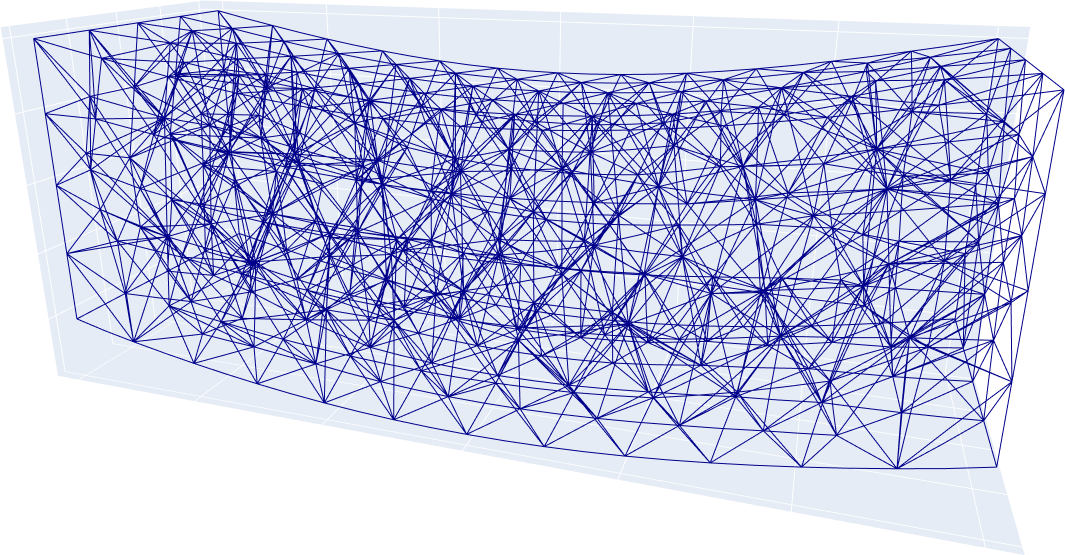}
		\caption{}
	\end{subfigure}
	\begin{subfigure}{0.3\linewidth}
	\centering
		\includegraphics[width=1.0\linewidth]{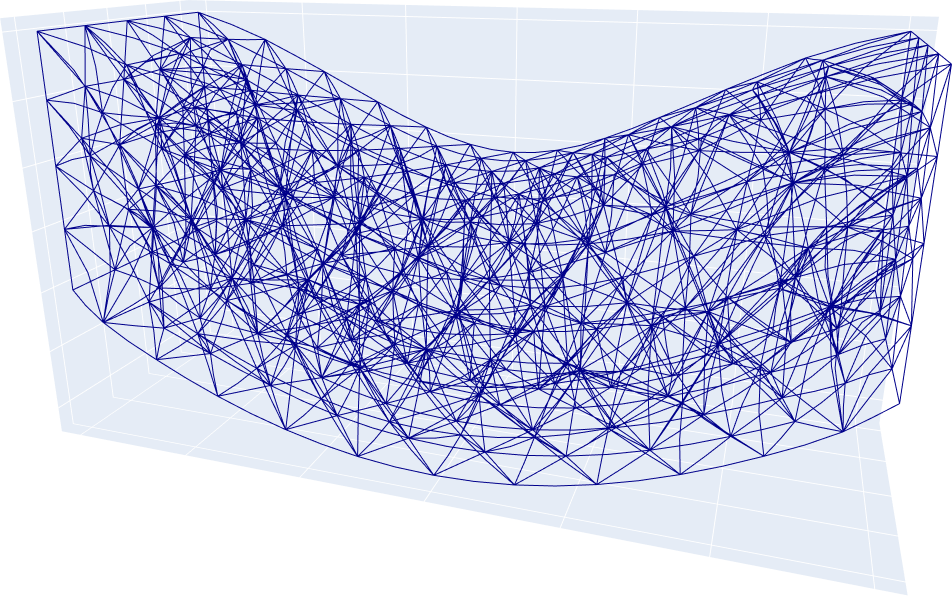}
		\caption{}
	\end{subfigure}
	\caption{(a) Initial geometry with 1031 elements, (b) displacement field of the quadratic element for $\Lc = 10^3$, (c) displacement field of the quadratic element for $\Lc = 10^{-3}$.}
	\label{fig:cauchy2}
\end{figure}

\subsection{Bounded stiffness}\label{ssec:bounded}
In the example in \cref{ssec:cauchy} one can observe the influence of the characteristic length $\Lc$ on the reaction of the system. In fact, the characteristic length acts as an interpolation parameter between the micro-stiffness and the macro-stiffness, as shown in \cref{re:tocmicro} and \cref{sec:lczero}. In order to demonstrate this phenomenon, we employ the axis-symmetric cubic domain $\Omega=[-1,\,1]^3$ and the material parameters from the previous example. The characteristic length $\Lc$ is varied in $[10^{-3}, \, 10^3]$. The Dirichlet boundary is set to 
\begin{align}
    &\widetilde{\vb{u}} = \begin{bmatrix}
        1+z \\ 0 \\ 0
    \end{bmatrix} \, , && \widetilde{\Pm} = \D \widetilde{\vb{u}} = \begin{bmatrix}
        0 & 0 & 1 \\
        0 & 0 & 0 \\
        0 & 0 & 0 \\
    \end{bmatrix} \, , && \Gamma_D = \partial\Omega \cap \{\vb{x} \in \partial\Omega \; | \; z = \pm 1\} \, .
\end{align}
Note that while the problem is reminiscent of simple shear \cite{Rizzi_shear}, it is different since we are using a finite volume and only the upper and lower parts of the boundary are constrained. The problem is defined entirely as a boundary value problem and no body forces or micro-moments are applied. In order to extract an indicator for the stiffness of the system we measure the
Lebesgue integral of the displacement reaction forces in x-direction on the upper Dirichlet boundary
\begin{align}
    & r_x = \int_{\Gamma_D^1}  \vb{t}_x \, \dd X   \, , && \Gamma_D^1 = \Gamma_D \cap \{\vb{x} \in \partial\Omega \; | \; z = 1\} \, ,
\end{align}
where $\vb{t}_x$ represents the equivalent traction in x-direction (see \cref{ap:B}).

Considering \cref{fig:boundedstiff}, it is clear that the linear finite element formulation largely overestimates the stiffness of the system by about a factor of $28\%$ in the lower limit for the finer discretization. Further, the effects of mesh refinement are considerable. From the depiction of the displacement field we recognize a transition from a nearly linear displacement field for large $\Lc$ values to a higher order displacement function for low $\Lc$ values. Lastly, we observe that the intensity of the microdistortion field (first row of $\Pm$) shifts from the centre to the sides of the domain for a decreasing characteristic length $\Lc$.

\begin{remark}
    The lower limit $\mathbb{C}_\mathrm{macro}$ was estimated using $6000$ finite elements of the quadratic sequence in the \textbf{primal} formulation while setting the characteristic length to $\Lc = 0$. 
    In order to approximate the upper limit $\Cm$, we make use of the quadratic sequence in the \textbf{mixed} formulation with $\Lc = 10^9$ and $6000$ elements, since the mixed formulation is stable for very large values of the characteristic length.
\end{remark}

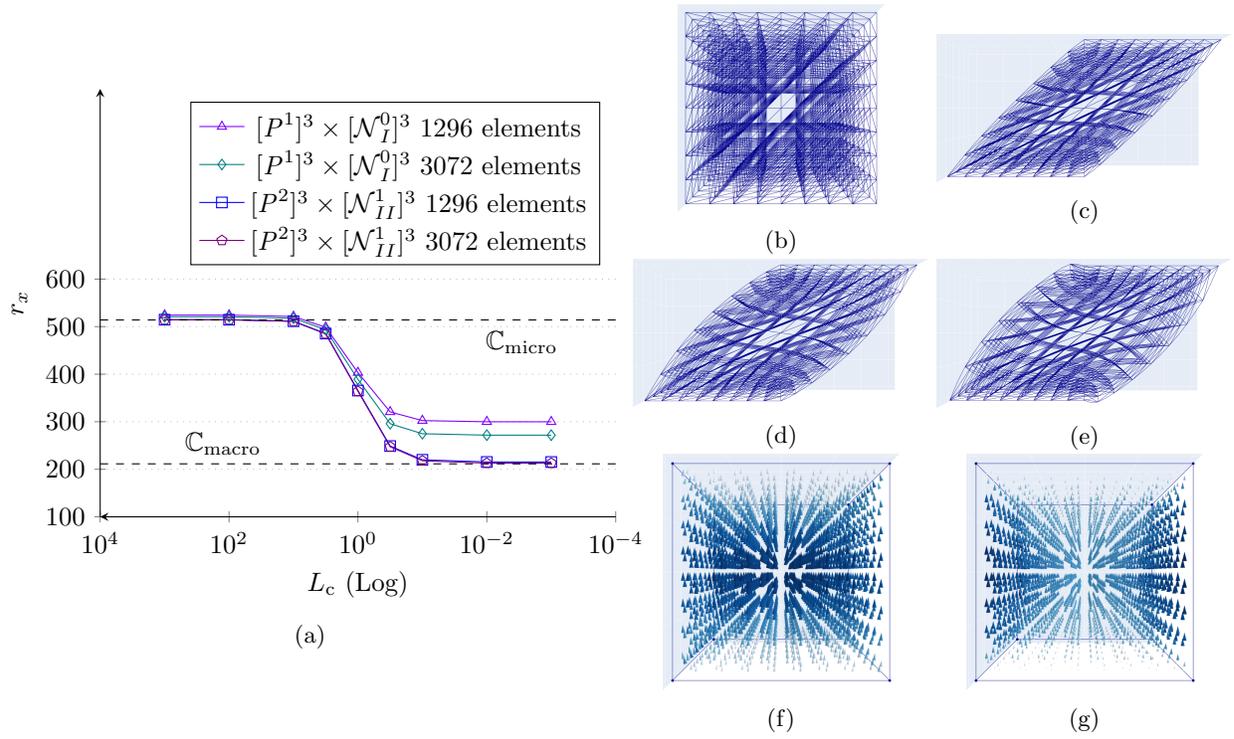
\begin{figure}
    \centering
    \begin{subfigure}{0.48\linewidth}
	\centering
		\begin{tikzpicture}
		\definecolor{npurple}{rgb}{0.4980392156862745,0.,1.}
			\begin{semilogxaxis}[
				/pgf/number format/1000 sep={},
				axis lines = left,
				xlabel={$\Lc$ (Log)},
				ylabel={$r_x$},
				xmin=1e-4, xmax=1e+4,
				ymin=100, ymax=1000,
				x dir=reverse,
				xtick={1e-4, 1e-2, 1, 1e+2, 1e+4},
				ytick={100, 200,300, 400,500, 600},
				legend pos= north east,
				ymajorgrids=true,
				grid style=dotted,
				]
				\addplot[
				color=npurple,
				mark=triangle,
				]
				coordinates {
					(1000, 524.2246632493815) (100, 524.1954713321225) (10, 521.3208250090676) (3.1622776601683795, 498.6456938715164) (1, 403.3567628527211) (0.31622776601683794, 320.64492689528964) (0.1, 302.40300170156945) (0.01, 299.90129779994214) (0.001, 299.87534467312844)
				};
				\addlegendentry{$[\Po^1]^3 \times [\Ned^0_{I}]^3$ 1296 elements}
				\addplot[
				color=teal,
				mark=diamond,
				]
				coordinates {
					(1000, 520.7731466469695) (100, 520.7423268221748) (10, 517.7046632849662) (3.1622776601683795, 493.5697126029215) (1, 389.27350339497895) (0.31622776601683794, 295.82126739389986) (0.1, 274.74582996196386) (0.01, 271.75433882552886) (0.001, 271.7228173988496)
				};
				\addlegendentry{$[\Po^1]^3 \times [\Ned^0_{I}]^3$ 3072 elements}
				\addplot[
				color=blue,
				mark=square,
				]
				coordinates {
					(1000, 514.6406897962419) (100, 514.6081364383472) (10, 511.39376200460856) (3.1622776601683795, 485.4431177606226) (1, 365.5967508281445) (0.31622776601683794, 248.81729004622335) (0.1, 219.9216672315428) (0.01, 215.36114895167944) (0.001, 215.31129270609406)
				};
				\addlegendentry{$[\Po^2]^3 \times [\Ned^1_{II}]^3$ 1296 elements}
				\addplot[
				color=violet,
				mark=pentagon,
				]
				coordinates {
					(1000, 514.2648057577406) (100, 514.2317626733782) (10, 510.9694336213919) (3.1622776601683795, 484.6673274164019) (1, 363.96566314459443) (0.31622776601683794, 247.34046638459458) (0.1, 217.8793005193988) (0.01, 212.87278280585573) (0.001, 212.81546001093588)
				};
				\addlegendentry{$[\Po^2]^3 \times [\Ned^1_{II}]^3$ 3072 elements}
				\addplot[dashed,color=black, mark=none]
				coordinates {
					(1e+4, 514.0757380591552)
					(1e-4, 514.0757380591552)
				};
				\addplot[dashed,color=black, mark=none]
				coordinates {
					(1e+4, 211.4239248393674)
					(1e-4, 211.4239248393674)
				};
			\end{semilogxaxis}
			\draw (5.,2.6) node[anchor=north west]{$\Cm$};
			\draw (1.,1.25) node[anchor=north west]{$\mathbb{C}_\mathrm{macro}$};
		\end{tikzpicture}
		\caption{}
\end{subfigure}
\begin{subfigure}{0.48\linewidth}
   \centering
   \begin{subfigure}{0.48\linewidth}
   \centering
   \includegraphics[width=0.7\linewidth]{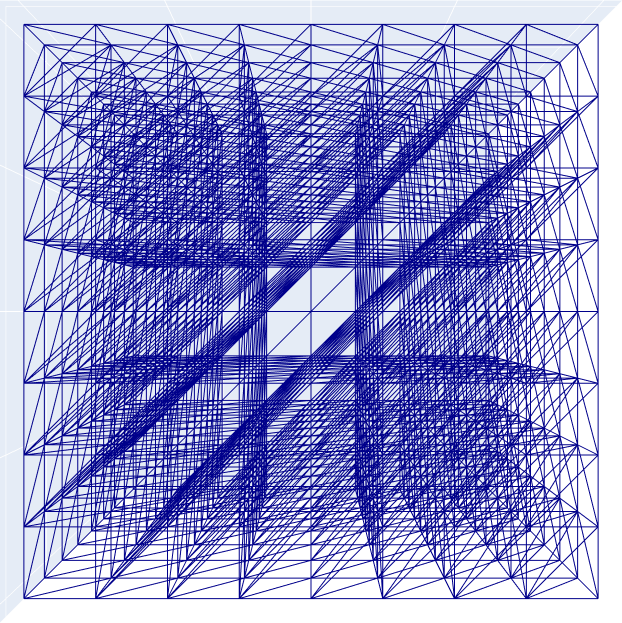}
   \caption{}
   \end{subfigure}
   \begin{subfigure}{0.48\linewidth}
   \centering
   \includegraphics[width=1.0\linewidth]{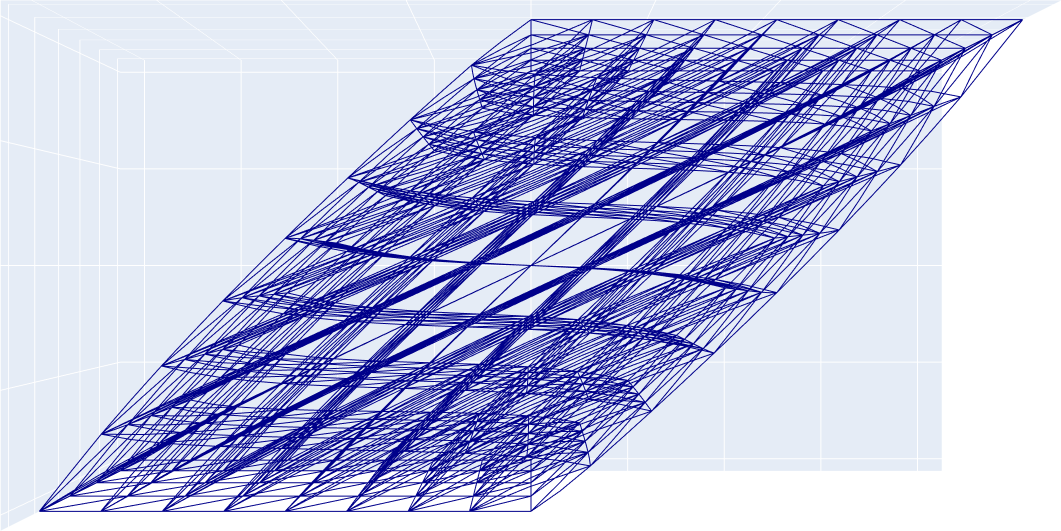}
   \caption{}
   \end{subfigure}
   \begin{subfigure}{0.48\linewidth}
   \centering
   \includegraphics[width=1.0\linewidth]{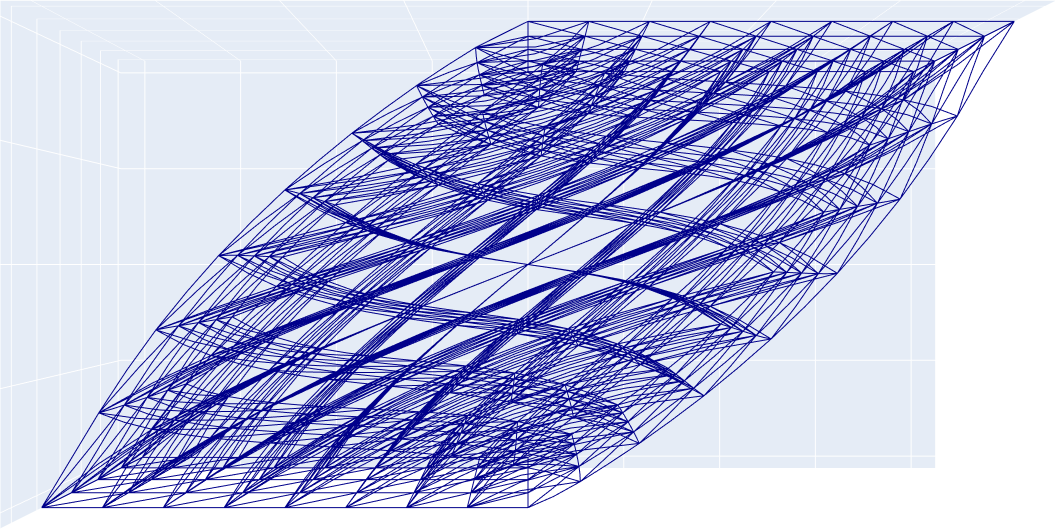}
   \caption{}
   \end{subfigure}
   \begin{subfigure}{0.48\linewidth}
   \centering
   \includegraphics[width=1.0\linewidth]{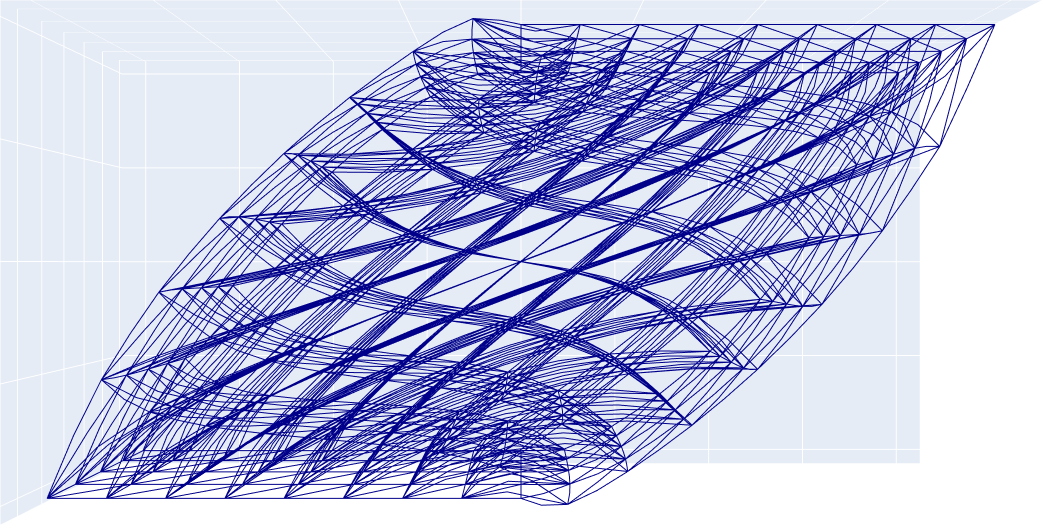}
   \caption{}
   \end{subfigure}
   \begin{subfigure}{0.48\linewidth}
   \centering
   \includegraphics[width=0.8\linewidth]{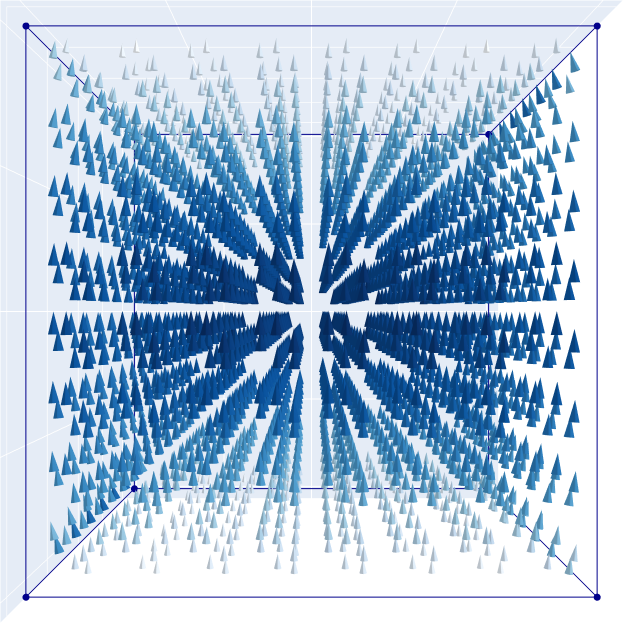}
   \caption{}
   \end{subfigure}
   \begin{subfigure}{0.48\linewidth}
   \centering
   \includegraphics[width=0.8\linewidth]{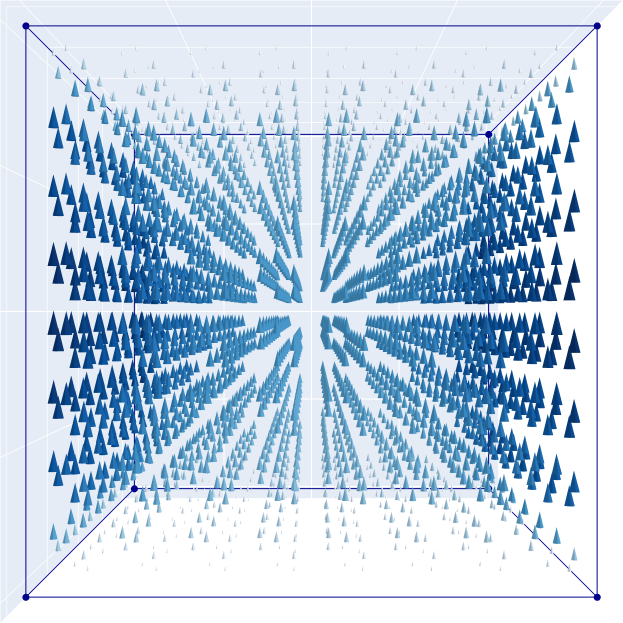}
   \caption{}
   \end{subfigure}
\end{subfigure}
\caption{(a) Reaction forces with varying characteristic lengths $\Lc$, (b) initial geometry with 3072 elements, (c) cutout view of the displacement field of the quadratic element with $\Lc = 10^3$, (d) cutout view of the displacement field of the quadratic element with $\Lc = 1$, (d) cutout view of the displacement field of the quadratic element with $\Lc = 10^{-3}$, (f) microdistortion field of the quadratic element with $\Lc = 10^3$, (g) microdistortion field of the quadratic element with $\Lc = 10^{-3}$.
}
\label{fig:boundedstiff}
\end{figure}

\section{Conclusions and outlook}\label{ch:8}
Existence and uniqueness in the primal formulation of the relaxed micromorphic model requires the employment of N\'ed\'elec elements, since $\Ned_I^p \, , \Ned_{II}^p \subset \Hc{}$. The construction of the lowest order N\'ed\'elec elements was demonstrated along with a solution to the arising orientation problem, such that no correction functions are necessary and a single reference element suffices.

While not being a mixed formulation, the use of the exact de Rham sequence is still recommended in order to exactly satisfy the discrete consistent coupling condition, as shown in \cref{ch:consist}. This is otherwise not possible for the general case if the exact sequence is not respected. 

The example in \cref{ssec:cauchy} depicted the behaviour of the relaxed micromorphic theory with $\Lc \to 0$ as given by the derivation in \cref{eq:tocauchy}.
Clearly, the characteristic length $\Lc$ is a crucial component of the theory as it governs the relation to the classical Cauchy continuum. Further, it is demonstrated that the latter comparison is strongly dependent on the approximation capacity of the finite subspace. Due to the relatively large error with respect to the expected result using the linear finite element sequence, $\| \vb{u}_\mathrm{Cauchy} - \vb{u}_\mathrm{relaxed}\|_{\Le} / \| \vb{u}_\mathrm{Cauchy}\|_{\Le} \geq 9.5 \% $, it is recommended to use at least quadratic or higher order finite elements. This conclusion is further supported by the example in \cref{ssec:bounded}, where the linear formulation largely overestimated the expected stiffness of the system for $\Lc \to \infty$.

While the primal formulation is relatively robust in $\Lc$, it may become unstable as $\Lc \to \infty$. As demonstrated in the example in \cref{ch:robust}, the computation can be re-stabilised by using a mixed formulation. As per the de Rham complex, this requires the usage of $\Hd{}$ elements, such as Raviart-Thomas $\RT$ or Brezzi-Douglas-Marini $\BDM$, and completely discontinuous finite elements. Further, for efficient simulations with higher order elements it is advantageous to split the $\Hd{}$ elements into source and solenoidal parts as the hyperstress field $\bm{D}$ is completely solenoidal, meaning the source part of the interpolant should be dropped or must be otherwise compensated. The construction of the lowest order Raviart-Thomas elements along with the lowest order discontinuous elements was demonstrated in \cref{ch:5} and \cref{ch:dis}, respectively. The orientation problem for the latter elements is circumvented by the same methodology as for the N\'ed\'elec elements. 

For further development we note the computation time increases rapidly with more elements and higher order polynomials. The development of an efficient preconditioner is to be considered for future works.


\section*{Acknowledgements}
Patrizio Neff acknowledges support in the framework of the DFG-Priority Programme
2256 “Variational Methods for Predicting Complex Phenomena in Engineering Structures and Materials”, Neff 902/10-1, Project-No.
440935806. Michael Neunteufel and Joachim Sch\"oberl acknowledge support by the Austrian Science Fund (FWF) project F65.
\label{sec:acknowledgements}

\bibliographystyle{spmpsci}   

\bibliography{Ref}   

\appendix
\section{Derivation of the strong form and analytical solutions} \label{ap:A}
The variation of the free energy functional with respect to the displacement field $\vb{u}$ reads
\begin{align}
    \delta_u I = \int_\Omega \langle \Ce \sym \D \delta  \vb{u} , \, \sym(\D \vb{u} - \Pm) \rangle + \langle \Cc \skw \D \delta  \vb{u} , \, \skw(\D \vb{u} - \Pm) \rangle - \langle \delta \vb{u} ,\, \vb{f} \rangle \, \dd X  \qquad \forall \, \delta \vb{u} \in \C^1_{\Gamma_D}(\Omega,\mathbb{R}^3) \, ,
\end{align}
where $\C^1_{\Gamma_D}(\Omega,\mathbb{R}^3)$ denotes the set of differentiable functions which are zero at the Dirichlet boundary.

By using the Green-type identity 
\begin{align}
    \int_\Omega \langle \D \vb{v} , \, \bm{T} \rangle \, \dd X &= \int_{\partial \Omega}  \langle \vb{v}, \, \bm{T} \bm{\nu}  \rangle \,  \dd A  - \int_{\Omega} \langle \vb{v} ,\, \Di \bm{T} \rangle \, \dd X \, , && \vb{v} \in \C^1(\Omega,\mathbb{R}^3) \, , \bm{T} \in \C^1(\Omega,\mathbb{R}^{3 \times 3}) \, ,
\end{align}
and splitting the boundary $\partial \Omega = \Gamma_D \cup \Gamma_N$, such that $\Gamma_D \cap \Gamma_N = \emptyset$, one finds
\begin{align}
    \delta_u I =& \int_{\Gamma_N^u} \langle \delta \vb{u} ,\, [\Ce \sym (\D \vb{u}- \bm{P}) + \Cc \skw (\D \vb{u} - \bm{P})] \, \bm{\nu} \rangle \, \dd A \notag \\ &-\int_\Omega \langle \delta \vb{u} , \, \Di[\Ce \sym (\D \vb{u} - \bm{P}) + \Cc \skw (\D \vb{u} - \bm{P})] - \vb{f} \rangle  \, \dd X  \, , \qquad \forall \, \delta \vb{u} \in \C^1_{\Gamma_D}(\Omega,\mathbb{R}^3) \, ,
    \label{eq:trac}
\end{align}
where $\vb{u}$ on $\Gamma_D^u$ is directly embedded in the space.

The variation of the energy with respect to the microdistortion field $\Pm$ reads
\begin{align}
    \delta_P I = \int_\Omega & - \langle \Ce \sym \delta \bm{P} , \, \sym(\D \vb{u} - \Pm) \rangle - \langle \Cc \skw \delta \bm{P} , \, \skw(\D \vb{u} - \Pm) \rangle + \langle \Cm \sym \delta \bm{P} , \, \sym\Pm \rangle \notag \\ 
    &+ \muma \Lc^2 \, \langle \Curl \delta \bm{P} , \, \Curl \Pm \rangle -\langle \delta \bm{P}, \, \bm{M} \rangle \, \dd X \, , \qquad \forall \, \delta \bm{P} \in \C^1_{\Gamma_D}(\Omega,\mathbb{R}^{3 \times 3}) \, .
\end{align}
Applying the Green-type identity for the Curl-operator
\begin{align}
    \int_\Omega \langle \Curl \bm{Q} , \, \bm{T} \rangle \, \dd X &= \int_{\partial \Omega} \langle \bm{Q} \, , \bm{T} \times \bm{\nu} \rangle \, \dd A + \int_\Omega \langle \bm{Q} , \, \Curl \bm{T} \rangle \, \dd X \, , && \bm{Q}, \, \bm{T} \in  \C^1(\Omega,\mathbb{R}^{3\times 3}) \, ,
\end{align}
along with the split of the boundary yields
\begin{align}
    \delta_P I =& -\int_\Omega \langle \delta \bm{P} \, , \Ce  \sym (\D \vb{u} - \Pm) + \Cc  \skw(\D \vb{u} - \Pm) - \Cm \sym \Pm - \muma \, \Lc ^ 2  \Curl\Curl\Pm + \bm{M} \rangle \, \dd X  \notag \\ &+ \muma \Lc^2 \int_{\Gamma_N^P} \langle \delta \bm{P} , \,  \Curl \Pm \times \,  \bm{\nu} \rangle \, \dd A \, , \qquad \forall \, \delta \bm{P} \in \C^1_{\Gamma_D}(\Omega,\mathbb{R}^{3 \times 3}) \, ,
\end{align}
where $\Pm \times \, \bm{\nu}$ on $\Gamma_D^P$ is directly embedded into the space. 

Consequently, the strong form reads
\begin{subequations}
\begin{align}
	-\Di[\Ce \sym (\D \vb{u} - \bm{P}) + \Cc \skw (\D \vb{u} - \bm{P})] &= \vb{f} && \text{in} \quad \Omega \, , \\
	-\Ce  \sym (\D \vb{u} - \Pm) - \Cc  \skw(\D \vb{u} - \Pm) + \Cm \sym \Pm + \muma \, \Lc ^ 2  \Curl\Curl\Pm &= \bm{M} && \text{in} \quad \Omega \, , \\
	\vb{u} &= \widetilde{\vb{u}} && \text{on} \quad \Gamma_D^u \, , \\
	\Pm \times \, \bm{\nu} &= \widetilde{\Pm} \times \bm{\nu} && \text{on} \quad \Gamma_D^P \, , \\
	[\Ce \sym (\D \vb{u}- \bm{P}) + \Cc \skw (\D \vb{u} - \bm{P})] \, \bm{\nu} &= 0 && \text{on} \quad \Gamma_N^u \, ,\\
	\Curl \Pm \times \, \bm{\nu}  &= 0  && \text{on} \quad \Gamma_N^P \, .
\end{align}
\end{subequations}

By pushing predefined displacement $\widetilde{\vb{u}}$ and microdistortion $\widetilde{\bm{P}}$ fields into the strong form we can derive corresponding forces and micro-moments
\begin{align}
    \vb{f} &= -\Di[\Ce \sym (\D \widetilde{\vb{u}} - \widetilde{\bm{P}}) + \Cc \skw (\D \widetilde{\vb{u}} - \widetilde{\bm{P}})] \, , \\
    \bm{M} &= -\Ce  \sym (\D \widetilde{\vb{u}} - \widetilde{\bm{P}}) - \Cc  \skw(\D \widetilde{\vb{u}} - \widetilde{\bm{P}}) + \Cm \sym \widetilde{\bm{P}} + \muma \, \Lc ^ 2  \Curl\Curl\widetilde{\bm{P}} \, .
\end{align}
Employing the latter forces and micro-moments in the domain while constraining the entire boundary with the prescribed fields
\begin{align}
    \Gamma_D = \partial \Omega \, , && \vb{u} \at_{\Gamma_D} = \widetilde{\vb{u}} \, , && \Pm \times \, \bm{\nu} \at_{\Gamma_D} = \widetilde{\Pm} \times \, \bm{\nu} \, ,
\end{align}
ensures the prescribed fields to be the analytical solution.

\section{Equivalent traction}\label{ap:B}
The traction on the Neumann boundary is defined using \cref{eq:trac} 
\begin{align}
    \vb{t} =  [\Ce \sym (\D \vb{u}- \bm{P}) + \Cc \skw (\D \vb{u} - \bm{P})] \, \bm{\nu} \, .
\end{align}
Consequently, the primal formulation can be rewritten with an equivalent traction on the Dirichlet boundary as 
\begin{align}
    a(\{\vb{u},\bm{P}\},\{\delta\vb{u},\delta\bm{P}\}) = l(\{\delta\vb{u},\delta\bm{P}\}) + \int_{\Gamma_D} \langle \delta\vb{u}, \, \vb{t} \rangle \, \dd A \qquad  \forall\, \{\delta\vb{u},\delta\bm{P}\}\in [\Hone(\Omega)]^3\times \HC{, \Omega} \, ,
\end{align}
from which the equivalent traction can be extracted \begin{align}
    \int_{\Gamma_D} \langle \delta\vb{u}, \, \vb{t} \rangle \, \dd A = a(\{\vb{u},\bm{P}\},\{\delta\vb{u},\delta\bm{P}\}) - l(\{\delta\vb{u},\delta\bm{P}\})\, .
\end{align}
The extraction is achieved by redefining the test function to unity on the upper Dirichlet boundary in the x-direction
\begin{align}
    \delta \vb{u} = \vb{v} = \left \{ \begin{matrix}
        \vb{e}_x & \text{on} & \Gamma_D^1 \\
        \bm{\psi} & \text{otherwise} & 
    \end{matrix} \right . \, ,
\end{align}
where the vector $\bm{\psi} \in [\Hone(\Omega)]^3$ is arbitrary. 
The discrete form now reads
\begin{align}
    \int_{\Gamma_D} \langle \vb{v}^h, \, \vb{t} \rangle \, \dd A = a(\{\vb{u}^h,\bm{P}\},\{\vb{v}^h,\delta\bm{P}\}) - l(\{\vb{v}^h,\delta\bm{P}\})\, ,
\end{align}
resulting in the corresponding scalar product
\begin{align}
    \int_{\Gamma_D} \langle \vb{v}^h, \, \vb{t} \rangle \, \dd A = \langle \bm{K} \vb{u}^d , \, \vb{v}^d \rangle - \langle \vb{f}^d , \, \vb{v}^d \rangle \, ,  
\end{align}
where $\vb{u}^d$ and $\vb{v}^d$ are the discrete vectors of the node values of the trial and test functions, respectively. The vector of the discrete values of the body forces and micro-moment is given by $\vb{f}^d$. The vector $\vb{v}^d$ is now defined, such that its nodal values reflect a unity function in the x-direction on the upper Dirichlet boundary. 
\begin{remark}
    Note that using the latter, no new computation of the stiffness matrix $\bm{K}$, the nodal values vector $\vb{u}^d$ or the force vector $\vb{f}^d$ is required.
\end{remark}

\end{document}